\documentclass[12pt]{article}
\pdfoutput=1

\usepackage{graphicx} 

\usepackage[utf8]{inputenc}
\usepackage[numbers]{natbib}
\usepackage{amsmath} 
\usepackage{bbm}
\usepackage{amsfonts,amssymb,amsthm}
\usepackage{mathrsfs}
\usepackage{enumerate}
\usepackage{indentfirst}
\usepackage{geometry}
\usepackage{url}
\usepackage[normalem]{ulem}
\usepackage{comment}
\usepackage{transparent}

\allowdisplaybreaks[4]
\numberwithin{equation}{section}

\newcommand{\dd}{\ensuremath{\mathrm{d}}}
\newcommand{\ind}{\mathbbm{1}}
\newcommand{\dequiv}{\overset{d}{=}}

\newtheorem{definition}{Definition}[section]

\newtheorem{theorem}[definition]{\textbf{Theorem}}
\newtheorem{lemma}[definition]{\textbf{Lemma}}
\newtheorem{proposition}[definition]{\textbf{Proposition}}
\newtheorem{corollary}[definition]{\textbf{Corollary}}
\newtheorem{remark}[definition]{\textbf{Remark}}
\usepackage[colorlinks=true, allcolors=blue]{hyperref}
\usepackage{xcolor}
\definecolor{brown}{rgb}{0.6,0.4,0.2}

\def\r{{\mathbb R}}
\def\e{{\mathbb E}}
\def\p{{\mathbb P}}
\def\q{{\mathbb Q}}

\def\d{\, \mathrm{d}}

\def\eps{{\varepsilon}}

\def\B{\mathcal B}

\def\cD{\mathcal D}

\def\cL{\mathcal L}

\def\cW{\mathcal W}

\def\cS{\mathcal S}
\def\tcS{\widetilde{\mathcal S}}

\def\cY{\mathcal Y}

\def\cF{\mathcal F}

\def\rW{\mathscr W}

\def\deltahat{\widehat{\delta}}
\def\cShat{\widehat {\mathcal S}}

\title{\Large{\textbf{ Meeting of squared Bessel flow lines and application to the skew Brownian motion}  }}
\author{Elie A\"id\'ekon\footnote{\scriptsize SMS, Fudan University, China, \texttt{aidekon@fudan.edu.cn}} \and Chengshi Wang\footnote{\scriptsize SMS, Fudan University, China, \texttt{cswang21@m.fudan.edu.cn}} \and Yaolin Yu\footnote{\scriptsize SMS, Fudan University, China, \texttt{ylyu23@m.fudan.edu.cn}}}
\date{}

\begin{document}

\maketitle

\begin{abstract}
We study the meeting level between squared Bessel (BESQ) flow lines of different dimensions, and show that it gives rise to a jump Markov process. We apply these results to the skew Brownian flow introduced by Burdzy and Chen \cite{burdzy2001local} and Burdzy and Kaspi \cite{burdzy2004lenses}. It allows us to extend the results of \cite{burdzy2001local} and of Gloter and Martinez \cite{gloter2013distance} describing the local time flow of skew Brownian motions. Finally, we compute the Hausdorff dimension of exceptional times revealed by Burdzy and Kaspi \cite{burdzy2004lenses} when skew Brownian flow lines bifurcate.
\end{abstract}

{\bf Classification}. 60J65, 60J55, 60H10.

{\bf Keywords}. Skew Brownian motion, local time, stochastic flow, squared Bessel process.

\section{Introduction}

Given a white noise $\cW$ on $\r_+\times \r$ and some positive number $\delta$, we consider the strong solution $(\cS_{r,x}(a),\, x\ge r)$ of the SDE
\begin{equation}\label{eq:intro BESQ}
    \cS_{r,x}(a)=a+ \int_r^x \cW([0,\cS_{r,s}(a)],\dd s) +\delta(x-r),\,  x\ge r.
\end{equation}

\noindent For each $(a,r)\in \r_+\times \r$, the distribution of $(\cS_{r,x}(a),\, x\ge r)$ is  a squared Bessel process  (${\rm BESQ}^\delta_a$) of dimension $\delta$ starting from $a$, see Section \ref{s:BESQ process}. Following Dawson and Li \cite{dawson-li12}, we can also view it as some flow line  emanating from the point $(a,r)$. A certain version of the collection of processes $(\cS_{r,x}(a),\, x\ge r)_{(a,r)\in \r_+\times \r}$ is called ${\rm BESQ}^\delta$ flow in \cite[Section 3]{aïdékon2023stochastic}. The idea of a flow whose marginals are BESQ processes trace back to Pitman and Yor \cite{pitmanyor}. These flows naturally appear in the context of continuous-state branching processes, see Bertoin and Le Gall \cite{bertoin-legall00}, Lambert  \cite{lambert} and of generalized Ray--Knight theorems, see Carmona, Petit and Yor \cite{carmona_besq} and A\"id\'ekon, Hu and Shi \cite{aidekon2024infinite}. In \cite{aïdékon2023stochastic}, they were used to deduce disintegration theorems of a perturbed reflecting Brownian motion stopped at a random time given its occupation field. In the present paper, we will prove that BESQ flows are connected to the skew Brownian flow  introduced in Burdzy and Chen \cite{burdzy2001local} and Burdzy and Kaspi \cite{burdzy2004lenses}. \\

The first part of the paper is devoted to the study of the interaction between flow lines of different drifts  driven by the same white noise $\cW$. We distinguish three different situations. In all cases,  $\cS$ is the ${\rm BESQ}^\delta$ flow defined in \eqref{eq:intro BESQ}. The drifts $\delta,\deltahat,\delta'$ are positive numbers.

Let $Y=(Y_x,\,x\ge 0)$ denote the ${\rm BESQ}^{\deltahat}$ flow line  driven by $\cW$ starting from $(0,0)$. Suppose that $\deltahat < \delta+2$. Then any flow line $\cS_{r,\cdot}(0)$ emanating from $(0,r)$ with $r\ge 0$ meets $Y$. This is due to the fact that $\cS_{r,x}(0)-Y_x$ for $x\ge r$ is a ${\rm BESQ}^{\deltahat-\delta}$ process, hence hits $0$ a.s. if $\deltahat-\delta< 2$. Call $U(r)$ the meeting level of $Y$ and $\cS_{r,\cdot}(0)$, i.e.  
\begin{equation*}\label{U_r}
  U(r):=\inf\{x\ge r\,:\, \cS_{r,x}(0)=Y_x\}.  
\end{equation*}

\noindent See Figure \ref{fig:intro1} (a). For $a,b>0$, we let ${\mathcal B}(a,b)$ denote the beta($a,b$) distribution, i.e. the distribution on $[0,1]$ with density $\frac{1}{B(a,b)}x^{a-1}(1-x)^{b-1}$. For simplicity, we consider in the introduction the right-continuous versions of the processes. The following proposition shows that $U$ is a Markov process. We then define the law of $U$ given $U(0)=z$ as the law of this Markov process with starting position $z$.

\begin{theorem}\label{thm: main left}
    The process $(U(r)-r,\, r\ge 0)$ is a homogeneous Feller process starting from $0$. For any $r>0$, $\frac{U(r)}{r}$ has distribution ${\mathcal B}(\frac{2-\deltahat+\delta}{2},\frac{\deltahat}{2})$. Conditionally on $U(0)=z>0$, the process $U$ stays constant then jumps at time $\mathrm{x}$ where  $\frac{{\mathrm x}}{z}$ has distribution ${\mathcal B}(1,\frac{\delta}{2})$. Conditionally on ${\mathrm x}=x$, $\frac{z-x}{U(x)-x}\sim {\mathcal B}(\frac{2-\deltahat+\delta}{2},1)$.
\end{theorem}

The second situation involves the same flow line $Y$, but we look now at the meeting level between $Y$ and $\cS_{r,\cdot}(0)$ when $r\le 0$. Therefore we require $\delta<\deltahat +2$ while we set again, but now for $r\le 0$,
\[
U(r):=\inf\{x\ge 0\,:\, \cS_{r,x}(0)=Y_x\}.
\]
\noindent This case is also pictured in Figure \ref{fig:intro1} (a). For future reference, we artificially introduce a parameter $\delta'$, which is simply equal to $\delta$ in the following theorem.
\begin{theorem}\label{thm: main right} 
    The process $(U(-r)+r,\, r\ge 0)$ is a homogeneous Feller process starting from $0$. For any $r>0$, $\frac{r}{U(-r)+r}$ has distribution ${\mathcal B}(\frac{2-\delta+\deltahat}{2},\frac{\delta'}{2})$. Conditionally on $U(0)=z>0$, the process $U(-r)$ stays constant then jumps at time $\mathrm{x}$ where  $\frac{z}{z+{\mathrm x}}$ has distribution ${\mathcal B}(\frac{\delta'}{2},1)$. Conditionally on ${\mathrm x}=x$, $\frac{z+x}{U(-x)+x}\sim {\mathcal B}(\frac{2-\delta+\deltahat}{2},1)$.
\end{theorem}

The third case is slightly different. We suppose that the flow line $Y^*$ is solution of \eqref{eq:intro BESQ} starting from $(0,0)$, but with $(-\cW^*,\deltahat)$ instead of $(\cW,\delta)$, where $\cW^*$ is the image of $\cW$ by the map $(a,r)\mapsto (a,-r)$. In the terminology of \cite{aïdékon2023stochastic}, $Y^*$ is a dual flow line.  The meeting level between $Y^*$ and flow lines of $\cS$ is defined as, for $r\ge 0$,
\[
V(-r):=\inf\{ x\in [-r,0]\,:\, \cS_{-r,x}(0)=Y^*_{-x}\}.
\]

\noindent See Figure \ref{fig:intro1} (b).  We suppose that $\deltahat+\delta>2$ (otherwise the meeting only takes place when $Y^*=0$).
\begin{theorem}\label{thm: main dual}
    The process $(V(-r)+r,\, r\ge 0)$ is a homogeneous Feller process starting from $0$. For any $r>0$, $\frac{V(-r)+r}{r}$ has distribution ${\mathcal B}(\frac{\deltahat}{2},\frac{\delta}{2})$. Conditionally on $V(0)=z>0$,  the  process $V$ stays constant then jumps at time $\mathrm{x}$ where  $\frac{z}{z+{\mathrm x}}$ has distribution ${\mathcal B}(\frac{\delta}{2},1)$. Conditionally on ${\mathrm x}=x$, $\frac{V(-x)+x}{x+z}\sim {\mathcal B}(\frac{\deltahat+\delta}{2},1)$.
\end{theorem}

\begin{figure}[htbp]
\centering
       \scalebox{0.8}{ 
        \def\svgwidth{0.8\columnwidth}
        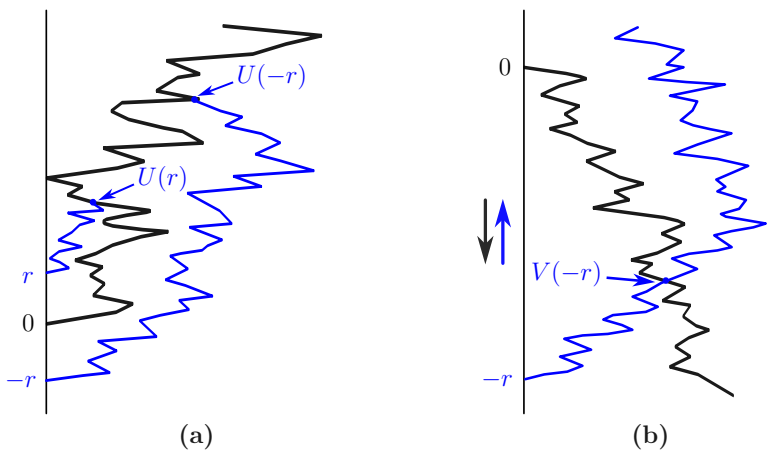
    }
    \caption{The black line represents $Y$/$Y^*$ in picture \textbf{(a)}/\textbf{(b)}, both starting from $(0,0)$. The blue lines in both pictures represent the ${\rm BESQ}^\delta$ flow $\cS$. In picture \textbf{(b)}, the black line goes down and the blue line goes up.}
    \label{fig:intro1}
\end{figure}

\bigskip

These theorems characterize the distribution of the processes via their entrance law and jump distributions. One could also express their transition probabilities in terms of the semigroup of BESQ processes with varying dimensions, or identify their Lamperti transform (since these processes inherit the $1$-self similarity of the BESQ flow).  The process $(U(r)-r,\, r\ge 0)$ of Theorem \ref{thm: main left}, resp. $(V(-r)+r,\, r\ge 0)$ of Theorem \ref{thm: main dual}, returns to $0$ when the flow line $Y$, resp. $Y^*$, returns to $0$, i.e. when its drift $\deltahat$ belongs to $(0,2)$. The proofs  rely on a decomposition theorem for BESQ flows along some flow line in Section \ref{s:decomposition},  which is a generalization of the additivity property of BESQ processes \cite{pitman2018squared,pitmanyor}. Along the way, the concept of $\oplus/\ominus$-flow line  will prove useful in order to obtain some analog of the Markov property for BESQ flows. 

\bigskip

The second part of the paper is concerned with  a seemingly unrelated topic, the skew Brownian flow introduced in  \cite{burdzy2001local, burdzy2004lenses}.  Let $B=(B_t,\,t\ge 0)$ be a standard Brownian motion, $\beta\in (-1,1)$ and $r\in \mathbb{R}$. We consider  the strong solution $X^r$ of the SDE
\begin{equation}\label{eq:SDE Xtr}
    X_t^r=L_t^r-B_t,\, t\ge 0
\end{equation}
where 
\begin{equation}\label{def:L}
L_t^r:=r+\beta \ell_t^r
\end{equation}

\noindent and $\ell_t^{r}$ is the symmetric local time of $X^{r}$ at position $0$:
\begin{align}\label{def:Ltilde}
    \ell_t^{r} = 
    \lim\limits_{\varepsilon \rightarrow 0} \frac{1}{2\varepsilon} \int_0^t \ind_{\{|X_u^{r}| < \varepsilon\}} \, \mathrm{d}u, \, t \ge 0.
\end{align}

\noindent The process $X^r$ behaves as the Brownian motion $-B$ when it is nonzero, but has an asymmetry at $0$ when $\beta\neq 0$. In distribution, $X^0$ is a concatenation of Brownian excursions whose signs are chosen positive with probability $p=\frac{1+\beta}{2}$ and negative otherwise. In particular $|X^0|$ is always distributed as a reflecting Brownian motion. The extreme cases $\beta\in \{-1,1\}$ correspond to reflecting Brownian motions while equation \eqref{eq:SDE Xtr} does not have a solution when $|\beta|>1$, see Harrison and Shepp \cite{harrison1981skew}. Skew Brownian motions were introduced by It\^o and McKean \cite{itomckean} and further studied by Walsh \cite{walsh1978diffusion}. We refer to Lejay \cite{lejay2006constructions} for a review on this topic. The collection  of solutions $(X^r)_{r\in \r}$ driven by the same Brownian motion $B$ is a coalescing flow \cite{barlow1999coalescence, burdzy2001local}: for any $r,r'$, the solutions $X^r$ and $X^{r'}$ meet in a finite time and stay equal afterwards. We refer to \cite{arratia,lejan_raimond,toth_werner} and the references therein for other examples of coalescence in stochastic flows.

The following theorem gives the connection with the first part of the paper. We restrict to $\beta\in (0,1)$ for convenience. Let  $(\cL(t,x))_{t\ge 0,x\in \r}$ be a bicontinuous version of  the local time of $B$ and $\cW$ be the white noise on $\r_+\times \r$ defined through
\begin{align}\label{def:W}
    \cW(g)=\int_0^{\infty} g(\cL(t,B_t), B_t) \dd B_t
\end{align}

\noindent for any $g \in L^2(\r_+\times \r)$. Equation \eqref{def:W} appears in \cite{aidekon2024infinite} where it is observed that $\cW$ is indeed a white noise as a result of the occupation times formula.  For any $r\in \mathbb{R}$, we set
\begin{equation}\label{def:taurx}
\tau^r_x:=\inf\{ t\ge 0 \,:\,  L_t^r > x\},\, x\ge r.
\end{equation}

\noindent In the following theorem, a ${\rm BESQ}(\delta_1\,|_0\, \delta_2)$ flow line $(\cS_{r,x}(a),\,x\ge r)$ behaves as a ${\rm BESQ}$ flow line with drift $\delta_1$ when $x\le 0$ (hence necessarily $r\le x\le 0$) and with drift $\delta_2$ when $x\ge 0$.

\begin{theorem}\label{thm:representation_by_WN}
Let $\beta\in(0,1)$ and $\delta:=\frac{1-\beta}{\beta}$. Fix $r\in \r$. The process $(\cL(\tau^r_x,x),\,x\ge r)$ is the  flow line starting at the point $(0,r)$ of the ${\rm BESQ}(2+\delta \, |_0 \, \delta)$ flow driven by $\cW$.
\end{theorem}

\noindent Thanks to Theorem \ref{thm:representation_by_WN}, we can rephrase questions on skew Brownian motions in the framework of BESQ flows.  Fix  $\beta\in (-1,0)\cup (0,1)$ and $\widehat{\beta}\in (0,1)$. Let as before $X^r$ be the skew Brownian motion associated to the parameter $\beta$ while $(\widehat{X},\widehat{L})$ is defined by \eqref{eq:SDE Xtr} and \eqref{def:L} with $(\widehat{\beta},0)$ in place of $(\beta,r)$. 
Following \cite{burdzy2001local, gloter2013distance}, consider for $r\in \mathbb{R}$ the r.v. 
\begin{equation}\label{def:Tbetax}
    T(r) := \inf\{t\ge 0\,:\, X_t^{r}=\widehat X_t\} \in [0,\infty]. 
\end{equation}

\noindent The time $T(r)$ is finite if and only if the skew Brownian motions $\widehat X$ and $X^r$ hit each other. 

\begin{theorem} \label{thm: skewBM meeting}
     Set $\delta= \frac{1-|\beta|}{|\beta|}$ and $\deltahat=\frac{1-\beta_0}{\beta_0}$.
    \begin{enumerate}[(i)]
\item   Suppose that  $\beta\in (0,1)$ and $\deltahat<\delta+2$.  The process $(\widehat L_{T(r)} ,\, r\ge 0)$ is distributed as the process $(U(r),\, r\ge 0)$ of Theorem \ref{thm: main left}. 
\item Suppose that $\beta\in (0,1)$ and $\delta<\deltahat+2$. The process $(\widehat L_{T(-r)}, r\ge 0)$ is distributed as the process $(U(-r),\, r\ge 0)$ of Theorem \ref{thm: main right} with $\delta'=2+\delta$ there. 
\item Suppose that $\beta\in (-1,0)$. The  process $(\widehat L_{T(r)},\,r\ge 0)$ is distributed as the process $(-V(-r),\, r\ge 0)$ of Theorem \ref{thm: main dual} with   $2+\delta$ in place of $\delta$ there. 
 \end{enumerate}
\end{theorem}

The marginals of these processes were computed in  \cite[Theorem 3, Corollary 2, Theorem 4]{gloter2013distance}. In each case, the assumptions on $(\delta,\deltahat)$ are necessary and sufficient conditions to have $T(r)<\infty$.  As observed in \cite{burdzy2001local}, the next result is an analog of the second Ray--Knight theorem for the skew Brownian flow. Fix $\beta\in (0,1)$ and let $(X^r, L^r)$ be given by \eqref{eq:SDE Xtr} and \eqref{def:L}. For $z>0$, let $\tau_z^0$ be as in \eqref{def:taurx}.   Statement (i) of the following theorem was already proved in \cite[Theorem 1.2]{burdzy2001local}. 
\begin{theorem}\label{thm: skewBM RK}
Let $z>0$ and $\beta\in (0,1)$. Set $\delta=\frac{1-\beta}{\beta}$.
\begin{enumerate}[(i)]
\item \cite[Theorem 1.2]{burdzy2001local} The process $(L^r_{\tau_z^0},\,r\ge 0)$ is distributed as $(U(r),\,r\ge 0)$ of Theorem \ref{thm: main left} with $\deltahat=0$, conditioned on $U(0)=z$.
\item The process $(L^{-r}_{\tau_z^0},\,r\ge 0)$ is distributed as $(V(-r),\,r\ge 0)$ of Theorem \ref{thm: main dual} with $(2+\delta,0)$ in place of $(\delta,\deltahat)$, conditioned on $V(0)=z$.
\end{enumerate}
\end{theorem}

The method of proof is different from \cite{burdzy2001local,gloter2013distance}. We deduce our results as consequences of the theorems in the first part of the paper, once the connection with the BESQ flow is established. Note that we could as well fix  starting points $(0,0)$ and $(0,r)$ in the BESQ flow and make the dimension $\delta$ vary. This setting would be related to the meeting of skew Brownian motions with varying parameters $\beta$, which is the topic of \cite{Gloter2018BouncingSB}. It would lead for example  to a solution of Open Problem 1.8 of \cite{burdzy2001local}. This problem was given a near to complete answer in \cite{Gloter2018BouncingSB}, the entrance law being left open in that work. We omit to present such results for sake of brevity.

\bigskip

We finish the presentation with a brief and informal discussion on bifurcation events, and refer to Sections \ref{s:bifur} and \ref{s:bifur_time} for the rigorous definitions. Even if the solution of the SDE \eqref{eq:intro BESQ} is unique a.s., there exist exceptional points $(a,r)$ where ${\rm BESQ}^\delta$ flow lines bifurcate at their starting point. These points are the so-called ancestors in continuous-state branching processes  \cite[Section 2.2]{bertoin-legall00}. We show in Section \ref{s:bifur} that such points have Hausdorff dimension $\frac32$ (i.e. it is the dimension of the graph of the standard Brownian motion \cite{taylor}). More interestingly, when $\delta_1 < \delta_2< \delta_1+2$, there are points which are  common bifurcation points for the flows with drifts $\delta_1$ and $\delta_2$ where the flow lines of drift $\delta_1$ and $\delta_2$ are interlaced, see Figure \ref{fig:bifur1} in Section \ref{s:bifur}. These points have Hausdorff dimension $\min(2-d,\frac{3-d}{2})$ where $d:=\delta_2-\delta_1$. Such bifurcation events have a natural interpretation for the skew Brownian flow, in the case we not only allow the starting point to vary, but also the starting time. These bifurcation points  correspond to the so-called ordinary/semi-flat bifurcation times studied by Burdzy and Kaspi \cite{burdzy2004lenses} (we omit the proof of this statement for concision). By a time-reversal argument, we show in Section \ref{s:bifur_time} that ordinary/semi-flat bifurcation times have Hausdorff dimension respectively $\frac12$ and $\frac{2-\delta}{4}$ where $\delta=\frac{1-|\beta|}{|\beta|}$ (there is no semi-flat bifurcation time when $|\beta|\le \frac13$).  This last result does not use the link with the BESQ flow.

\bigskip

\paragraph{Related works} In \cite{pitman2018squared}, Pitman and Winkel interpret \eqref{eq:SDE Xtr} as a decomposition of the Brownian motion $B$ into two parts, depending on whether $B$ is above or below  $L^x$. They show that these two processes are time-changes of perturbed reflecting Brownian motions (PRBM). In view of the link between BESQ flows and PRBM \cite{aidekon2024infinite}, it is natural to reinterpret this picture as a decomposition of $\r_+\times \r$ along a ${\rm BESQ}^\delta$ flow line driven by the white noise $\cW$ given by \eqref{def:W}. It  was the starting point of this project.

The connection between BESQ flows and skew Brownian motions is reminiscent of a connection between flow lines of the planar Gaussian free field and a(nother) flow of skew Brownian motions \cite{gwynne_holden_sun, borga}, which appears in the setting of the $\gamma$-Liouville quantum gravity. We can view our paper as an analog  when $\gamma\to 0$.

\paragraph{Structure of the paper} Section \ref{s:BESQ_Flow} introduces BESQ flows. Section \ref{s:decomposition} contains the decomposition theorem which is at the heart of the proofs of Theorems \ref{thm: main left}, \ref{thm: main right}, \ref{thm: main dual} in Section \ref{s:meeting}. In Section \ref{s:skewBM},  we prove the connection between BESQ flows and skew Brownian motions of Theorem \ref{thm:representation_by_WN} (Section \ref{s:repre_skew_BM}), and prove Theorems \ref{thm: skewBM meeting} and \ref{thm: skewBM RK}. The computation of the Hausdorff dimension of bifurcation times is the topic of Section \ref{s:bifur_time}.

\bigskip

\paragraph{Acknowledgements} We are grateful to Quan Shi for introducing us to the paper \cite{pitman2018squared}, and for  helpful discussions. We thank Xin Sun for pointing out the analogy with \cite{gwynne_holden_sun, borga}. We  thank  Miguel Martinez for useful explanations on the paper \cite{Gloter2018BouncingSB} and Wenjie Sun for drawing our attention to the issue of bifurcation events. E.A. was supported by NSFC grant QXH1411004.


\section{BESQ flows}\label{s:BESQ_Flow}


\subsection{BESQ processes}
\label{s:BESQ process}
Let $\delta \ge 0$.  The squared Bessel process of dimension $\delta$ started at $a\ge 0$, denoted by ${\rm BESQ}^\delta_a$,  is the  unique solution of
\begin{equation}\label{eq:besselprocess}
S_x = a+2\int_0^x \sqrt{|S_r|} \, \d B_r + \delta x, \qquad x\ge 0,
\end{equation}

\noindent where
$B$ is a standard Brownian motion. The ${\rm BESQ}^\delta$ process hits zero at a positive time if and only if $\delta< 2$. It is absorbed at $0$ when $\delta= 0$ and is  reflecting
at $0$ when $\delta\in (0,2)$. When $\delta<0$, we will take for definition of the ${\rm BESQ}^\delta_a$ process the unique solution of 
 \begin{equation}\label{eq:besselprocess negative}
S_x = a+2\int_0^x \sqrt{|S_r|} \, \d B_r + \delta \min(x,T_0), \qquad x\ge 0,
\end{equation}

\noindent where $T_0:=\inf\{x\ge 0\,:\, S_x=0\}$, so that the process is absorbed when hitting $0$, which happens in a finite time a.s. We refer to \cite[Chapter XI]{revuz2013continuous} for background on BESQ processes. When discussing BESQ processes with varying dimensions, the following definition will come in handy. Let $\cF=(\cF_x,\,x\in \r)$ be a  right-continuous filtration. 

\begin{definition}\label{def:drift}
    An $\cF$-predictable process $\overline{\delta}:\r\to\r$ is called a drift function if there exists a deterministic vector $(\delta_i)_{1\le i\le n} \in \r^n$ and $\cF$-stopping times $-\infty=\tilde{t}_1<\tilde{t}_2<\ldots<\tilde{t}_{n+1}=\infty$ such that $\overline{\delta}=\delta_i$ on $A_i:=(\tilde{t}_{i},\tilde{t}_{i+1})$.

    If $\overline{\delta'}$ is another drift function and $r_1$ is an $\cF$-stopping time, we write $\overline{\delta}\,|_{r_1}\, \overline{\delta'}$ for the drift function $\overline{\delta}(x)\ind_{(-\infty,r_1)}(x)+ \overline{\delta'}(x)\ind_{[r_1,\infty)}(x)$.
\end{definition}

We then write ${\rm BESQ}_a(\overline{\delta})$ for the distribution of the continuous process starting from $a$ which is distributed as a ${\rm BESQ}^{\delta_i}$ process on each $A_i$.


\subsection{Martingale measures}
\label{s:mart}

In the setting of \cite[Chapter 2]{Walsh1986AnIT}, let $\cW$ be a white noise on $\r_+\times \r$ with respect to 
some right-continuous filtration  $\cF=(\cF_x,\, x\in \r)$. If $f$ is a bounded predictable process, one can define \cite[Theorem 2.5]{Walsh1986AnIT} the orthogonal martingale measure $f\cdot \cW$ as,  for any Borel set $A\subset \r_+$ with finite Lebesgue measure, and any $x\le y$,
$$
(f\cdot \cW)_{x,y}(A) := \int_{A\times [x,y]} f(\ell,r)\cW(\dd \ell,\dd r).
$$

\noindent  In that case, the stochastic integral with respect to $f\cdot \cW$ can be expressed as 
\begin{equation}\label{eq:fW}
\int_{\r_+\times [x,y]} h(\ell,r) (f\cdot \cW)(\dd \ell, \dd r)
=
\int_{\r_+\times [x,y]} h(\ell,r)f(\ell,r) \cW(\dd \ell,\dd r)
\end{equation}

\noindent for any predictable process $h$ such that $\e[\int_{\r_+\times [x,y]} h^2(\ell,r)f^2(\ell,r) \dd \ell \dd r]<\infty$. 

\bigskip

For any nonnegative predictable process $(S_r,\, r\in \r)$, we also let $\theta_{S} \cW$ be the martingale measure on $\r_+\times \r$
$$
(\theta_S \cW)_{x,y}(A) := \int_{\r_+\times [x,y]} \ind_{A}(\ell-S_r) \cW(\dd \ell, \dd r).
$$

\noindent  By computation of quadratic variations, we observe that $\theta_S \cW$ is still a white noise on $\r_+\times \r$ with respect to $\cF$. For any predictable $h$ such that $\e[\int_{\r_+\times [x,y]} h^2(\ell,r) \dd \ell \dd r]<\infty$, we have the representation 
    \begin{equation}\label{eq:thetaSW}
         \int_{\r_+\times[x,y]} h(\ell,r) \theta_S \cW (\dd \ell,\dd r) = \int_{\r_+\times[x,y]} \ind_{[S_r,\infty)}(\ell)h(\ell-S_r,r) \cW(\dd \ell,\dd r).
    \end{equation}

\noindent This can be proved by standard arguments, using the density of simple functions.


\subsection{Basic facts about BESQ flows}
\label{s:def}
Stochastic flows related to branching processes form a well-established topic \cite{bertoin-legall00,dawson-li12, lambert, pitmanyor}. A particular case is the BESQ flow introduced in \cite{pitmanyor} in relation to the Brownian motion. We will use the following definition  which slightly differs from \cite[Definition 3.3 \& Definition 3.4]{aïdékon2023stochastic}, but defines the same object\footnote{The regularity condition (iii) in \cite{aïdékon2023stochastic}  was weaker than the perfect flow property. Our definition is then a consequence of  \cite[Proposition 3.9]{aïdékon2023stochastic}. 
In the case $\delta\in(0,2)$ and $\cS$ is killed, we use \cite[Proposition A.1]{aïdékon2023stochastic}. }.

\begin{definition}
\label{def:BESQflow}
Let $\delta \in \r$.  We call ${\rm BESQ}^\delta$ flow  driven by $\cW$ a collection ${\mathcal S}$ of continuous processes $({\mathcal S}_{r,x}(a),\, x\ge r)_{r\in \r,a\ge 0}$ such that: 
\begin{enumerate}[(1)]
\item
for each $(a,r)\in \r_+\times \r$, the process $({\mathcal S}_{r,x}(a),\, x\ge r)$ is almost surely the strong solution of the following SDE  
\begin{equation}\label{eq:BESQflow}
{\mathcal S}_{r,x}(a) = 
a + 2 \int_r^x  {\mathcal W}([0,{\mathcal S}_{r,s}(a)], \d s) + 
\delta (x-r) 
,\, x\ge r
\end{equation}
 and which, in the case $\delta<0$, is absorbed when hitting $0$.
 \item Almost surely,
 \begin{enumerate}[(i)]
 \item for all $r\in \r$ and $a\ge 0$, $ {\mathcal S}_{r,r}(a)=a$,
 \item for all $r\le x$, $a\mapsto  {\mathcal S}_{r,x}(a)$ is c\`adl\`ag,
 \item (Perfect flow property) for any $r\le x \le y$ and $a\ge 0$, $\cS_{r,y}(a)=\cS_{x,y}\circ \cS_{r,x}(a)$. 
\end{enumerate} 
\end{enumerate}

We call killed ${\rm BESQ}^\delta$ flow driven by $\cW$ the flow obtained from the ${\rm BESQ}^\delta$ flow by absorbing at $0$ the flow line $\cS_{r,\cdot}(a)$ at time $\inf\{x>r\,:\,\cS_{r,x}(a)=0\}$.

We call general ${\rm BESQ}^\delta$ flow  a  ${\rm BESQ}^\delta$ flow or a  killed ${\rm BESQ}^\delta$ flow. 
\end{definition}

When $\delta\notin (0,2)$ there is no difference between the $\rm BESQ^\delta$ and the killed $\rm BESQ^\delta$ flow.  We will sometimes call the $\rm BESQ^\delta$ flow a non-killed $\rm BESQ^\delta$ flow to distinguish it with its killed version. When $\cS$ is the killed $\rm BESQ^\delta$ flow with $\delta\in(0,2)$, there are times when $\cS_{r,x}(0)>0$ for some $x>r$ until it comes back to $0$ where it gets absorbed. It happens when the non-killed $\rm BESQ^\delta$ flow line emanating from $(0,r)$ starts by an excursion away from $0$. It implies that $\cS$ does not satisfy the perfect flow property, see Proposition \ref{p:perfect} (i) for the analog.  When $\delta\le 0$, flow lines are absorbed at $0$ and no flow lines can exit $a=0$. When $\delta\ge 2$, flow lines do not hit $0$, except at their starting point if they start from $a=0$.

Recall the notation of Section \ref{s:BESQ process}. The flow line $\cS_{r,\cdot}(a)$ of a $\rm BESQ^\delta$ flow is a ${\rm BESQ}_a^\delta$ process. Equation \eqref{eq:BESQflow} implies the following statements, that will be referred throughout the paper as property ({\bf P}). Let $a'\ge a\ge 0$ and $\cS$, $\cS'$ be resp. a $\rm BESQ^\delta$ flow and a $\rm BESQ^{\delta'}$ flow driven by $\cW$.
\begin{enumerate}[({\bf P}1)]
\item  If $\delta'\ge \delta\ge 0$,  the process $\cS'_{r,\cdot}(a')- \cS_{r,\cdot}(a)$ is  a ${\rm BESQ}_{a'-a}^{\delta'-\delta}$ process independent of $\cS_{r,\cdot}(a)$. 
\item  If $\delta\ge 0$ and $\delta'<\delta$  the process $\max(\cS'_{r,\cdot}(a')- \cS_{r,\cdot}(a),0)$ is  a ${\rm BESQ}_{a'-a}^{\delta'-\delta}$ process independent of $\cS_{r,\cdot}(a)$. 
\item  When $\delta <0$, the flow line $\cS_{r,\cdot}(a)$  is  absorbed at $0$ at some time $\varphi(a,r)$. Conditionally on $\varphi(a,r)=t$, if $\delta'\ge \delta$,  $\cS'_{r,\cdot}(a')- \cS_{r,\cdot}(a)$ is a ${\rm BESQ}_{a'-a}(\delta'-\delta\,|_t\, \delta')$ independent of $\cS_{r,\cdot}(a)$.
\item If $\delta'<\delta<0$, conditionally on $\varphi(a,r)=t$, $\max(\cS'_{r,\cdot}(a')- \cS_{r,\cdot}(a),0)$ is a ${\rm BESQ}_{a'-a}(\delta'-\delta\,|_t\, \delta')$ independent of $\cS_{r,\cdot}(a)$.
\end{enumerate}

 It is a form of the well-known  additivity property for BESQ processes, see \cite{pitman2018squared, pitmanyor}.  \\

  The general ${\rm BESQ}^\delta$ flow is determined by $\cW$ in the sense that if $\cS$ and $\cS'$ are both driven by $\cW$, then a.s., $\cS_{r,x}(a)=\cS'_{r,x}(a)$ for all $r\le x$ and $a\ge 0$, see \eqref{eq:approximation+} below.  In our setting, the following result can be deduced from the perfect flow property and the construction of the killed ${\rm BESQ}^\delta$ flow. 

\begin{proposition}\label{p:perfect}(\cite[Proposition 3.9 \& A.5]{aïdékon2023stochastic})
Let $\cS$ be a general ${\rm BESQ}^\delta$ flow. Then almost surely:
\begin{enumerate}[(i)]

 \item ((Almost) perfect flow property) For every $r\le x\le y$ and $a\ge 0$ with $\cS_{r,x}(a)>0$, $\cS_{r,y}(a)=\cS_{x,y}\circ \cS_{r,x}(a)$. 

 \item (Coalescence) If $r,r'< x$, $0\le a,a'$ and $\cS_{r,x}(a)=\cS_{r',x}(a')$, then $\cS_{r,y}(a)=\cS_{r',y}(a')$ for all $y\ge x$. If $\cS$ is a ${\rm BESQ}^\delta$ flow, it also holds  when $\max(r,r')=x$.

\end{enumerate}   
\end{proposition}

When the flow $\cS$ is killed (and $\delta\in (0,2)$), the reason we need to avoid the case $\max(r,r')=x$ is the existence of these exceptional times $r$ such that $\cS_{r,x}(0)>0$ for $x>r$ close enough to $r$. For $a>0$ and $r\le x$, let $\cS_{r,x}(a-):=    \lim_{a'\uparrow a} \cS_{r,x}(a')$.

\begin{proposition}\label{p:properties}
Let $\cS$ be a general $\rm BESQ^\delta$ flow. The following statements hold almost surely.
\begin{enumerate}[(i)]
    \item For any $r<x$ and $a\ge 0$, $\cS_{r,x}(a')=\cS_{r,x}(a)$ for all $a'>a$ close enough to $a$.
    \item For any $r<x$ and $a>0$, $\cS_{r,x}(a')=\cS_{r,x}(a-)$ for any $0\le a'<a$ close enough to $a$.
\end{enumerate}
\end{proposition}

\begin{proof}
    Almost surely, for any $r<x$ and any bounded interval $I$, the set $\{\cS_{r,x}(a): a\in I\}$ is finite. One can see this property for example from the embedding of the BESQ flow in the PRBM \cite[Proposition 3.6]{aïdékon2023stochastic}. Hence (i) and (ii) follow  from Definition \ref{def:BESQflow} (ii). 
\end{proof}

If $\mathcal{D}=\{(a_n,r_n)\}_{n\ge 1}$ is a countable dense (possibly random) subset of $\r_+\times\r$, then for every $a\ge 0$ and $x\ge r$, \cite[Proposition 3.7]{aïdékon2023stochastic}
\begin{align}\label{eq:approximation+}
\cS_{r,x}(a)=\inf_{(a_n,r_n)\in\mathcal{D}:\, r_n\le r,\, \cS_{r_n,r}(a_n)>a} \cS_{r_n,x}(a_n).
\end{align} 

\noindent The set on which the infimum is taken is not empty. Indeed, following the reasoning of step (i) in the proof of \cite[Proposition 2.6]{aïdékon2023stochastic}, a.s., for any $0\le a<a'$ and $r\in \r$, one can find  $(a_n,r_n)\in \mathcal{D}$ arbitrarily close to $(a,r)$ such that
\begin{equation}\label{eq:approx-AHS}
    r_n< r \textrm{ and } \cS_{r_n,r}(a_n)\in (a,a').
\end{equation}

\noindent Then \eqref{eq:approximation+} is a consequence of Proposition \ref{p:perfect} (ii) and Proposition \ref{p:properties} (i). In the same spirit, we present the following property of BESQ flows, called instantaneously coalescing property. 

\begin{proposition}\label{instantaneous_coalescing}
Let $\cS$ be a general $\rm BESQ^\delta$ flow for $\delta\in \r$. Then with probability 1, for every $(a,r)\in \r_+\times\r$ and $\varepsilon>0$, there exist some  $(a_n,r_n)\in \mathcal{D}$ arbitrarily close to $(a,r)$ 
such that $\cS_{r,x}(a)=\cS_{r_n,x}(a_n)$ for all $x\geq r+\varepsilon$. 
\end{proposition}

\begin{proof}
    By Proposition \ref{p:properties} (i), for $(a,r)\in \r_+\times\r$ and $\varepsilon>0$,  one can find $a'>a$ such that $\cS_{r,r+\varepsilon}(a')=\cS_{r,r+\varepsilon}(a)$. By \eqref{eq:approx-AHS}, 
   there exists $(a_n,r_n)\in \mathcal{D}$ such that $r_n\le r$ and $\cS_{r_n,r}(a_n) \in (a,a')$. Hence $\cS_{r,x}(a)=\cS_{r_n,x}(a_n)= \cS_{r,x}(a')$ for all $x\ge r+\varepsilon$ by the coalescent property in Proposition \ref{p:perfect} (ii). 
\end{proof}

\noindent We finally state the following comparison principle.
\begin{proposition}\label{p:comparison}
    Suppose $\delta \le \delta'$ and let $\cS$ and $\cS'$ be respectively a  ${\rm BESQ}^\delta$ flow and a ${\rm BESQ}^{\delta'}$ flow driven by the same white noise $\cW$.   
Then almost surely, $\cS\le \cS'$, i.e. for all $a,a'\ge 0$ and all $r\le x $,  we have the implications: if $\cS_{r,x}(a)\le a'$, resp. $a\le \cS'_{r,x}(a')$, then $\cS_{r,y}(a)\le \cS'_{x,y}(a')$, resp. $\cS_{x,y}(a)\le \cS'_{r,y}(a')$, for all $y\ge x$. 
\end{proposition}
\begin{proof}
 For fixed $(a,r)\in \r_+\times\r$, property ({\bf P}1)  if $\delta\ge 0$ or ({\bf P}3) if $\delta<0$ implies that $(\cS'_{r,r+s}(a)-\cS_{r,r+s}(a),\, s\ge 0)$ is a non-negative process a.s., hence $\cS_{r,y}(a)\le \cS'_{r,y}(a)$  for all $y\ge r$. It holds simultaneously for all $(a_n,r_n)$ in a deterministic countable set $\cD$ dense in $\r_+\times \r$. Suppose that $\cS_{r,x}(a)\le a'$. Let $y>x$. By Proposition \ref{p:properties} (i) with $(x,y,a')$ in place of $(r,x,a)$, one can find $b>a'$ such that $\cS'_{x,y}(a')=\cS'_{x,y}(b)$. Let $(a_n,r_n) \in \cD$ such that  both $\cS_{r_n,x}(a_n)$ and $\cS'_{r_n,x}(a_n)$ are in $I=(\cS_{r,x}(a),b)$. To find such $(a_n,r_n)$, we choose $(a_p,r_p)$ such that $r_p<x$ and $\cS_{r_p,x}(a_p) \in I$ by \eqref{eq:approx-AHS}, then $(a_q,r_q)$ such that  $r_q \in(r_p,x)$ and  $\cS'_{r_q,x}(a_q) \in (\cS_{r_p,x}(a_p),b)$ by the same equation. Any point $(a_n,r_n)$ such that $r_n \in (r_p,x)$ and $\cS_{r_p,r_n}(a_p) < a_n < \cS'_{r_q,r_n}(a_q)$ would fit by the coalescence property in Proposition \ref{p:perfect} (ii). Then, $\cS_{r,y}(a) \le \cS_{r_n,y}(a_n)$  by the coalescence property, which is smaller than $ \cS'_{r_n,y}(a_n) $ by what we already proved, which is itself smaller than $\cS'_{x,y}(b)$ by another use of the coalescence property. We proved $\cS_{r,y}(a)\le \cS'_{x,y}(a')$ indeed. A similar reasoning yields the case $a\le \cS'_{r,x}(a')$.
\end{proof}


\subsection{Bifurcation events}\label{s:bifur}

Let $\delta \in \r$ and $\cS$ be a general $\rm BESQ^\delta$ flow driven by $\cW$. Recall that for $a>0$ and $r\le x$, we write $\cS_{r,x}(a-):=    \lim_{a'\uparrow a} \cS_{r,x}(a')$.

\begin{definition}\label{def:bifur}
We call a point $(a, r) \in (0,\infty)\times\r$ a bifurcation point if $\cS_{r,x}(a) > \cS_{r,x}(a-)$ for some  $x > r$. 
\end{definition}

We could as well look at possible bifurcation points on the line $a=0$, but we omit their discussion for sake of brevity. Notice that the bifurcation between $\cS_{r,\cdot}(a-)$ and $\cS_{r\cdot}(a)$ may happen only at the beginning by  Proposition \ref{p:perfect} (ii) and Proposition \ref{p:properties} (ii).
 Proposition \ref{p:bifur dual} characterizes bifurcation points in terms of the dual flow $\cS^*$ of the following proposition. 

\begin{proposition} \cite[Proposition 2.7]{aïdékon2023stochastic} \label{p:BESQdual}
Define the dual flow ${\mathcal S}^*$ by, for $r\le x$ and $b\ge 0$,
\begin{equation}\label{def:dual}
    {\mathcal S}^*_{r,x}(b) := \inf\{a\ge 0\,:\, {\mathcal S}_{-x,-r}(a)>b\}.
\end{equation}
Then ${\mathcal S}^*$ is a general ${\rm BESQ}^{2-\delta}$ flow driven by the white noise $-\cW^*$, where $\cW^*$ is the image of $\cW$ under the map $(a,r)\mapsto(a,-r)$. In the case $\delta\in (0,2)$, ${\mathcal S}^*$ is killed if ${\mathcal S}$ is not killed, and it is not killed if ${\mathcal S}$ is killed.  
\end{proposition}

\noindent The fact that  ${\mathcal S}^*$ is driven by $-\cW^*$ is for example a consequence of the embedding of ${\mathcal S}$ and ${\mathcal S}^*$ in the PRBM \cite[Proposition 3.6]{aïdékon2023stochastic} and of \cite[Theorem 5.1]{aidekon2024infinite}. We will call the flow lines of $\cS$ forward flow lines, by opposition to the dual flow lines which are the flow lines of $\cS^*$. In the next proposition, we say that a bifurcation point $(a,r)\in (0,\infty)\times \r$ is an ancestor of $(b,x)$ if $\cS_{r,x}(a-) \le b <\cS_{r,x}(a)$, see \cite{bertoin-legall00} for this notion in the setting of CSBPs.
\begin{proposition}\label{p:bifur dual}
 Almost surely, for any $a>0$, $b\ge 0$ and $r<x$,
\begin{equation}\label{equivalence bifurcation}
   (a,r) \textrm{ is an ancestor of } (b,x) \Leftrightarrow   a=\cS^*_{-x,-r}(b). 
\end{equation}
\end{proposition}
 \begin{proof} 
 It is a consequence of the definition of $\cS^*$ and Proposition \ref{p:properties}.   
 \end{proof}

The following proposition shows that forward and dual flow lines do not cross, a standard property for stochastic flows. 

\begin{proposition}\label{p:properties dual}
Let $\cS$ be a general $\rm BESQ^\delta$ flow and $\cS^*$ be its dual. The following statements hold almost surely.
\begin{enumerate}[(i)]
\item For any $r<x$, $b\ge 0$ and $a\ge \cS^*_{-x,-r}(b)$,  $\cS_{r,x}(a)>b$ and 
\begin{equation}\label{eq:properties dual+}
    \cS_{r,y}(a)\ge \cS^*_{-x,-y}(b),\qquad  y\in [r,x].
\end{equation}
Moreover  $\cS_{r,y}(a)> \cS^*_{-x,-y}(b)$ if $y\in (r,x)$ and $\cS_{r,y}(a)>0$.
\item  For any $r<x$, $b\ge 0$ such that $a:=\cS^*_{-x,-r}(b)>0$ and $0\le a'< a$, 
\[
\cS_{r,y}(a')\le \cS^*_{-x,-y}(b),\qquad  y\in [r,x].
\]
 Moreover $\cS_{r,y}(a')< \cS^*_{-x,-y}(b)$ if  $y\in (r,x)$ and $\cS^*_{-x,-y}(b)>0$.
\end{enumerate}
\end{proposition}
 \begin{proof}
 From the definition of the dual flow in \eqref{def:dual}, we have $\cS_{r,x}(a')>b$ for any $a'>\cS^*_{-x,-r}(b)$. By Proposition \ref{p:properties} (ii), it implies $\cS_{r,x}(a)> b$ also at $a=\cS^*_{-x,-r}(b)$. Let $y\in (r,x)$. Let $c:=\cS^*_{-x,-y}(b)$. If $c=0$, \eqref{eq:properties dual+} is clear. Otherwise, by the perfect flow property of Proposition \ref{p:perfect} (i) applied to $\cS^*$, $\cS^*_{-x,-r}(b)=\cS^*_{-y,-r}(c)$. Hence if $a\ge \cS^*_{-x,-r}(b)=\cS^*_{-y,-r}(c)$, then $\cS_{r,y}(a)>c$ by what we just proved with $(y,c)$ in place of $(x,b)$. It implies (i). If $a'<\cS^*_{-x,-r}(b)$, then $\cS_{r,x}(a')\le b$ by definition of $\cS^*$. Statement (ii) is then a consequence of (i), reversing the roles of $\cS$ and $\cS^*$.
 \end{proof}

\bigskip

As a corollary, we deduce that forward flow lines cannot hit bifurcation points in $(0,\infty)\times \r$, except possibly at their starting point. 
\begin{corollary}\label{c: no middle bifurcation}
    Let $\cS$ be a general ${\rm BESQ}^\delta$ flow. Almost surely, for any $a\ge 0$ and $r<y$ such that $\cS_{r,y}(a)>0$, the point $(\cS_{r,y}(a),y)$ is not a bifurcation point.
\end{corollary}
\begin{proof}
    By Proposition \ref{p:bifur dual}, it is enough to show that a.s., one cannot find $a,b\ge 0$ and $r<y<x$ such that $\cS_{r,y}(a)=\cS^*_{-x,-y}(b)>0$. Let $a,b\ge 0$ and $r<y<x$. If $a\ge \cS^*_{-x,-r}(b)$, we apply Proposition \ref{p:properties dual} (i) to see that if $\cS_{r,y}(a)>0$, then $\cS_{r,y}(a)>\cS^*_{-x,-y}(b)$. If $a<\cS^*_{-x,-r}(b)$ and $\cS^*_{-x,-y}(b)>0$, we use Proposition \ref{p:properties dual} (ii).
\end{proof}

\begin{theorem}\label{t:bifur}
The set of bifurcation points has Hausdorff dimension $\frac32$ almost surely.
\end{theorem}
\begin{proof} 
Let  $\mathcal{D}$ be a deterministic countable set  of points dense in $(0,\infty)\times \r$. Any bifurcation point should be an ancestor of a point in $\mathcal D$. By  Proposition \ref{p:bifur dual} (i), it yields that bifurcation points  are exactly the points  with $a>0$ which lie on the dual flow lines emanating from a point in $\mathcal D$ (except its starting point). Because ${\mathcal D}$ is countable, it remains to compute the Hausdorff dimension of the graph of a single dual flow line. According to Proposition \ref{p:BESQdual}, a dual flow line is a (possibly killed) $\rm BESQ^{2-\delta}$ process, whose graph  has dimension $\frac32$ by equation \eqref{BM:haus_dim}. 
\end{proof}

\bigskip

We close this section by studying points which are common bifurcation points for two different flows. We consider two $\rm BESQ$ flows $\cS^1$ and $\cS^2$ of dimensions $\delta_1<\delta_2$. 

\begin{theorem}\label{t:bifur quadri}
    The set of points $(a,r)\in (0,\infty)\times \r$ such that $\cS^1_{r,x}(a)>\cS^2_{r,x}(a-)$ for some $x>r$ is nonempty if and only if $d:=\delta_2-\delta_1 \in (0,2)$. In that case, its Hausdorff dimension is $\min(2-d,\frac{3-d}{2})$ a.s. 
\end{theorem}

\begin{remark}
    Such points are necessarily bifurcation points for both $\cS^1$ and $\cS^2$. Indeed the comparison principle in Proposition \ref{p:comparison} shows that $\cS^2_{r,x}(a)\ge \cS^1_{r,x}(a)$ and $\cS^2_{r,x}(a-)\ge \cS^1_{r,x}(a-)$ for all $r\le x$ and $a>0$. Hence $\cS^1_{r,x}(a)>\cS^2_{r,x}(a-)$ implies $\cS^1_{r,x}(a-)\le \cS^2_{r,x}(a-)<\cS^1_{r,x}(a)\le \cS^2_{r,x}(a)$. See Figure \ref{fig:bifur1}.
\end{remark}

\begin{figure}[htbp]
\centering
   \scalebox{0.9}{ 
        \def\svgwidth{0.7\columnwidth}
\begingroup%
  \makeatletter%
  \providecommand\color[2][]{%
    \errmessage{(Inkscape) Color is used for the text in Inkscape, but the package 'color.sty' is not loaded}%
    \renewcommand\color[2][]{}%
  }%
  \providecommand\transparent[1]{%
    \errmessage{(Inkscape) Transparency is used (non-zero) for the text in Inkscape, but the package 'transparent.sty' is not loaded}%
    \renewcommand\transparent[1]{}%
  }%
  \providecommand\rotatebox[2]{#2}%
  \newcommand*\fsize{\dimexpr\f@size pt\relax}%
  \newcommand*\lineheight[1]{\fontsize{\fsize}{#1\fsize}\selectfont}%
  \ifx\svgwidth\undefined%
    \setlength{\unitlength}{345bp}%
    \ifx\svgscale\undefined%
      \relax%
    \else%
      \setlength{\unitlength}{\unitlength * \real{\svgscale}}%
    \fi%
  \else%
    \setlength{\unitlength}{\svgwidth}%
  \fi%
  \global\let\svgwidth\undefined%
  \global\let\svgscale\undefined%
  \makeatother%
  \begin{picture}(1,0.52797623)%
    \lineheight{1}%
    \setlength\tabcolsep{0pt}%
    \put(0,0){\includegraphics[width=\unitlength,page=1]{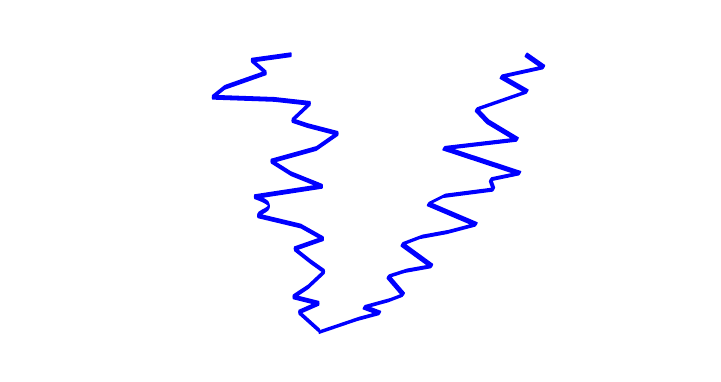}}%
    \put(0.4038871,0.01457691){\color[rgb]{0,0,0}\transparent{0.82599998}\makebox(0,0)[lt]{\lineheight{1.25}\smash{\begin{tabular}[t]{l}$(a,r)$\end{tabular}}}}%
    \put(0.52406002,0.49231692){\color[rgb]{0,0,0}\transparent{0.82599998}\makebox(0,0)[lt]{\lineheight{1.25}\smash{\begin{tabular}[t]{l}$\cS_{r,x}^1(a)$\end{tabular}}}}%
    \put(0.13470827,0.49231692){\color[rgb]{0,0,0}\transparent{0.82599998}\makebox(0,0)[lt]{\lineheight{1.25}\smash{\begin{tabular}[t]{l}$\cS_{r,x}^1(a-)$\end{tabular}}}}%
    \put(0.68398858,0.49231692){\color[rgb]{0,0,0}\transparent{0.82599998}\makebox(0,0)[lt]{\lineheight{1.25}\smash{\begin{tabular}[t]{l}$\cS_{r,x}^2(a)$\end{tabular}}}}%
    \put(0.32620421,0.49231692){\color[rgb]{0,0,0}\transparent{0.82599998}\makebox(0,0)[lt]{\lineheight{1.25}\smash{\begin{tabular}[t]{l}$\cS_{r,x}^2(a-)$\end{tabular}}}}%
    \put(0,0){\includegraphics[width=\unitlength,page=2]{bifur1.pdf}}%
  \end{picture}%
\endgroup%

    }
    \caption{Schematic representation of the bifurcation point. The black lines represent the flow $\cS^1$ and the blue lines represent the flow $\cS^2$. Here, $\cS^1_{r,x}(a) > \cS^2_{r,x}(a-)$ for some $x>r$, where $\cS^1$ and $\cS^2$ are two $\rm BESQ$ flows of dimensions $\delta_1<\delta_2$.}
    \label{fig:bifur1}
\end{figure}

\begin{proof}[Proof of Theorem \ref{t:bifur quadri}]
    Let $\mathcal D$ be a deterministic countable set dense in $(0,\infty)\times \r$. Any such bifurcation point $(a,r)$  must be an ancestor of a point $(b,x)\in \mathcal D$ simultaneously for $\cS^1$ and $\cS^2$. Conversely, if it is an ancestor of $(b,x)$, then $\cS^1_{r,x}(a-)\le b< \cS^1_{r,x}(a)$ and $\cS^2_{r,x}(a-)\le b<\cS^2_{r,x}(a)$ so that $\cS^2_{r,x}(a-) <\cS^1_{r,x}(a)$. Hence we can fix $(b,x)$ and consider the bifurcation points $(a,r)$ for $\cS^1$ and $\cS^2$ which are ancestors of $(b,x)$.  Recall by Proposition \ref{p:BESQdual} that the dual flows $\cS^{1,*}$ and $\cS^{2,*}$ are killed BESQ flows with dimensions $2-\delta_1$ and $2-\delta_2$ respectively. By Proposition \ref{p:bifur dual}, the ancestors of $(b,x)$  are the points which belong to both dual lines $\cS^{1,*}_{-x,\cdot}(b)$ and $\cS^{2,*}_{-x,\cdot}(b)$ (except its starting point) until $\cS^{2,*}_{-x,\cdot}(b)$ hits $0$ (if $2-\delta_2<2$, after which time the dual line stays at $0$).  By property ({\bf P}1) if $2-\delta_2\ge 0$ or ({\bf P}3) if $2-\delta_2< 0$ of Section \ref{s:def}, $\cS^{1,*}_{-x,\cdot}(b)-\cS^{2,*}_{-x,\cdot}(b)$ is up to this time a ${\rm BESQ}^{d}_0$ process.  In particular it goes back to $0$ a.s. if and only if  $d<2$. In that case, the Hausdorff dimension of the set of ancestors of $(b,x)$ is $\min(2-d, \frac{3-d}{2})$ by Corollary \ref{dimension:BESQ_difference}.    
\end{proof}


\subsection{BESQ flows with varying parameters}

We will deal with BESQ flows, where we allow the parameter $\delta$ to take different values. It motivates the following definition. Recall the notion of a drift function and the notation of Definition \ref{def:drift}. 
\begin{definition}\label{def:vary}
   Let $\cW$ be a white noise with respect to some filtration $\cF$. Let $\overline{\delta}$ be a drift function with respect to $\cF$. We call a collection of continuous processes $\cS=(\cS_{r,x}(a),\, x\ge r)_{r\in \r,a\ge 0}$ a ${\rm BESQ}(\overline{\delta})$ flow driven by $\cW$ if 
   
    (1) The restriction of $\cS$ to each $A_i$ is a ${\rm BESQ}^{\delta_i}$ flow driven by $\cW$.

    (2) The regularity conditions of Definition \ref{def:BESQflow} hold.

If $E$ is a (possibly random) Borel set, we call $\rm BESQ(\overline{\delta})$ flow killed on $E$ the flow obtained from the $\rm BESQ(\overline{\delta})$ flow by absorbing the flow line $\cS_{r,\cdot}(a)$ at $0$ at time $\inf\{x \in E\cap (r,\infty) \,:\, \cS_{r,x}(a)=0\}$.

 If $E=\r$, we will simply say that $\cS$ is a killed  $\rm BESQ(\overline{\delta})$ flow. We call general  $\rm BESQ(\overline{\delta})$ flow a $\rm BESQ(\overline{\delta})$ flow or a  killed $\rm BESQ(\overline{\delta})$ flow.
\end{definition}

Notice that a flow line of a ${\rm BESQ}(\overline{\delta})$ flow is the strong solution of (remember that a ${\rm BESQ}^{\delta}$ flow line is absorbed at $0$ when $\delta<0$)
\begin{equation}\label{eq:BESQflow vary}
\dd_x {\mathcal S}_{r,x}(a) = 
2 {\mathcal W}([0,{\mathcal S}_{r,x}(a)], \dd x) + 
\begin{cases}
    \overline{\delta}(x) \dd x & \textrm{ if } \overline{\delta}(x)\ge 0, \\
    \overline{\delta}(x) \ind_{\{{\mathcal S}_{r,x}(a)>0\}} \dd x & \textrm{ if } \overline{\delta}(x)< 0.
\end{cases}
\end{equation}

If $\delta\in \r$, we identify $\delta$ and the constant function $\overline{\delta}\equiv \delta$, so that a $\rm BESQ(\delta)$ flow is simply a $\rm BESQ^\delta$ flow.  
The following lemma shows that one can construct the $\rm BESQ(\overline{\delta})$ flow by gluing $\rm BESQ^{\delta}$ flows. 

\begin{lemma}\label{l:gluing}
     Let $\cS$ be a (non-killed) ${\rm BESQ}(\overline{\delta})$ flow.  The restriction of $\cS$ to the closure of each $A_i$ gives a ${\rm BESQ}^{\delta_i}$ flow, call it $\cS^i$.  Let $x>r$ and $a\ge 0$. Let $i\le j$ such that $\inf A_i\le r<\sup A_i$, and $\inf A_j < x\le \sup A_j$. We can define by induction $t_i:=r$, $a_i:=a$, and for $j \ge k> i$, $a_k:=\cS^{k-1}_{t_{k-1},t_k}(a_{k-1})$, $t_{k}:=\inf A_k$. Almost surely for all $r\le x$ and $a\ge 0$, $\cS_{r,x}(a)=\cS^j_{t_{j},x}(a_j)$.  
 \end{lemma}
 \begin{proof}
 We only show that $\cS^i$ is indeed a ${\rm BESQ}^{\delta_i}$ flow driven by $\cW$. The rest follows from the perfect flow property. So far we know that the restriction of $\cS$ to $A_i=(\widetilde{t}_i,\widetilde{t}_{i+1})$ is a ${\rm BESQ}^{\delta_i}$ flow driven by $\cW$ by definition. Let ${\mathcal T}^i$ be the ${\rm BESQ}^{\delta_i}$ flow driven by $\cW$. We have ${\mathcal T}^i=\cS^i$ on $A_i$. Let us show that is is also the case on the closure of $A_i=(\widetilde{t}_i,\widetilde{t}_{i+1})$.  If $\inf A_i=\widetilde{t}_i>-\infty$ and $0\le a<a'$, $\cS_{\widetilde{t}_i,x}(a)\le {\mathcal T}^i_{\widetilde{t}_i,x}(a')$ for $x>\widetilde{t}_i$ small enough by continuity, which implies that $\cS_{\widetilde{t}_i,x}(a)\le {\mathcal T}^i_{\widetilde{t}_i,x}(a')$ for all $\widetilde{t}_i\le x< \widetilde{t}_{i+1}$ by the perfect flow property. By right-continuity, it implies that $\cS_{\widetilde{t}_{i},x}(a)\le {\mathcal T}^i_{\widetilde{t}_{i},x}(a)$. We also have $\cS_{\widetilde{t}_{i},x}(a)\ge {\mathcal T}^i_{\widetilde{t}_{i},x}(a)$ by symmetry. So $\cS_{\widetilde{t}_{i},x}(a)={\mathcal T}^i_{\widetilde{t}_{i},x}(a)$ for $\widetilde{t}_{i}\le r\le x<\widetilde{t}_{i+1}$. We then extend the flow to $x=\widetilde{t}_{i+1}$ if $\widetilde{t}_{i+1}<\infty$ by continuity.
 \end{proof}

Conversely,  such a gluing produces a $\rm BESQ({\overline \delta})$ flow since the regularity conditions of Definition \ref{def:BESQflow} still hold after gluing.  We now discuss duality for ${\rm BESQ}(\overline{\delta})$ flows. The analog of Proposition \ref{p:BESQdual}  reads as follows.

\begin{proposition}\label{p:BESQdual vary}
    Let $\overline{\delta}$ be a deterministic drift function and $\overline{\delta}^*(x):=\overline{\delta}(-x)$. If $\cS$ is a $\rm BESQ(\overline{\delta})$ flow, resp. killed $\rm BESQ(\overline{\delta})$ flow, driven by $\cW$, then its dual \eqref{def:dual} is a killed $\rm BESQ(2-\overline{\delta}^*)$ flow, resp. a   $\rm BESQ(2-\overline{\delta}^*)$ flow, driven by $-\cW^*$.
\end{proposition}
\begin{proof}
    Let $\cS$ be a $\rm BESQ(\overline{\delta})$ flow. By Lemma \ref{l:gluing}, $\cS$ is obtained by the gluing of the $\rm BESQ^{\delta_i}$ flows $\cS^i$, $1\le i\le n$. On each $A_i$, the dual of $\cS^i$ is a killed $\rm BESQ^{2-\delta_i}$ flow $\cS^{*,i}$ driven by $-\cW^*$. Let $b\ge 0$, $r< x$,  $\mathcal Y$ be the killed $\rm BESQ(2-\overline{\delta}^*)$ flow  driven by $-\cW^*$ and $a:={\mathcal Y}_{-x,-r}(b)$. We want to show that $a=\cS^*_{-x,-r}(b)$. Let $(a_k)_{i\le k\le j}$, $r=t_i< t_{i+1}<  \cdots < t_j < x$ as defined in Lemma \ref{l:gluing} and let $t_{j+1}:=x$. Suppose first that ${\mathcal Y}_{-x,\cdot}(b)$ does not touch $0$ on $(-x,-r)$.  By the perfect flow property, it is the composition of flow lines of $\cS^{*,k}$, $i\le k\le j$, so that $a=\cS^{*,i}_{-t_{i+1},-t_i}(b_{i+1})$ where $b_{j+1}=x$, $b_k=\cS^{*,k}_{-t_{k+1},-t_{k}}(b_{k+1})={\mathcal Y}_{-x,-t_k}(b)$, $i< k\le j$. By Proposition \ref{p:properties dual} (i) applied to  $A_i$, $a_{i+1}=\cS_{t_i,t_{i+1}}(a)> b_{i+1}$. By induction, $\cS_{t_k,t_{k+1}}(a_k) > b_{k+1}$, $i\le k\le j$. By composition, $\cS_{t,x}(a) > b$. Similarly, let $a'<a$. 
     Using now Proposition \ref{p:properties dual} (ii), we have $\cS_{t_i,t_{i+1}}(a')\le b_{i+1}$. The independence of the flows $\cS^{k}$ and Proposition \ref{instantaneous_coalescing} rules out the equality. Hence $\cS_{t_i,t_{i+1}}(a')< b_{i+1}$ and we can proceed by induction to show that  $\cS_{t,x}(a') \le b$.
    It yields that $\cS^*_{-x,-r}(b)=a$ by definition of $\cS^*$. We deal now with the case where ${\mathcal Y}_{-x,\cdot}$ touched $0$ on $(-x,-r)$, in particular $a=0$. Let $r'>r$ be the sup of $s\in (r,t_j)$ such that ${\mathcal Y}_{-x,-s}(b)=0$. If $r'<t_j$, we have $\cS_{r',x}(0)> b$ from what we already proved, and since  $\cS_{r,x}(a)\ge \cS_{r',x}(0)$ (by the perfect flow property for example), we get $\cS^*_{-x,-r}(b)=0=a$ indeed. If $r'=t_j$, we have $\cS_{t_j,x}(0)>b$ by reasoning on the $\rm BESQ^{\delta_j}$ flow $\cS^j$, hence $\cS^*_{-x,-r}(b)=0=a$ in this case as well. We proved that $\cY$ is the dual flow of $\cS$. If $a:={\mathcal Y}_{-x,-r}(b)$, then $\cS_{r,x}(a)>b$ and $\cS_{r,x}(a')\le b$ for all $a'<a$. 

    Finally, we show that the dual of ${\mathcal Y}$ is $\cS$, which would end the proof. Fix $c\ge 0$, $r< x$ and let $b:=\cS_{r,x}(c)$. If we set $a:={\mathcal Y}_{-x,-r}(b)$, we proved that $\cS_{r,x}(a) > b$, hence $a> c$, i.e. ${\mathcal Y}^*_{r,x}(c)\le b$.  We prove the reverse inequality. We can suppose that $b>0$. We want to show that for all $b'<b$, ${\mathcal Y}_{-x,-r}(b')\le c$. Take $b'<b$. We can suppose that ${\mathcal Y}_{-x,-r}(b')>0$. For $a'<{\mathcal Y}_{-x,-r}(b')$, we have $\cS_{r,x}(a')\le b'<b$. It yields that $a'<c$ hence ${\mathcal Y}_{-x,-r}(b')\le c$ indeed by taking the limit in $a'$. We proved ${\mathcal Y}^*_{r,x}(c)= b=\cS_{r,x}(c)$ which is what we wanted to show.

\end{proof}

Proposition \ref{p:perfect} still holds for the general $\rm BESQ(\overline{\delta})$ flow $\cS$ as a consequence of the perfect flow property when $\cS$ is non-killed and by construction when $\cS$ is killed. Proposition \ref{p:properties} also remains true. Indeed, in the notation of Lemma \ref{l:gluing}, for each $r\in \r$, one can find $r'>r$ such $(r,r')$ is contained in some $A_i$. We can then apply Proposition \ref{p:properties} to $A_i$ when $x\le r'$, then the coalescence property of Proposition \ref{p:perfect} when $x\ge r'$.  Similarly, equation \eqref{eq:approx-AHS} is valid by restriction to each $A_i$. It implies \eqref{eq:approximation+} and Proposition \ref{instantaneous_coalescing}. The comparison principle in Proposition \ref{p:comparison} still holds for some  ${\rm BESQ}(\overline{\delta})$ and  ${\rm BESQ}(\overline{\delta'})$ flows with $\overline{\delta}\le \overline{\delta'}$, by following the same proof. 
The proofs of Proposition \ref{p:bifur dual} and \ref{p:properties dual} still hold without change. Finally, Corollary \ref{c: no middle bifurcation} is still true. Indeed, let $r<y$ and $a\ge 0$ such that $\cS_{r,y}(a)>0$. If $\inf A_i<y<\sup A_i$ for some $i$, then we can apply Corollary \ref{c: no middle bifurcation} to the ${\rm BESQ}^{\delta_i}$ flow to show that $(\cS_{r,y}(a),y)$ is not a bifurcation point. If $y=\inf A_i$ for some $i$, one take $(a_n,r_n)$ in some deterministic countable set ${\mathcal D}$ such that $\cS_{r_n,y}(a_n)=\cS_{r,y}(a)$ by Proposition \ref{instantaneous_coalescing},  and by independence of $\cS^{i-1}$ and $\cS^i$, argue that a.s.,  $(\cS_{r_n,y}(a_n),y)$ cannot be a bifurcation point.

\bigskip

\fbox{
\begin{minipage}{0.9\textwidth}
\textit{
From now on, we will freely apply Propositions \ref{p:perfect}, \ref{p:properties}, \ref{instantaneous_coalescing}, \ref{p:bifur dual}, \ref{p:properties dual}, Corollary \ref{c: no middle bifurcation}, equations  \eqref{eq:approximation+} and \eqref{eq:approx-AHS} to general ${\rm BESQ}(\overline{\delta})$ flows, and Proposition \ref{p:comparison} to ${\rm BESQ}(\overline{\delta})$ flows.
}
\smallskip
\end{minipage}
}

\bigskip

\noindent The following lemma characterizes a ${\rm BESQ}(\overline{\delta})$ killed on some set $E$. 
\begin{lemma}\label{l:killed}
    Let $E$ be some Borel set of $\r$ and $\cS=(\cS_{r,x}(a),\, x\ge r)_{(a,r)\in \r_+\times \r}$ be a collection of continuous processes. 
    For $(a,r)\in \r_+\times \r$, we let $\varphi(a,r):=\inf\{x \in E\cap (r,\infty)\,:\, \cS_{r,x}(a) =0\}$. We suppose that 
    \begin{enumerate}[(1)]
        \item for each fixed $(a,r)$, the process $(\cS_{r,x}(a),\, x\ge r)$ is the strong solution of \eqref{eq:BESQflow vary} up to time $\varphi(a,r)$. 
        \item Almost surely,
        \begin{enumerate}[(i)]
            \item for all $r\in \r$ and $a\ge 0$, $\cS_{r,r}(a)=a$,
            \item for all $a\ge 0$ and  $r\le x<\varphi(a,r)$, the map $a' \mapsto \cS_{r,x}(a')$ is c\`adl\`ag at $a$,
            \item for any  $a\ge 0$ and $r\le x \le y <\varphi(a,r)$, $y<\varphi(\cS_{r,x}(a),x)$ and $\cS_{r,y}(a)=\cS_{x,y}\circ \cS_{r,x}(a)$,
            \item for all $r\ge 0$, $a\ge 0$ and $x\ge \varphi(a,r)$, $\cS_{r,x}(a)=0$.
        \end{enumerate}
    \end{enumerate}
Then $\cS$ is a ${\rm BESQ}(\overline{\delta})$ flow killed on $E$ driven by $\cW$.
\end{lemma}
\begin{proof}
    Note that $\varphi(a,r)\le \varphi(a',r)$ if $a\le a'$. Otherwise, the flow line $\cS_{r,\cdot}(a')$ would have met $\cS_{r,\cdot}(a)$ at a time when both are still positive, at which time we necessarily have coalescence by (iii). Similarly,   $\varphi(a,r)=\varphi(\cS_{r,x}(a),x)$ for any $x<\varphi(a,r)$ by (iii).  Let $\tcS$ be the non-killed ${\rm BESQ}(\overline{\delta})$ flow driven by $\cW$. Let $\mathcal{D}$ be a deterministic countable set of points $(a_n,r_n)$ dense in $\r_+\times \r$. By pathwise uniqueness, one has a.s., for any $(a_n,r_n)\in \mathcal{D}$, $\cS_{r_n,x}(a_n)= \tcS_{r_n,x}(a_n)$ for all $x\in [r_n,\varphi(a_n,r_n)]$.  Let $(a,r)\in \r_+\times \r$. By \eqref{eq:approx-AHS} (applied to the killed ${\rm BESQ}(\overline{\delta})$ flow, say), one can find $(a_n,r_n)$ such that $r_n\le r$, $\min_{r'\in [r_n,r]} \tcS_{r_n,r'}(a_n)>0$  and $b_n:=\tcS_{r_n,r}(a_n)>a$ is arbitrarily close to $a$. Since $b_n=\cS_{r_n,r}(a_n)$, (iii) implies $\cS_{r,x}(b_n)=\cS_{r_n,x}(a_n)$ if $x<\varphi(a,r)$ where we use that $\varphi(a,r) \le \varphi(b_n,r)=\varphi(a_n,r_n)$. Recall that $\cS_{r_n,x}(a_n)=\tcS_{r_n,x}(a_n)$ by pathwise uniqueness. Using (ii) and Proposition \ref{p:properties} (i) for $\tcS$, we deduce that $\cS_{r,x}(a)=\tcS_{r,x}(a)$ for all $x<\varphi(a,r)$.  Finally, the flow lines are absorbed at $0$ at time $\varphi(a,r)$ by (iv). It completes the proof. 
\end{proof}



\section{Decomposition of BESQ flows along a flow line}
\label{s:decomposition}

For a deterministic constant $r_0\in \r$, we let $S=(S_x,\, x\ge r_0)$ be some continuous, non-negative, $\cF$-predictable process. 
We introduce the martingale measures  $\cW^-_S$ and $\cW^+_S$ defined by
\begin{align}
 \label{eq:W-}   \cW^-_S(\dd \ell, \dd x) &:= \cW(\dd \ell, \dd x),\, && \ell\le S_x, x\ge r_0, \\
\label{eq:W+}    \cW^+_S(\dd \ell, \dd x)  &:= \cW(S_x+\dd \ell, \dd x), && \ell\ge 0, x\in \r.
\end{align}

\noindent To be more precise, we extend $S$ to $\r$ by setting $S_x=0$ when $x< r_0$ and 
we define $\cW^-_S$ and $\cW^+_S$ as the martingale measures $f^-_S \cdot \cW$ and $\theta_{S} \cW$ of Section \ref{s:mart} where $f^-_S(\ell,x):= \ind_{ [0,S_x]}(\ell,x)$. We already observed in Section \ref{s:mart} that $\cW^+_S$ defines a white noise with respect to $\cF$. 

\bigskip

Following Definition \ref{def:drift}, let ${\overline \delta}$ and ${\overline \delta'}$ be two drift functions,  and let $r_1$ be an $\cF$-stopping time.  Let $\cW^-_S$ and $\cW^+_S$ be defined via \eqref{eq:W-} and \eqref{eq:W+} with the white noise $\cW$ and $(S_x,\,x\ge r_0)$ being the ${\rm BESQ}(\overline{\delta})$ flow line driven by $\cW$ starting at $(0,r_0)$.
Finally, we let $\cS$ be the  ${\rm BESQ}(\overline{\delta}\,|_{r_1}\,\overline{\delta'})$ flow driven by $\cW$ and set $Y_x:=\cS_{r_0,x}(0)$ for $x\ge r_0$. 

\begin{definition}\label{d:minus plus flow line}
We say that $Y$ is a {\bf $\ominus$-flow line}, resp. a {\bf $\oplus$-flow line}, if the stopping time $r_1$ is a stopping time with respect to the natural filtration of $\cW^-_S$, resp. $\cW^+_S$. 
\end{definition}

The following two propositions will be used in Section \ref{s:meeting} to study the interaction between flow lines in a BESQ flow.


\begin{proposition}\label{p:independence_W-&W^+}
Let $Z\ge 0$ be a  random variable independent of $\cW$ and $r_0\in \r$ a constant. In the setting of Definition \ref{def:vary}, let $\overline{\delta}$ be a deterministic drift function and let $(S_x,\,x\ge r_0)$ be the BESQ$(\overline{\delta})$ flow line driven by $\cW$ starting at $(Z,r_0)$. Define $\cW^-_S$ and $\cW^+_S$ via \eqref{eq:W-} and \eqref{eq:W+}. Then $\cW^-_S$ and $\cW^+_S$ are independent. In particular, $\cW^+_S$ is a white noise independent of $S$ and therefore of $S_0=Z$.
\end{proposition}
\begin{proof}
   Let $\cW_1$ and $\cW_2$ be two independent white noises and $Z'\dequiv Z$ independent of $(\cW_1,\cW_2)$. Define the process $S'$ as the BESQ$(\overline{\delta})$ flow line driven by $\cW_1$ starting at $(Z',r_0)$. Let $\cW_1^-$ be as in \eqref{eq:W-} with $(\cW,S)$ replaced with  $(\cW_1,S')$.  Then define the white noise $\cW'$ as, for every deterministic $g\in L^2(\r_+\times\r)$,
   \begin{align}\label{eq:W'}
       \cW'(g) := \int_{\r_+\times
        \r} g(\ell,r) \cW_1^-(\dd \ell,\dd r) + \int_{\r_+\times\r} g(\ell+S'_r,r) \cW_2(\dd \ell, \dd r).
   \end{align}
   By density of simple functions, one can extend \eqref{eq:W'} to predictable processes (with appropriate integrability conditions).  Consider $\cW'^-_{S'}$ and $\cW'^+_{S'}$ in the notation \eqref{eq:W-} and \eqref{eq:W+}. By \eqref{eq:fW} applied to $\cW'$ and $f(\ell,r)=\ind_{[0,S'_r]}(\ell)$, for any suitable test function $h$,
   \begin{align*}
       \int_{\r_+\times
        \r} h(\ell,r) \cW'^-_{S'}(\d \ell,\d r) 
        &= \int_{\r_+\times \r}  h(\ell,r) \ind_{[0,S'_r]}(\ell) \cW'(\d \ell,\d r) \\
        &=
        \int_{\r_+\times \r}  h(\ell,r) \ind_{[0,S'_r]}(\ell) \cW_1^-(\d \ell,\d r) \\
        &= 
        \int_{\r_+\times \r}  h(\ell,r)  \cW_1^-(\d \ell,\d r)
   \end{align*}

   \noindent where the second equality is \eqref{eq:W'} and the third one comes from \eqref{eq:fW} applied to $\cW_1$ and $f(\ell,r)=\ind_{[0,S'_r]}(\ell)$. Hence, $\cW'^-_{S'} = \cW_1^-$. Similarly, by \eqref{eq:thetaSW} applied to $\theta_{S'}$ and $\cW'$, then \eqref{eq:W'},  for any suitable test function $h$,
   \begin{align*}
       \int_{\r_+\times
        \r} h(\ell,r) \cW'^+_{S'}(\d \ell,\d r) 
        &= \int_{\r_+\times \r}   \ind_{[S'_r,\infty)}(\ell) h(\ell-S'_r,r) \cW'(\d \ell,\d r) \\
        &=
        \int_{\r_+\times \r}  h(\ell,r)  \cW_2(\d \ell,\d r). 
   \end{align*}
     
\noindent We used that the integral with respect to $\cW_1^-$ vanishes in \eqref{eq:W'} by \eqref{eq:fW}.   Thus $\cW'^+_{S'} = \cW_2$.  We proved
\begin{equation}\label{eq:W'S}
    (\cW_1^-,\cW_2) = (\cW'^-_{S'},\cW'^+_{S'}).
\end{equation}

\noindent In view of \eqref{eq:fW} applied to $\cW_1$ and $f(\ell,r)=h(\ell,r)=\ind_{[0,\cS'_r]}(\ell)$, 
\[
 \int_{r_0}^y  \cW_1^-([0,S'_r],\d r) = \int_{r_0}^y  \cW_1([0,S'_r],\d r)
\]

\noindent hence the process $S'$ is also driven by $\cW_1^-$. Observe that 
\[
\int_{r_0}^y  \cW_1^-([0,S'_r],\d r) = \int_{r_0}^y  \cW_1^-(\r_+,\d r)
\]

\noindent hence $S'$ is measurable with respect to $\cW_1^-$. Applying \eqref{eq:W'} to $g(\ell,r):=\ind_{[0,\cS_r]}(\ell)$, we conclude that $S'$ is the BESQ$(\overline{\delta})$ flow line driven by $\cW'$ starting at $(Z',r_0)$. It proves that $(\cW',S')$ has the same distribution as $(\cW,S)$, and from \eqref{eq:W'S} we deduce that $(\cW_1^-,\cW_2)$ is distributed as $(\cW_\cS^-,\cW_\cS^+)$.  The independence between $\cW^-_S$ and $\cW^+_S$ is then a consequence of the independence between $\cW_1$ and $\cW_2$. Since $S$ is driven by $\cW^-_S$ for the same reason $S'$ is driven by $\cW_1^-$, it follows that $\cW^+_S$ is independent of $S$.
\end{proof}

\bigskip


\begin{proposition}\label{p:difference}
    In the setting of Definition \ref{def:vary}, let $\overline{\delta_0}, \overline{\delta_1}$, $\overline{\delta_2}$ be deterministic drift functions such that $\overline{\delta_0}(x)>0$ on $\r$.  Let $r_0\in \r$ be a constant and $r_1\ge r_0$ be an $\cF$-stopping time. We write $Y=(Y_x,\, x\ge r_0)$ for the ${\rm BESQ}(\overline{\delta_0}\,|_{r_1}\, \overline{\delta_1})$ flow line  starting at $(0,r_0)$ and we let $\cS$ be the ${\rm BESQ}(\overline{\delta_2})$ flow, both driven by $\cW$. Let $\cW^-_Y$ and $\cW^+_Y$ defined as in \eqref{eq:W-} and \eqref{eq:W+} from $(\cW,Y)$ in place of $(\cW,S)$. We extend $Y$ to $\r$ by setting $Y_x=0$ when $x<r_0$.  Consider the collection of processes 
 \begin{align*}
\cS^+_{r,x}(a) &:=\cS_{r,x}(a+ Y_r)- Y_x, &&  a\ge 0,\,  t(Y)>x\ge r, \\
\cS^-_{r,x}(a) &:=  \cS_{r,x}(a), && a\le Y_r,\, t(Y)>x\ge r\ge r_0
\end{align*}

\noindent where $t(Y):=\inf\{x\ge r_1\,:\, Y_x=0 \hbox{ and } \overline{\delta_1}(x)< 0\}$. We impose that 
\begin{equation}\label{eq:killing S-+}
    \cS^+_{r,x}(a)=0  \mbox{ \rm if } x\ge \varphi_+(a,r), \qquad \cS^-_{r,x}(a)=Y_x \mbox{ \rm if } x\ge \varphi_-(a,r),
\end{equation}

\noindent where, for any $a$ and $r$, 
\begin{align}\label{not:phi+}
 \varphi_+(a,r) & :=\inf\{y> r\,:\, \cS_{r,y}(a+ Y_r)\le Y_y  
\hbox{ and }  \overline{\delta_2}(y)\le  (\overline{\delta_0}\,|_{r_1}\, \overline{\delta_1})(y)\},\\
\varphi_-(a,r) & :=\inf\{y> r\,:\, \cS_{r,y}(a)\ge Y_y\hbox{ and } \overline{\delta_2}(y) \ge (\overline{\delta_0}\,|_{r_1}\, \overline{\delta_1})(y)\}. \nonumber 
\end{align}

\noindent See Figure \ref{fig:decomp}. Then
\begin{enumerate}[(i)]
\item  $\cS^+$ is a ${\rm BESQ}(\overline{\delta_2}\,|_{r_0}\,\overline{\delta_2} - \overline{\delta_0}\,|_{r_1}\, \overline{\delta_2} -\overline{\delta_1})$ flow driven by the white noise $\cW^+_Y$ killed on $\{\overline{\delta_2} \le (\overline{\delta_0}\,|_{r_1}\, \overline{\delta_1})\}$ and restricted to $(-\infty,t(Y))$.
\item  The noise  $\cW$ driving the SDEs of $Y$ and $\cS^-$ can be replaced with $\cW^-_Y$. In particular, $Y$ and $\cS^-$ are measurable with respect to $\cW^-_Y$ and $r_1$.
\item If $Y$ is a $\ominus$-flow line, then $\cW^+_Y$ is a white noise on $\r_+ \times \r$  independent of $\cW^-_Y$. 
\item Let $b\in \r$ and $\overline{\delta'}:= (\overline{\delta_0} \,|_b\, \overline{\delta_1})$. If $Y$ is a $\oplus$-flow line, then conditionally on 
$r_1=b$, $\cW^-_Y$ is independent of $\cW^+_Y$ and has the law of $\cW^-$ associated by \eqref{eq:W-} to  the ${\rm BESQ}(\overline{\delta'})$ flow line starting at $(0,r_0)$. 
\end{enumerate}
\end{proposition}

\begin{figure}[htbp]
\centering
   \scalebox{1}{ 
        \def\svgwidth{0.4\columnwidth}
\begingroup%
  \makeatletter%
  \providecommand\color[2][]{%
    \errmessage{(Inkscape) Color is used for the text in Inkscape, but the package 'color.sty' is not loaded}%
    \renewcommand\color[2][]{}%
  }%
  \providecommand\transparent[1]{%
    \errmessage{(Inkscape) Transparency is used (non-zero) for the text in Inkscape, but the package 'transparent.sty' is not loaded}%
    \renewcommand\transparent[1]{}%
  }%
  \providecommand\rotatebox[2]{#2}%
  \newcommand*\fsize{\dimexpr\f@size pt\relax}%
  \newcommand*\lineheight[1]{\fontsize{\fsize}{#1\fsize}\selectfont}%
  \ifx\svgwidth\undefined%
    \setlength{\unitlength}{369.3364563bp}%
    \ifx\svgscale\undefined%
      \relax%
    \else%
      \setlength{\unitlength}{\unitlength * \real{\svgscale}}%
    \fi%
  \else%
    \setlength{\unitlength}{\svgwidth}%
  \fi%
  \global\let\svgwidth\undefined%
  \global\let\svgscale\undefined%
  \makeatother%
  \begin{picture}(1,0.91148365)%
    \lineheight{1}%
    \setlength\tabcolsep{0pt}%
    \put(0,0){\includegraphics[width=\unitlength,page=1]{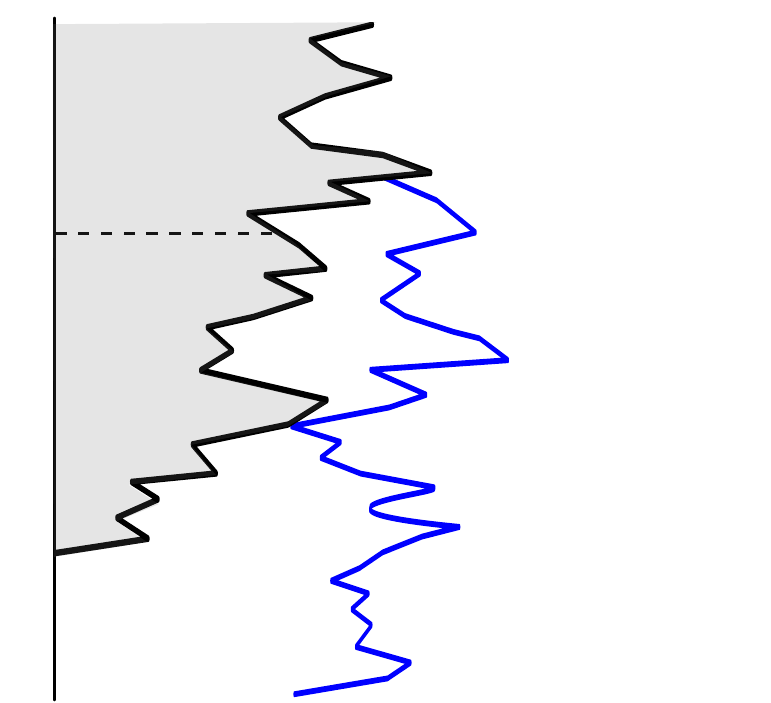}}%
    \put(0.50089546,0.87817402){\color[rgb]{0.10196078,0.10196078,0.10196078}\transparent{0.82599998}\makebox(0,0)[lt]{\lineheight{1.25}\smash{\begin{tabular}[t]{l}$Y$\end{tabular}}}}%
    \put(0.136147,0.67240905){\color[rgb]{0,0,0}\transparent{0.82599998}\makebox(0,0)[lt]{\lineheight{1.25}\smash{\begin{tabular}[t]{l}$\cW_Y^-$\end{tabular}}}}%
    \put(-0.00006943,0.18123615){\color[rgb]{0,0,0}\transparent{0.82599998}\makebox(0,0)[lt]{\lineheight{1.25}\smash{\begin{tabular}[t]{l}$r_0$\end{tabular}}}}%
    \put(-0.00006943,0.5895304){\color[rgb]{0,0,0}\transparent{0.82599998}\makebox(0,0)[lt]{\lineheight{1.25}\smash{\begin{tabular}[t]{l}$r_1$\end{tabular}}}}%
    \put(0.78687073,0.79747316){\color[rgb]{0,0,0}\transparent{0.82599998}\makebox(0,0)[lt]{\lineheight{1.25}\smash{\begin{tabular}[t]{l}$\cW_Y^+$\end{tabular}}}}%
    \put(0,0){\includegraphics[width=\unitlength,page=2]{decomposition.pdf}}%
  \end{picture}%
\endgroup%

    }
    \caption{The black line represents a ${\rm BESQ}(\overline{\delta_0}\,|_{r_1}\, \overline{\delta_1})$ flow line $Y=(Y_x,\, x\ge r_0)$. The blue lines represent the ${\rm BESQ}(\overline{\delta_2})$ flow $\cS$, both driven by $\cW$. The shaded gray area represents $\cW_Y^-$, while the white area indicates $\cW_Y^+$.}
    \label{fig:decomp}
\end{figure}

\bigskip

\begin{proof}
\begin{enumerate}[(i)]
\item 
Recall that $Y$ is a ${\rm BESQ}(\overline{\delta_0}\,|_{r_1}\, \overline{\delta_1})$ flow line starting at $(0,r_0)$ and $\overline{\delta_0}>0$. Observe that for all $x\in (r_1,t(Y))$, if $\overline{\delta_1}(x)<0$, then $Y_x>0$. By \eqref{eq:BESQflow vary}, we have for any $x \in (r_0,t(Y))$,
\[
\dd Y_x = 2 \, \cW([0, Y_x],\dd x) +   (\overline{\delta_0}\,|_{r_1}\, \overline{\delta_1})(x) \dd x.
\]

\noindent Since $Y_r=0$ when $r\le r_0$, we can write for any $x<t(Y)$
\[
\dd Y_x = 2 \, \cW([0, Y_x],\dd x) +   (0\,|_{r_0}\,\overline{\delta_0}\,|_{r_1}\, \overline{\delta_1})(x) \dd x.
\]

\noindent By Proposition \ref{p:comparison}, a.s. for all $r\le x\le \varphi_+(a,r)$ and $a\ge 0$, $\cS_{r,x}(a+Y_r)\ge Y_x$. Otherwise, let $x\in [r,\varphi_+(a,r))$  such that $\cS_{r,x}(a+Y_r)= Y_x$ and $\cS_{r,y}(a+Y_r)< Y_y$ 
for $y>x$ close enough to $x$. Since $x<\varphi_+(a,r)$, $\overline{\delta_2}(y) > (\overline{\delta_0}\,|_{r_1}\, \overline{\delta_1})(y)$. 
By the perfect flow property of $\cS$ and of the ${\rm BESQ}(\overline{\delta_0}\,|_{r_1} \, \overline{\delta_1})$ flow $\cS'$, we have $\cS_{x,y}(Y_x)<\cS'_{x,y}(Y_x)$, 
which contradicts the comparison principle at the point $(Y_x,x)$ between $\cS$ and $\cS'$.  Let $r\le x<\varphi_+(a,r)$ such that $\overline{\delta}_2(x)< 0$. If $x\le r_1$, then $\cS_{r,x}(a+Y_r)>Y_x\ge 0$ by definition \eqref{not:phi+} and the fact that $\overline{\delta_0}>0$. If $x>r_1$ and $\overline{\delta_1}(x)\ge 0$, then $\cS_{r,x}(a+Y_r)>Y_x\ge 0$ for the similar reason.  If $x \in (r_1,t(Y))$ and $\overline{\delta_1}(x)<0$, then $Y_x>0$ by definition of $t(Y)$, hence $\cS_{r,x}(a+Y_r)\ge Y_x>0$. We proved that a.s. for any $r\le x<\min(t(Y),\varphi_+(a,r))$ such that $\overline{\delta}_2(x)< 0$, $\cS_{r,x}(a+Y_r)>0$ and by \eqref{eq:BESQflow vary},   
\[
\dd_x \cS_{r,x}(a+Y_r) = 2 \, \cW([0, \cS_{r,x}(a+Y_r)],\dd x) + \overline{\delta_2}(x)\dd x.
\]

\noindent We deduce that for any $r\le x<\min(t(Y),\varphi_+(a,r))$,
\begin{align*}
    \dd_x \, \cS^+_{r,x}(a) &= \dd_x \, (\cS_{r,x}(a+Y_r) - Y_x)\\
    &=  2 \, \cW([Y_x, \cS_{r,x}(a+Y_r)],\dd x) +   (\overline{\delta_2}\,|_{r_0}\, \overline{\delta_2}-\overline{\delta_0}\,|_{r_1}\, \overline{\delta_2}-\overline{\delta_1})(x) \dd x\\
    &= 2 \cW^+_Y([0,\cS^+_{r,x}(a)],\dd x) + (\overline{\delta_2}\,|_{r_0}\, \overline{\delta_2}-\overline{\delta_0}\,|_{r_1}\, \overline{\delta_2}-\overline{\delta_1})(x) \dd x
\end{align*}
by \eqref{eq:thetaSW}. By construction,  $\cS^+_{r,\cdot}(a)$ is absorbed at $0$ if it touches $0$ on $\{\overline{\delta_2} \le (\overline{\delta_0}\,|_{r_1}\, \overline{\delta_1})\}$, i.e. at time $\varphi_+(a,r)$. Therefore, up to the regularity conditions of Lemma \ref{l:killed},  $\cS^+$ is a ${\rm BESQ}(\overline{\delta_2}\,|_{r_0}\,\overline{\delta_2} - \overline{\delta_0}\,|_{r_1}\, \overline{\delta_1})$ flow on $(-\infty,t(Y))$ driven by the white noise $\cW^+_Y$ and killed on $\{\overline{\delta_2} \le (\overline{\delta_0}\,|_{r_1}\, \overline{\delta_1})\}$. We check that $\cS^+$ satisfies the regularity conditions of Lemma \ref{l:killed}. Statement (i) says that $\cS^+_{r,r}(a)=a$ which is true. Let $r< x<t(Y)$ and $a'\ge a\ge 0$ such that $x< \varphi_+(a,r)$. Then  $\cS^+_{r,x}(a)=\cS_{r,x}(a+Y_r)-Y_x$ and $\cS^+_{r,x}(a')=\cS_{r,x}(a'+Y_r)-Y_x$. We deduce the c\`adl\`ag property (ii). Moreover, if $r \le x\le y< \min(\varphi_+(a,r),t(Y))$, 
then by the perfect flow property of $\cS$, $y<\varphi_+( \cS_{r,x}(a+Y_r), x)$ and  $\cS_{r,y}^+(a)=\cS_{r,y}(a+Y_r)-Y_y=\cS_{x,y}\circ \cS_{r,x}(a+Y_r)-Y_y= \cS_{x,y}^+ ( \cS_{r,x}(a+Y_r)-Y_x)= \cS_{x,y}^+\circ \cS_{r,x}^+(a)$. That proves the perfect flow property (iii). Together with the absorption at time $\varphi_+(a,r)$, it yields that the regularity conditions of Lemma \ref{l:killed} are satisfied. 

\item  By  the same reasoning as above, the comparison principle (Proposition \ref{p:comparison}) implies $\cS^-_{r,x}(a) \le Y_x$ for all $r_0\le r\le x<t(Y)$ and $a\le Y_r$. Indeed, if it is not the case, we find
$r \in [r_0,t(Y))$, $a\le Y_r$ and $x \in [r,\varphi_-(a,r))$ such that $\cS_{r,x}(a) = Y_x$ and $\cS_{r,y}(a) > Y_y$ for $y>x$ close enough to $x$. Then, by the definition of $\varphi_-$, $\overline{\delta_2}(y) < (\overline{\delta_0}\,|_{r_1}\, \overline{\delta_1})(y)$ which is in contradiction with $\cS_{x,y}(Y_x)>\cS'_{x,y}(Y_x)$. 
Statement (ii) is then a consequence of \eqref{eq:fW} with $f(\ell,r)=\ind_{[0,Y_r]}(\ell)$ and $h(\ell,r)=\ind_{[x,y]}(r)$ (in the case of $Y$) or $h(\ell,r)=\ind_{[x,y]}(r)\ind_{[0,\cS^-_{x,r}(a)]}(\ell)$  (in the case of $\cS^-$).

\item Recall the definition of a $\ominus$-flow line in Definition \ref{d:minus plus flow line}. We let $S$ be the ${\rm BESQ}(\overline{\delta_0})$ flow line starting at $(0,r_0)$ and define $\cW^-_S$ and $\cW^+_S$ via \eqref{eq:W-} and \eqref{eq:W+}.  We introduce the following martingale measures ($A\subset \r_+$ is an arbitrary Borel set with finite Lebesgue measure and $y\ge x\ge 0$).
\begin{itemize}
\item  $\cW_1$ is  $\cW_Y^-$ restricted to $(-\infty,r_1]$. Equivalently, $\cW_1$ is $\cW_\cS^-$ restricted to $(-\infty,r_1]$.
\item $\cW_2(A\times [-y,-x]):=\cW_S^+(A\times [r_1-y,r_1-x])=\cW_Y^+(A\times [r_1-y,r_1-x])$.
\item $\cW_3(A\times [x,y]):=\cW_Y^-(A\times[r_1+x,r_1+y])$.
\item $\cW_4(A\times [x,y]):= \cW^+_Y(A\times [r_1+x,r_1+y])$.
\end{itemize}

\begin{figure}[htbp]
\centering
   \scalebox{1}{ 
        \def\svgwidth{0.4\columnwidth}
        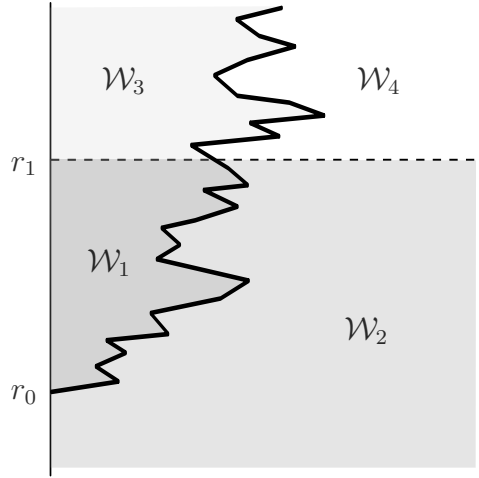
    }
    \caption{The regions shaded in colors from darkest to lightest represent the areas of $\cW_1$, $\cW_2$, $\cW_3$, and $\cW_4$, respectively.}
    \label{fig:decomp_proof}
\end{figure}

See Figure \ref{fig:decomp_proof}. We observe that $\cW^+_Y$ is measurable with respect to  $(\cW_2,\cW_4,r_1)$ while $\cW^-_Y$  is measurable with respect to $(\cW_1,\cW_3)$.  We will prove that:
\begin{enumerate}[(a)]
\item $\cW_2$ is a white noise independent of $(r_1,\cW_1)$.
\item Conditionally on $\cW_1$, the martingale measures $\cW_2$ and $\cW_3$ are independent.
\item $\cW_4$ is a white noise independent of $(r_1,\cW_1,\cW_2,\cW_3)$.
\end{enumerate}

\noindent It will imply that $(\cW_2,\cW_4)$ is a pair of independent white noises, independent of $(r_1,\cW_1,\cW_3)$ hence of $\cW_Y^-$. Since $\cW^+_Y$ is obtained by concatenating $\cW_2$ and $\cW_4$ at level $r_1$, it is still a white noise independent of $\cW_Y^-$.  
It remains to prove (a), (b) and (c).

Statement (a) is a consequence of Proposition \ref{p:independence_W-&W^+} and the fact that $r_1$ is measurable with respect to $\cW_S^-$, hence is independent of $\cW_S^+$. 

We introduce the martingale measure $\widetilde{\cW}$ defined via $\widetilde{\cW}(A\times[x,y])=\cW(A\times[r_1+x,r_1+y])$, and let ${\widetilde Y}_r:= Y_{r_1+r}$ for $0\le r<t(Y)-r_1$. Then $\widetilde{\cW}$ is  a white noise independent of $(\cW_1,\cW_2)$ and ${\widetilde Y}$ is a ${\rm BESQ}(\overline{\delta_1})$ flow line driven by ${\widetilde \cW}$ starting from $(Y_{r_1},0)$. Moreover, $\cW_3$ is equal to ${\widetilde \cW}_{\widetilde Y}^-$ in the notation \eqref{eq:W-}, hence is measurable with respect to $({\widetilde \cW},Y_{r_1})$. Since $Y_{r_1}$ is measurable with respect to $\cW_1$ (by (ii)), we deduce statement (b).

Finally, statement (c) is a consequence of Proposition \ref{p:independence_W-&W^+} applied to $(\widetilde {\cW},\widetilde{Y})$ in place of $(\cW,S)$ and the fact that $r_1$ is a stopping time with respect to $\cF$. It shows that, conditionally on $(\cW_1,\cW_2)$, $\cW_4$ is a white noise independent of $\cW_3$.

\item We keep the notation of the proof of (iii). We need to prove that conditionally on $r_1=b$, the pair $(\cW_1,\cW_3)$ is independent of $(\cW_2,\cW_4)$, that $\cW_1$ has the law of $\cW_S^-$ restricted to $(-\infty,b]$, and conditioning further on $\cW_1$ and $Y_{r_1}=y_0$, the martingale measure $\cW_3$ has the law of $\cW_{\widehat Y}^-$ where ${\widehat Y}$ is the ${\rm BESQ}(\overline{\delta_1})$ flow line starting at $(y_0,0)$. 
Recall that $Y_{r_1}$ is measurable with respect to $\cW_1$ by (ii). Therefore it is enough to prove that conditionally on $r_1=b$:
\begin{enumerate}[(a')]
\item $\cW_1$ is independent of $\cW_2$ and has the law of $\cW_S^-$ restricted to $(-\infty,b]$.
\item Conditionally on $Y_{r_1}=y_0$, $\cW_3$ is  independent of $(\cW_1,\cW_2)$ and has the law of $\cW_{\widehat Y}^-$.
\item $\cW_4$ is a white noise independent of $(\cW_1,\cW_2,\cW_3)$.
\end{enumerate}

We have already seen that statement (c') is a consequence of Proposition \ref{p:independence_W-&W^+}, so let us prove (a') and (b'). Statement (a') comes from Proposition \ref{p:independence_W-&W^+} applied to $(\cW,S)$, and the fact that $r_1$ is measurable with respect to $\cW_S^+$, hence is independent of $\cW_S^-$. Finally, we recall that we can write $\cW_3$ as $\widetilde{\cW}^-_{\widetilde Y}$ where ${\widetilde Y}$ is the ${\rm BESQ}(\overline{\delta_1})$ flow line starting at $(Y_{r_1},0)$ driven by ${\widetilde \cW}$, that $\widetilde{\cW}$ is independent of $(\cW_1,\cW_2)$ and that  $Y_{r_1}$ is measurable with respect to $\cW_1$. It yields (b').
\end{enumerate}
\end{proof}

\section{Meeting of flow lines}\label{s:meeting}
This section  studies the interaction between flow lines of different parameters. Recall from the introduction that ${\mathcal B}(a,b)$ denotes the beta($a,b$) distribution. Theorems \ref{thm: main left}, \ref{thm: main right}, \ref{thm: main dual} are particular cases of Theorems \ref{thm:U_process}, \ref{thm:U_process-}, \ref{thm:U_process*} respectively.


\subsection{Meeting of a forward flow line from the left}\label{s:meetingL}

Fix $z\ge 0$, $\delta > 0$ and $\deltahat <\delta+2$. Let $\cS$ be the ${\rm BESQ}^\delta$ flow and $Y^0=(Y^0_x)_{x\ge 0}$ be the ${\rm BESQ}(\delta \,|_z\, \deltahat)$ flow line starting at $(0,0)$ both driven by $\cW$ . For $r\ge 0$, let 
\begin{equation}\label{def:U}
    U(r):= \inf\{x\ge \max(r,z)\,:\, \cS_{r,x}(0)= Y^0_x \}.
\end{equation}

\noindent See Figure \ref{fig:meeting1}. We observe that $(U(r),\,r\ge 0)$ is non-decreasing. It is also left-continuous as can be seen for example from Proposition \ref{p:properties} (i) and equation \eqref{eq:approx-AHS}. If $\deltahat \le 0$, since $Y^0$ is absorbed when hitting $0$ after $z$, $U(r)=r$ for all $r$ greater than this hitting time. In the general case, using property ({\bf P}) of Section \ref{s:def} and that $\deltahat- \delta<2$, we see that $U(r)<\infty$ a.s. Moreover, $U(r)$ is a stopping time for the natural filtration of $\cW$. For $r\ge 0$, we define the process $Y^r$ by 
\begin{align}\label{def:Yx}
    Y_{x}^r:=\begin{cases}
    \cS_{r,x}(0) & \text{if } x\in[r,U(r)], \\
    Y^0_x & \text{if } x>U(r).
\end{cases}
\end{align}

\begin{figure}[htbp]
\centering
   \scalebox{1}{ 
        \def\svgwidth{0.4\columnwidth}
\begingroup%
  \makeatletter%
  \providecommand\color[2][]{%
    \errmessage{(Inkscape) Color is used for the text in Inkscape, but the package 'color.sty' is not loaded}%
    \renewcommand\color[2][]{}%
  }%
  \providecommand\transparent[1]{%
    \errmessage{(Inkscape) Transparency is used (non-zero) for the text in Inkscape, but the package 'transparent.sty' is not loaded}%
    \renewcommand\transparent[1]{}%
  }%
  \providecommand\rotatebox[2]{#2}%
  \newcommand*\fsize{\dimexpr\f@size pt\relax}%
  \newcommand*\lineheight[1]{\fontsize{\fsize}{#1\fsize}\selectfont}%
  \ifx\svgwidth\undefined%
    \setlength{\unitlength}{190.49326258bp}%
    \ifx\svgscale\undefined%
      \relax%
    \else%
      \setlength{\unitlength}{\unitlength * \real{\svgscale}}%
    \fi%
  \else%
    \setlength{\unitlength}{\svgwidth}%
  \fi%
  \global\let\svgwidth\undefined%
  \global\let\svgscale\undefined%
  \makeatother%
  \begin{picture}(1,1.05723665)%
    \lineheight{1}%
    \setlength\tabcolsep{0pt}%
    \put(0,0){\includegraphics[width=\unitlength,page=1]{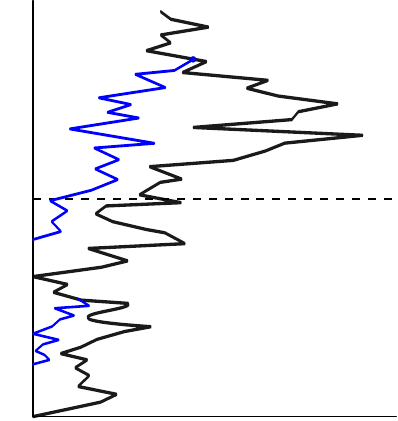}}%
    \put(0.02384631,0.54686684){\color[rgb]{0,0,0}\makebox(0,0)[lt]{\lineheight{1.25}\smash{\begin{tabular}[t]{l}$z$\end{tabular}}}}%
    \put(0.02407601,0.00454664){\color[rgb]{0,0,0}\makebox(0,0)[lt]{\lineheight{1.25}\smash{\begin{tabular}[t]{l}$0$\end{tabular}}}}%
    \put(0.01316879,0.12808214){\color[rgb]{0,0,1}\makebox(0,0)[lt]{\lineheight{1.25}\smash{\begin{tabular}[t]{l}$r_1$\end{tabular}}}}%
    \put(0,0){\includegraphics[width=\unitlength,page=2]{meeting_1.pdf}}%
    \put(0.56252912,0.59184917){\color[rgb]{0,0,1}\makebox(0,0)[lt]{\lineheight{1.25}\smash{\begin{tabular}[t]{l}$U(r_1)$\end{tabular}}}}%
    \put(0.01316879,0.43964279){\color[rgb]{0,0,1}\makebox(0,0)[lt]{\lineheight{1.25}\smash{\begin{tabular}[t]{l}$r_2$\end{tabular}}}}%
    \put(0,0){\includegraphics[width=\unitlength,page=3]{meeting_1.pdf}}%
    \put(0.62414789,0.94643583){\color[rgb]{0,0,1}\makebox(0,0)[lt]{\lineheight{1.25}\smash{\begin{tabular}[t]{l}$U(r_2)$\end{tabular}}}}%
    \put(0,0){\includegraphics[width=\unitlength,page=4]{meeting_1.pdf}}%
  \end{picture}%
\endgroup%

    }
    \caption{The blue lines represent the ${\rm BESQ}^\delta$ flow lines starting from different points $(0, r_i)$, $i=1,2$. The black line represents $Y^0$, the ${\rm BESQ}(\delta \,|_z\, \widehat{\delta})$ flow line starting from $(0, 0)$. The levels $U(r_i)$, $i=1,2$, are the meeting levels as defined in \eqref{def:U}.}
    \label{fig:meeting1}
\end{figure}

\noindent  The following proposition shows that it is a $\oplus$-flow line in the terminology of Section \ref{s:decomposition}.

\begin{proposition}\label{boundary}
For every fixed $r\ge 0$, $Y^r$ is a $\oplus$-flow line. 
\end{proposition}
\begin{proof}
We adopt the framework of Section \ref{s:decomposition} with $r_0=r$, $S_x=\cS_{r,x}(0)$, $Y_x=Y^r_x$, $r_1=U(r)$. We need to show that $U(r)$ is a stopping time with respect to the filtration of $\cW^+_S$ defined in \eqref{eq:W+}. Let $\widetilde{\cS}$ be  the ${\rm BESQ}(\delta \,|_z\, \deltahat)$ flow driven by $\cW$. By Proposition \ref{p:difference} (i) applied to $(S, \widetilde{\cS}, \delta, \delta\,|_z\,\deltahat)$ in place of $(Y,\cS,\overline{\delta_0}\,|_{r_1} \overline{\delta_1},\overline{\delta_2})$ there, the process $\tcS_{0,x}^+(0):=\widetilde{\cS}_{0,x}(0)-S_x=Y^0_x-\cS_{r,x}(0)$  is driven by $\cW_S^+$. In particular, $U(r)=\inf\{x\ge \max(r,z)\,:\, \tcS_{0,x}^+(0)=0 \}$ is a stopping time with respect to $\cW^+_S$. 
\end{proof}

Observe that $Y^r \ge Y^y$ if $r\le y$. Consider $\cW^+_{Y^r}$ in the notation \eqref{eq:W+}, and let $\mathscr{W}^+_r:=\sigma(\cW^+_{Y^r})$. It defines a filtration. 

\begin{proposition}\label{p:Ux adapted}
    The process $(U(r))_{r\ge 0}$ is adapted to the filtration $\mathscr{W}^+$.
\end{proposition}
\begin{proof}
    Fix $r\ge 0$. Let as above $\tcS$ be  the ${\rm BESQ}(\delta \,|_z\, \deltahat)$ flow driven by $\cW$. Proposition \ref{p:difference}  (i) applied to $(Y^r,\tcS, r,U(r), \delta , \deltahat,\delta\,|_z \deltahat )$ in place of $(Y,\cS, r_0, r_1, \overline{\delta_0},\overline{\delta_1},\overline{\delta_2})$ implies that,  in similar notation to the proposition,  $\tcS^+$ is driven by $\cW_{Y^r}^+$ up to $t(Y)$ and killed on $\{ \delta\,|_z \deltahat  \le \delta\,|_{U(r)}\, \deltahat \}$. Notice that if $\tcS_{0,x}^+(0)=0$ for some $x< z$, then $\tcS_{0,z}^+(0)=0$ and that $U(r)=\inf\{x\ge \max(r,z)\,:\, \tcS_{0,x}^+(0)=0 \}$. Hence for any $y\ge \max(r,z)$, the event $\{U(r)>y\}$ is the event that the flow line driven by $\cW_{Y^r}^+$ starting from $(0,0)$ has not touched $0$ on the interval $[0,y]$. It is therefore measurable with respect to $\cW_{Y^r}^+$ (it is even a stopping time). 
    \end{proof}


Let $P^z$ be the distribution of $\left(U(r)-r\right)_{r\ge 0}$.
\begin{theorem}\label{thm:U_process}
    \begin{enumerate}[(i)]
\item Suppose $z=0$ and $\deltahat>0$. For all $r>0$, $\frac{r}{U(r)}$ has distribution $\mathcal{B}(\frac{2-\widehat \delta+\delta}{2},\frac{\widehat \delta}{2})$. 
\item The family $(P^z,z\ge 0)$ defines a time-homogeneous Feller process. For each $z\ge 0$, the process  $\left(U(r)-r\right)_{r\ge 0}$ is a time-homogeneous $(\mathscr{W}^+_r)_{r\ge 0}$-adapted Feller process starting from $z$. 
\item Suppose $z>0$ and write $\mathrm{x}:=\inf\{r\ge 0\,:\, U(r)>z\}$, $\mathfrak{U}:=\lim_{r\downarrow \mathrm{x}} U(r)$. We have $\frac{\mathrm x}{z}\sim\mathcal{B}(1,\frac{\delta}{2})$ and conditionally on  $\{{\mathrm{x}}=x\}$, $\frac{z-x}{\mathfrak{U}-x}\sim \mathcal{B}(\frac{2-\deltahat+\delta}{2},1)$.
\end{enumerate}
\end{theorem}

\begin{proof}
\begin{enumerate}[(i)]
\item Let $r>0$. By property ({\bf P})  of Section \ref{s:def}, the process
$$
\left\{
\begin{array}{ll}
Y^0_x & \hbox{ if } x\in [0,r],\\
\max(Y^0_x - \cS_{r,x}(0),0) & \hbox{ if } x>r
\end{array}
\right.
$$

\noindent is a ${\rm BESQ}_0(\widehat \delta \,|_r\, \widehat \delta-\delta)$ process. The result follows from Lemma \ref{distribution:hitting0} by taking $\delta_1=\widehat \delta$ and $\delta_2=\widehat \delta-\delta$. 
\item Let $r>0$.  We first show the Markov property of $\left(U(r)-r\right)_{r\ge 0}$. 
Let $\cW^-_{Y^r}$ be defined via \eqref{eq:W-}. By Proposition \ref{p:difference} (ii) (applied to $Y=Y^r$ and $\cS$ the ${\rm BESQ}^\delta$ flow driven by $\cW$), the process $(U(y),\, y\ge r)$ is measurable with respect to $U(r)$ and $\cW^-_{Y^r}$: 
$$
U(y)=\inf\{x\ge \max(y,U(r))\,:\,\cS^-_{y,x}(0)=Y_x^r \}
$$
where we used the notation $\cS^-$ of the proposition. Since $Y^r$ is a $\oplus$-flow line by Proposition \ref{boundary}, Proposition \ref{p:difference} (iv) implies that conditionally on $U(r)=r+z'$, $\cW^-_{Y^r}$ is independent of $\cW^+_{Y^r}$ and has the law of $\cW^-$ associated by \eqref{eq:W-} to the 
${\rm BESQ}(\delta \,|_{r+z'}\, \deltahat)$ flow line, denoted as $Y$, driven by $\cW$ starting at $(0,r)$. 
We deduce that conditionally on $\cW^+_{Y^r}$ and $U(r)=r+z'$, the r.v. $U(y)$ is distributed as
$$
\inf\{x\ge \max(y,r+z')\,:\,\cS_{y,x}(0)=Y_x \}
$$
where $\cS$ is the ${\rm BESQ}^\delta$ flow driven by $\cW$ (by another use of Proposition \ref{p:difference} (ii)).  It completes the proof of the Markov property and we can also verify the time-homogeneity.

We show now that it is a Feller process.  Notice that $U(r)\ge U(0)$ since $r\mapsto U(r)$ is non-decreasing. On the other hand, observe that $Y^0_x\le \widetilde{\cS}_{0,x}(0)$ where $\widetilde{\cS}$ is a ${\rm BESQ}^{\max(\delta,\deltahat)}$ flow driven by $\cW$ by the comparison principle. It yields that $U(r)\le \max(U(0),\widetilde U(r))$ where $\widetilde U(r):=\inf\{x\ge r \,:\, \cS_{r,x}(0)= \widetilde{\cS}_{0,x}(0)\}$. The distribution of $r/\widetilde U(r)$ is given by (i), with $\max(\delta,\deltahat)$ in place of $\deltahat$. The inequality $|U(r)-r-U(0)|\le \widetilde U(r)$ yields the Feller property.

\item  Let $z\ge y>x\ge 0$ and $u\ge z$.  Notice that a.s. for any $r\in (0,z)$, $\{\cS_{r,z}(0)=Y^0_z\} = \{\mathrm{x}\ge r\}$. Moreover
$$
P^z( \mathrm{x} \in [x,y),\, \mathfrak{U} > u)
\le
\p(\cS_{x,z}(0)=Y^0_z,\, {\cS}_{y,u}(0)< Y^0_u).
$$
On the other hand, we have
\begin{align*}
&P^z( \mathrm{x} \in [x,y),\, \mathfrak{U} > u) \\
&\ge P^z( \mathrm{x} \in [x,y),\, \mathfrak{U} > u,\, U \text{ only has one jump in } [x,y)) \\
&\ge \p(\cS_{x,z}(0)=Y^0_z,\, {\cS}_{y,u}(0)< Y^0_u) - P^z(\mathrm{x}\ge x,\, U \text{ jumps at least twice in } [x,y)).
\end{align*}

\noindent  We first prove that
\begin{equation}\label{eq:jump twice}
\lim_{y\downarrow x} \frac{1}{y-x} P^z(\mathrm{x}\ge x,\, U \text{ jumps at least twice in } [x,y)) = 0.
\end{equation}

\noindent Observe that, by Lemma \ref{distribution:hitting0} and property ({\bf P}) of Section \ref{s:def}, 
\begin{align}\label{jump_1}
    P^z({\rm x} \ge x) = \p(\cS_{x,z}(0) = Y^0_z) = \p(T_x<z)= \left( 1- \frac{x}{z}\right)^{\frac{\delta}{2}},
\end{align}
where $T_x$ is the hitting time of $0$ by a ${\rm BESQ}(\delta \, |_x\, 0)$ process. Let  $\rm x_2>\rm x$ be the second jump time of $U$.   Since $z\to P^z({\rm x} \in [0,\varepsilon))$ is decreasing,  the strong Markov property at time $\mathrm{x}$ implies that for every $z>0$, $P^z( {\rm x, x_2}\in [0,\varepsilon))\le P^{z}({\rm x} \in [0,\varepsilon))P^{z-\varepsilon}({\rm x} \in [0,\varepsilon))\le C \eps^2$ for $\varepsilon\in [0,\frac{z}{2}]$ by \eqref{jump_1}. By the Markov property at time $x$, using that $U(x)=z$ a.s. on the event $\{\mathrm{x}\ge x\}$, 
\[
P^z(\mathrm{x}\ge x,\, U \text{ jumps at least twice in } [x,y)) \le C (y-x)^2
\]

\noindent if $y-x\in [0,\frac{z-x}{2}]$. It proves \eqref{eq:jump twice}. Notice that \eqref{jump_1} already gives the distribution of $\rm x$. Let us find the conditional distribution of $\mathfrak{U}$. We deduce from \eqref{eq:jump twice} that 
\begin{align*}
\frac{1}{\dd x}P^z( \mathrm{x} \in \dd x,\, \mathfrak{U} > u)& =\lim_{y\downarrow x}\frac{1}{y-x}
P^z( \mathrm{x} \in [x,y),\, \mathfrak{U} > u) \\
&= \lim_{y\downarrow x}\frac{1}{y-x} \p(\cS_{x,z}(0)=Y^0_z,\, {\cS}_{y,u}(0)< Y^0_u).
\end{align*}

\noindent By property ({\bf P}) of Section \ref{s:def}, we have 
$$
\p(\cS_{x,z}(0)=Y^0_z,\, {\cS}_{y,u}(0)< Y^0_u)
= 
\p(T_x<z)\p(T_{y-x,z-x}>u-x)
$$
with $T_x$ as before and  $T_{a,b}$  the hitting time of $0$ after time $b$ by a ${\rm BESQ}_0(\delta\,|_a\, 0 \,|_{b}\, \deltahat-\delta)$ process. By Lemma \ref{l:conv_hitting0} (i), 
$$
\lim_{y\downarrow x} \frac{1}{y-x}\p(T_{y-x,z-x}>u-x)= \frac{\delta}{2(z-x)} \bigg( \frac{z-x}{u-x} \bigg)^{\frac{2+\delta+\deltahat}{2}}.
$$
Hence by \eqref{jump_1}, we get
$$
\frac{1}{\dd x}P^z( \mathrm{x} \in \dd x,\, \mathfrak{U} > u)
=
\frac{\delta}{2z} \bigg( \frac{z-x}{u-x} \bigg)^{\frac{2+\delta+\deltahat}{2}}  \left( 1- \frac{x}{z}\right)^{\frac{\delta}{2}-1}.
$$ 
It gives the joint distribution of $(\mathrm{x},\mathfrak{U})$.
\end{enumerate}
\end{proof}


\subsection{Meeting of a forward flow line from the right}\label{s:meetingR}

Fix $z\ge 0$, $\delta'>0$, $\delta\in \r$ and $\deltahat >\max(\delta-2,0)$. Let $\cS$ be the ${\rm BESQ}(\delta'\,|_0 \, \delta)$ flow driven by $\cW$, and $Y^0$ be the ${\rm BESQ}(\delta\,|_{z}\,\deltahat)$ flow line driven by $\cW$ starting at $(0,0)$. We use the definition of $U$ in \eqref{def:U} on $\r_-$, i.e. for $r\ge 0$, we let
\begin{equation}\label{def:U-}
    U(-r):= \inf\{x\ge z\,:\, \cS_{-r,x}(0)=Y^0_x \}.
\end{equation}

\noindent See Figure \ref{fig:meeting2}. For $r\ge 0$, $U(-r)$ is finite a.s. by property ({\bf P}) of Section \ref{s:def} and $\delta-\deltahat<2$. Using the analog of the notation  \eqref{def:Yx}, i.e.
\begin{align*}
    Y_{x}^{-r}:=\begin{cases}
    \cS_{-r,x}(0) & \text{if } x\in[-r,U(-r)], \\
    Y^0_x & \text{if } x>U(-r),
\end{cases}
\end{align*}

\noindent the process $Y^{-r}$ is the ${\rm BESQ}(\delta' \,|_0\, \delta\,|_{U(-r)}\, \deltahat )$ flow line starting at $(0,-r)$. The following proposition  is the analog of Proposition \ref{boundary}.

\begin{figure}[htbp]
\centering
   \scalebox{1}{ 
        \def\svgwidth{0.4\columnwidth}
\begingroup%
  \makeatletter%
  \providecommand\color[2][]{%
    \errmessage{(Inkscape) Color is used for the text in Inkscape, but the package 'color.sty' is not loaded}%
    \renewcommand\color[2][]{}%
  }%
  \providecommand\transparent[1]{%
    \errmessage{(Inkscape) Transparency is used (non-zero) for the text in Inkscape, but the package 'transparent.sty' is not loaded}%
    \renewcommand\transparent[1]{}%
  }%
  \providecommand\rotatebox[2]{#2}%
  \newcommand*\fsize{\dimexpr\f@size pt\relax}%
  \newcommand*\lineheight[1]{\fontsize{\fsize}{#1\fsize}\selectfont}%
  \ifx\svgwidth\undefined%
    \setlength{\unitlength}{200.51640548bp}%
    \ifx\svgscale\undefined%
      \relax%
    \else%
      \setlength{\unitlength}{\unitlength * \real{\svgscale}}%
    \fi%
  \else%
    \setlength{\unitlength}{\svgwidth}%
  \fi%
  \global\let\svgwidth\undefined%
  \global\let\svgscale\undefined%
  \makeatother%
  \begin{picture}(1,0.99965271)%
    \lineheight{1}%
    \setlength\tabcolsep{0pt}%
    \put(0,0){\includegraphics[width=\unitlength,page=1]{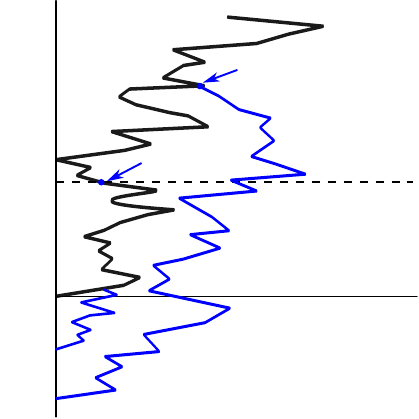}}%
    \put(0.01251049,0.14205297){\color[rgb]{0,0,1}\makebox(0,0)[lt]{\lineheight{1.25}\smash{\begin{tabular}[t]{l}$-r_1$\end{tabular}}}}%
    \put(0.01251049,0.02946883){\color[rgb]{0,0,1}\makebox(0,0)[lt]{\lineheight{1.25}\smash{\begin{tabular}[t]{l}$-r_2$\end{tabular}}}}%
    \put(0.0774345,0.27530866){\color[rgb]{0,0,0}\makebox(0,0)[lt]{\lineheight{1.25}\smash{\begin{tabular}[t]{l}$0$\end{tabular}}}}%
    \put(0.07721637,0.55558941){\color[rgb]{0,0,0}\makebox(0,0)[lt]{\lineheight{1.25}\smash{\begin{tabular}[t]{l}$z$\end{tabular}}}}%
    \put(0.34770818,0.59627018){\color[rgb]{0,0,1}\makebox(0,0)[lt]{\lineheight{1.25}\smash{\begin{tabular}[t]{l}$U(-r_1)$\end{tabular}}}}%
    \put(0.57653973,0.82348206){\color[rgb]{0,0,1}\makebox(0,0)[lt]{\lineheight{1.25}\smash{\begin{tabular}[t]{l}$U(-r_2)$\end{tabular}}}}%
  \end{picture}%
\endgroup%

    }
    \caption{The blue lines represent the ${\rm BESQ}(\delta'\,|_0\,\delta)$ flow lines starting from different points $(0, -r_i)$, $i=1,2$. The black line represents $Y^0$, the ${\rm BESQ}(\delta \,|_z\, \widehat{\delta})$ flow line starting from $(0, 0)$. The levels $U(-r_i)$, $i=1,2$, are the meeting levels as defined in \eqref{def:U-}.}
    \label{fig:meeting2}
\end{figure}

\begin{proposition}\label{boundary2}
For every fixed $r\ge 0$, $Y^{-r}$ is a $\ominus$-flow line.
\end{proposition}
\begin{proof}
We apply the setting of Section \ref{s:decomposition} with $r_0=-r$, $S_x=\cS_{-r,x}(0)$, $Y_x=Y^{-r}_x$, $r_1=U(-r)$ and we need to show that $U(-r)$ is a stopping time with respect to the filtration of $\cW^-_S$ defined in \eqref{eq:W-}. Let $\widetilde{\cS}$ be  the ${\rm BESQ}(\delta \,|_z\, \deltahat)$ flow driven by $\cW$.  By Proposition \ref{p:difference} (ii) applied to $(S,\tcS,\delta'\,|_0\,\delta,\delta\,|_z\, \deltahat)$ in place of  $(Y,\cS,\overline{\delta_0}\,|_{r_1} \overline{\delta_1},\overline{\delta_2})$ there, the processes $\tcS_{0,\cdot}^-(0)$ in the notation of the proposition and  $S_x=\cS_{-r,x}(0)$ are driven by $\cW_S^-$. Notice that $\tcS_{0,x}^-(0)=\widetilde{\cS}_{0,x}(0)=Y^0_x$ for $x\le U(-r)$. It entails that $U(-r)=\inf\{x\ge z\,:\, S_x= \tcS_{0,x}^-(0)\}$ is a stopping time with respect to $\cW^-_S$.
\end{proof}

Since $Y^{-r} \le Y^{-y}$ if $r\le y$, we can define a filtration via $\mathscr{W}^-_r:=\sigma(\cW^-_{Y^{-r}})$ in the notation \eqref{eq:W-}.  

\begin{proposition}\label{p:U-x adapted}
    The process $(U(-r))_{r\ge 0}$ is adapted to the filtration $\mathscr{W}^-$.
\end{proposition}
\begin{proof}
    Fix $r\ge 0$. Let as above $\tcS$ be  the ${\rm BESQ}(\delta \,|_z\, \deltahat)$ flow driven by $\cW$. Proposition \ref{p:difference}  (ii) applied to $(Y^{-r},\tcS,  -r,U(-r), \delta'\,|_0\, \delta ,\deltahat,\delta\,|_z \deltahat )$ in place of $(Y,\cS, r_0, r_1, \overline{\delta_0},\overline{\delta_1},\overline{\delta_2})$ implies that,  in the notation of the proposition,  $\tcS^-$ is driven by $\cW_{Y^{-r}}^-$ as well as $Y^{-r}$. We observe that $U(-r)=\inf\{x\ge z\,:\, \tcS_{0,x}^-(0)=Y^{-r}_x \}$ and if $\tcS_{0,x}^-(0)=Y^{-r}_x$ for some $x<z$, then $\tcS_{0,z}^-(0)=Y^{-r}_z$ by \eqref{eq:killing S-+}. Hence for any $y\ge z$,  the event $\{U(-r)>y\}$ is the event that the flow line driven by $\cW_{Y^{-r}}^-$ starting from $(0,0)$ has not hit $Y^{-r}$ on the interval $[0,y]$. It is therefore a stopping time  with respect to $\cW_{Y^{-r}}^-$, hence is also measurable with respect to it. 
    \end{proof}

The following theorem is the analog of Theorem \ref{thm:U_process}.  Denote by $Q^z$ the law of the process $(U(-r)+r)_{r\ge 0}$. Notice that this process is right-continuous.

\begin{theorem}\label{thm:U_process-}

    \begin{enumerate}[(i)]
\item Suppose $z=0$. For $r>0$, 
$\frac{r}{U(-r)+r}$ has distribution $\mathcal{B}(\frac{2-\delta+\widehat \delta}{2},\frac{\delta'}{2})$. 
\item The family $(Q^z,z\ge 0)$ defines a time-homogeneous Feller process. For each $z\ge 0$, the process $\left(U(-r)+r\right)_{r\ge 0}$ is a time-homogeneous $(\mathscr{W}^-_r)_{r\ge 0}$-adapted Feller process starting from $z$. 
\item Suppose $z>0$ and write $\mathrm{x}:=\inf\{r\ge 0\,:\, U(-r)>z\}$, $\mathfrak{U}:=U(-\mathrm{x})$. 
We have $\frac{z}{z+\mathrm{x}}\sim\mathcal{B}(\frac{\delta'}{2},1)$ and conditionally on  $\{{\mathrm{x}}=x\}$, $\frac{z+x}{\mathfrak{U}+x}\sim \mathcal{B}(\frac{2-\delta+\deltahat}{2},1)$.
\end{enumerate}
\end{theorem}

\begin{proof}
\begin{enumerate}[(i)]
\item Let $r>0$.  By the property ({\bf P}) of Section \ref{s:def}, the  process $$
\left\{
\begin{array}{ll}
\cS_{-r,-r+x}(0) & \hbox{ if } x\in [0,r],\\
\max(\cS_{-r,-r+x}(0) - Y^0_{-r+x},0) & \hbox{ if } x>r
\end{array}
\right.
$$

\noindent is a ${\rm BESQ}_0 (\delta'\,|_r\,\delta-\deltahat )$. Therefore $U(-r)+r$ is the hitting time of $0$ after time $r$ by a ${\rm BESQ}_0 (\delta'\,|_r\,\delta-\deltahat )$ process. The result follows from Lemma \ref{distribution:hitting0} by taking $\delta_1=\delta'$ and $\delta_2=\delta-\widehat \delta$. 

\item The proof follows the same approach of Theorem \ref{thm:U_process} (ii). For $r>0$ and any $y\ge r$, we can rewrite the expression for $U(-y)$:
\begin{equation}\label{eq:U(-y)}
    U(-y)=\inf\{x\ge U(-r)\,:\,\cS^+_{-y,x}(0)=0 \},
\end{equation}
where $\cS^+$ is defined in Proposition \ref{p:difference} (applied to $Y=Y^{-r}$ and $\cS$ the BESQ$(\delta' \, |_0 \, \delta)$ flow driven by $\cW$). Let $\cW^+_{Y^{-r}}$ be defined via \eqref{eq:W+}, then by Proposition \ref{p:difference} (i), $\cS^+$ is a ${\rm BESQ}(\delta'\,|_{-r}\, 0\,|_{U(-r)} \, \delta- \deltahat)$ flow driven by $\cW^+_{Y^{-r}}$ killed on $[-r,U(-r)]$ (and on $[U(-r),+\infty)$ if $\delta\le \deltahat$). Proposition \ref{boundary2} and Proposition \ref{p:difference} (iii) establish that $\cW_{Y^{-r}}^-$ and $\cW^+_{Y^{-r}}$ are independent. Since $U(-r)$ is measurable with respect to $\cW_{Y^{-r}}^-$ by Proposition \ref{p:U-x adapted}, conditionally on $U(-r)=-r+z'$, $\cS^+$ is independent of $\cW^-_{Y^{-r}}$ and is distributed as a ${\rm BESQ}(\delta'\,|_{-r}\, 0\,|_{-r+z'} \, \delta- \deltahat)$ flow killed on $[-r,-r+z']$.  We deduce the Markov property and we can check the time-homogeneity from \eqref{eq:U(-y)}.

We prove now that it is a Feller process as in the proof of Theorem \ref{thm:U_process}.  We have $U(-r)\ge U(0)$ by definition and since $Y^0_x\ge \widetilde{\cS}_{0,x}(0)$ where $\widetilde{\cS}$ is a ${\rm BESQ}^{\min(\delta,\deltahat)}$ flow driven by $\cW$, we  see that $U(-r)\le \max(U(0),\widetilde U(-r))$ where $\widetilde U(-r):=\inf\{x\ge 0 \,:\, \cS_{-r,x}(0)= \widetilde{\cS}_{0,x}(0)\}$. The inequality $|U(-r)+r-U(0)|\le \widetilde U(-r)+r$ yields the Feller property.

\item The proof follows the lines of Theorem \ref{thm:U_process} (iii) so we feel free to skip the details. Let $0\le x<y$ and $u\ge z$. We have 
$$
Q^z( \mathrm{x} \in (x,y],\, \mathfrak{U} > u)
=
\p(\cS_{-x,z}(0)=Y^0_z,\, {\cS}_{-y,u}(0)> Y^0_u) + O((y-x)^2).
$$

\noindent  The property ({\bf P}) of Section \ref{s:def} implies that
$$
\p(\cS_{-x,z}(0)=Y^0_z,\, {\cS}_{-y,u}(0)> Y^0_u)=\p(T_x\le z+x)\p(T_{y-x,x+z}>u+x),
$$
where $T_x$ is the hitting time of $0$ after time $x$ by a ${\rm BESQ}_0(\delta' \, |_x \, 0)$ process and  $T_{a,b}$ is the hitting time of $0$ after time $b$ by a ${\rm BESQ}_{-a}(\delta'\,|_0\, 0 \,|_{b}\, \delta-\deltahat)$ process. By Lemma \ref{l:conv_hitting0} (ii), 
$$
\lim_{y\downarrow x} \frac{1}{y-x}\p(T_{y-x,x+z}>u+x)= \frac{\delta'}{2(x+z)} \bigg( \frac{x+z}{x+u} \bigg)^{\frac{2+\deltahat-\delta}{2}}
$$
and from Lemma \ref{distribution:hitting0}, 
$$
\p( \mathrm{x} \in \dd x,\, \mathfrak{U} \ge u) 
=
\frac{\delta'}{2(x+z)} \bigg( \frac{x+z}{x+u} \bigg)^{\frac{2+\deltahat-\delta}{2}} \bigg(\frac{z}{x+z}\bigg)^{\frac{\delta'}{2}}. 
$$
The proof is complete.
\end{enumerate}
\end{proof}


\subsection{Meeting of a forward and a dual line}\label{s:meeting*}
Let $z\ge 0$, $\delta>0$, $\deltahat>2-\delta$. Let  $\cS$ be a ${\rm BESQ}(\delta\,|_0 \, 0)$ flow driven by $\cW$ and $Y^*$ be the flow line starting from $(0,-z)$ of the ${\rm BESQ}(\delta+\deltahat\,|_0 \, \deltahat)$ flow driven by $-\cW^*$. We let, for $r\ge 0$, 
\begin{equation}\label{def:V}
     V(-r):= \inf\{x\in [-r,z]\,:\, \cS_{-r,x}(0)=Y^{*}_{-x} \}.
\end{equation}

\noindent See Figure \ref{fig:meeting3}. The quantity $V(-r)$ is finite since $\cS_{-r,-r}(0)=0\le Y^{*}_{r}$ and $\cS_{-r,z}(0)\ge 0= Y^{*}_{-z}$.  
Write $\cW^{r,-}$ for the martingale measure $\cW_S^-$ of \eqref{eq:W-} with $S_x=\cS_{-r,x}(0)$. We define the filtration
\begin{equation}
\mathscr{W}^*_r:=\sigma(\cW^{r,-}),\, r\ge 0.\label{filtration*} 
\end{equation}

\begin{figure}[htbp]
\centering
   \scalebox{1}{ 
        \def\svgwidth{0.4\columnwidth}
\begingroup%
  \makeatletter%
  \providecommand\color[2][]{%
    \errmessage{(Inkscape) Color is used for the text in Inkscape, but the package 'color.sty' is not loaded}%
    \renewcommand\color[2][]{}%
  }%
  \providecommand\transparent[1]{%
    \errmessage{(Inkscape) Transparency is used (non-zero) for the text in Inkscape, but the package 'transparent.sty' is not loaded}%
    \renewcommand\transparent[1]{}%
  }%
  \providecommand\rotatebox[2]{#2}%
  \newcommand*\fsize{\dimexpr\f@size pt\relax}%
  \newcommand*\lineheight[1]{\fontsize{\fsize}{#1\fsize}\selectfont}%
  \ifx\svgwidth\undefined%
    \setlength{\unitlength}{246.27121813bp}%
    \ifx\svgscale\undefined%
      \relax%
    \else%
      \setlength{\unitlength}{\unitlength * \real{\svgscale}}%
    \fi%
  \else%
    \setlength{\unitlength}{\svgwidth}%
  \fi%
  \global\let\svgwidth\undefined%
  \global\let\svgscale\undefined%
  \makeatother%
  \begin{picture}(1,0.81392725)%
    \lineheight{1}%
    \setlength\tabcolsep{0pt}%
    \put(0,0){\includegraphics[width=\unitlength,page=1]{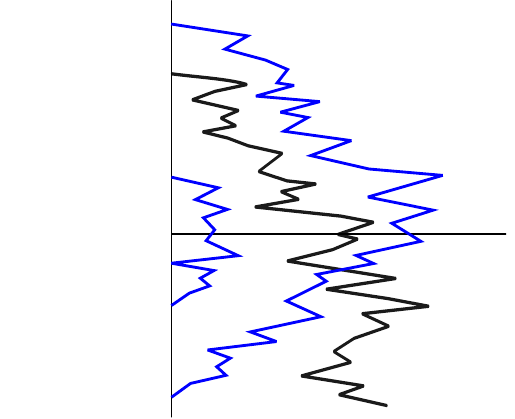}}%
    \put(-0.0141057,0.65999183){\color[rgb]{0,0,1}\makebox(0,0)[lt]{\lineheight{1.25}\smash{\begin{tabular}[t]{l}$V(-r_1)=$\end{tabular}}}}%
    \put(0,0){\includegraphics[width=\unitlength,page=2]{meeting_3.pdf}}%
    \put(0.82441099,0.29104155){\color[rgb]{0,0,1}\makebox(0,0)[lt]{\lineheight{1.25}\smash{\begin{tabular}[t]{l}$V(-r_2)$\end{tabular}}}}%
    \put(0,0){\includegraphics[width=\unitlength,page=3]{meeting_3.pdf}}%
    \put(0.27971381,0.34353887){\color[rgb]{0,0,0}\makebox(0,0)[lt]{\lineheight{1.25}\smash{\begin{tabular}[t]{l}$0$\end{tabular}}}}%
    \put(0.21162593,0.20289601){\color[rgb]{0,0,1}\makebox(0,0)[lt]{\lineheight{1.25}\smash{\begin{tabular}[t]{l}$-r_1$\end{tabular}}}}%
    \put(0.21162593,0.0286982){\color[rgb]{0,0,1}\makebox(0,0)[lt]{\lineheight{1.25}\smash{\begin{tabular}[t]{l}$-r_2$\end{tabular}}}}%
    \put(0.27588168,0.66128155){\color[rgb]{0,0,0}\makebox(0,0)[lt]{\lineheight{1.25}\smash{\begin{tabular}[t]{l}$z$\end{tabular}}}}%
    \put(0,0){\includegraphics[width=\unitlength,page=4]{meeting_3.pdf}}%
  \end{picture}%
\endgroup%

    }
    \caption{The blue lines represent the ${\rm BESQ}(\delta\,|_0\, 0)$ flow lines starting from different points $(0, -r_i)$, $i=1,2$, driven by $\cW$. The black line represents $Y^*$, the ${\rm BESQ}(\delta+\deltahat \,|_0\, \deltahat)$ flow line starting from $(0, -z)$ driven by $-\cW^*$. The levels $V(-r_i)$, $i=1,2$, are the meeting levels as defined in \eqref{def:V}.}
    \label{fig:meeting3}
\end{figure}

\noindent We first prove that the process $(V(-r),\,r\ge 0)$ is adapted to this filtration. Let $\widehat{\cS}$ be the killed $\rm BESQ(2-\deltahat\, |_0\, 2-\delta-\deltahat)$ flow driven by $\cW$. By Proposition \ref{p:BESQdual vary}, its dual $\cShat^*$ is a non-killed $\rm BESQ(\delta+\deltahat\, |_0\, \deltahat)$ flow driven by $-\cW^*$. In particular, by definition of $Y^*$, $Y^*_x=\cShat^*_{-z,x}(0)$.
\begin{proposition}\label{p:V adapted}

(i) Almost surely, for any $r\ge 0$: for any $x \in [V(-r),z)$
    \begin{equation}\label{eq: V upper bound}
        \widehat{\cS}_{x,z}\circ \cS_{-r,x}(0) >0.
    \end{equation}
    while for any $x\in [-r,V(-r))$, 
    \begin{equation}\label{eq: V lower bound}
         \widehat{\cS}_{x,z} \circ\cS_{-r,x}(0)=0.
    \end{equation}
    In particular,  with the convention that $\inf \emptyset=+\infty$, (see Figure \ref{fig:meeting3pr})
    \begin{equation}\label{eq: V second def}
        V(-r)=\inf \{x\in[-r,z): \cShat_{x,z}\circ\cS_{-r,x}(0)>0 \} \land z.
    \end{equation}
    
(ii)    For any $r\ge 0$, $V(-r)$ is measurable with respect to $\rW^*_r$.
\end{proposition}

\begin{proof}
      By Proposition \ref{p:comparison}, $\widehat{\cS}\le \cS$ since $2-\deltahat<\delta$. For any $x\in[-r,z)$, Proposition \ref{p:properties dual} (i) applied to $\widehat{\cS}$, $a=Y^*_{-x}$ and $b=0$ implies that $\widehat{\cS}_{x,z}(Y^*_{-x})>0$ and $\widehat{\cS}_{x,y}(Y^*_{-x})\ge Y^*_{-y}$ for all $y\in [x,z)$. 
    Suppose that $V(-r)<z$. By the perfect flow property of $\cS$, for $x=V(-r)$ and $y\in [V(-r),z)$,
    \[
    \cS_{-r,y}(0)=\cS_{x,y}\circ \cS_{-r,x}(0) = \cS_{x,y}(Y^*_{-x}) \ge \widehat{\cS}_{x,y}(Y^*_{-x})\ge Y^{*}_{-y}.
    \]

    \noindent Hence, for any $y \in [V(-r),z)$, $\widehat{\cS}_{y,z}\circ \cS_{-r,y}(0) \ge \widehat{\cS}_{y,z}(Y^*_{-y})>0$ which yields \eqref{eq: V upper bound}. On the other hand, for any $x\in [-r,V(-r))$, $\cS_{-r,x}(0)<Y^*_{-x}$ hence by Proposition \ref{p:properties dual} (ii) applied to $\widehat{\cS}$, $a'=\cS_{-r,x}(0)$ and $b=0$, $\widehat{\cS}_{x,z} \circ \cS_{-r,x}(0)=0$. It proves \eqref{eq: V lower bound} then \eqref{eq: V second def}. We now use Proposition \ref{p:difference} (ii) with $(\widehat{\cS},\cS_{-r,\cdot}(0))$ in place of $(\cS,Y)$. Since $2-\deltahat<\delta$, $\varphi_-(a,-r)=\infty$ in the notation of the proposition and we can also check that $t(Y)=\infty$, noting that $\overline{\delta_0}\,|_{r_1}\, \overline{\delta_1}=\delta\,|_0\,0$.  We conclude with Proposition \ref{p:difference} (ii). 
    \end{proof}

\begin{figure}[htbp]
\centering
   \scalebox{1}{ 
        \def\svgwidth{0.4\columnwidth}
\begingroup%
  \makeatletter%
  \providecommand\color[2][]{%
    \errmessage{(Inkscape) Color is used for the text in Inkscape, but the package 'color.sty' is not loaded}%
    \renewcommand\color[2][]{}%
  }%
  \providecommand\transparent[1]{%
    \errmessage{(Inkscape) Transparency is used (non-zero) for the text in Inkscape, but the package 'transparent.sty' is not loaded}%
    \renewcommand\transparent[1]{}%
  }%
  \providecommand\rotatebox[2]{#2}%
  \newcommand*\fsize{\dimexpr\f@size pt\relax}%
  \newcommand*\lineheight[1]{\fontsize{\fsize}{#1\fsize}\selectfont}%
  \ifx\svgwidth\undefined%
    \setlength{\unitlength}{178.99732455bp}%
    \ifx\svgscale\undefined%
      \relax%
    \else%
      \setlength{\unitlength}{\unitlength * \real{\svgscale}}%
    \fi%
  \else%
    \setlength{\unitlength}{\svgwidth}%
  \fi%
  \global\let\svgwidth\undefined%
  \global\let\svgscale\undefined%
  \makeatother%
  \begin{picture}(1,1.07793146)%
    \lineheight{1}%
    \setlength\tabcolsep{0pt}%
    \put(0,0){\includegraphics[width=\unitlength,page=1]{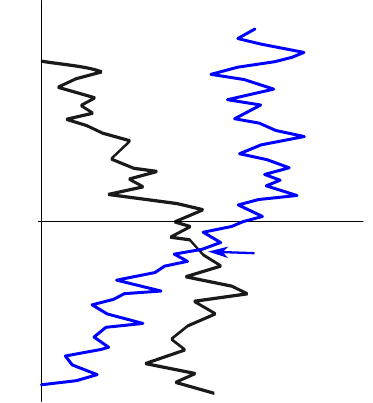}}%
    \put(0.68975694,0.37634625){\color[rgb]{0,0,1}\makebox(0,0)[lt]{\lineheight{1.25}\smash{\begin{tabular}[t]{l}$V(-r)$\end{tabular}}}}%
    \put(0,0){\includegraphics[width=\unitlength,page=2]{meeting_3_proof.pdf}}%
    \put(0.03850527,0.46427358){\color[rgb]{0,0,0}\makebox(0,0)[lt]{\lineheight{1.25}\smash{\begin{tabular}[t]{l}$0$\end{tabular}}}}%
    \put(-0.00298143,0.03212943){\color[rgb]{0,0,1}\makebox(0,0)[lt]{\lineheight{1.25}\smash{\begin{tabular}[t]{l}$-r$\end{tabular}}}}%
    \put(0.03278018,0.90143588){\color[rgb]{0,0,0}\makebox(0,0)[lt]{\lineheight{1.25}\smash{\begin{tabular}[t]{l}$z$\end{tabular}}}}%
    \put(0,0){\includegraphics[width=\unitlength,page=3]{meeting_3_proof.pdf}}%
  \end{picture}%
\endgroup%

    }
    \caption{The blue line represents a ${\rm BESQ}(\delta\,|_0\, 0)$ flow line starting from $(0, -r)$. The black line represents $Y^*$, the ${\rm BESQ}(\delta+\deltahat \,|_0\, \deltahat)$ flow line starting from $(0, -z)$ driven by $-\cW^*$. The brown lines represent $\widehat{\cS}$, the killed $\rm BESQ(2-\deltahat\, |_0\, 2-\delta-\deltahat)$ flow driven by $\cW$. This image illustrates \eqref{eq: V second def}.}
    \label{fig:meeting3pr}
\end{figure}

 \noindent The following theorem characterizes the distribution of the process $(V(-r),\, r\ge 0)$. Observe that the process is right-continuous. We denote by $P^*_z$ the law of the process $(V(-r)+r,\,r\ge 0)$.


\begin{theorem}\label{thm:U_process*}

    \begin{enumerate}[(i)]
\item Suppose $z=0$ and $\deltahat>0$. For any $r> 0$, $\frac{V(-r)+r}{r}$ has distribution $\mathcal{B}(\frac{\deltahat}{2},\frac{\delta}{2})$.
\item The family $(P^z_*,z\ge 0)$ defines a time-homogeneous Feller process. For any $z\ge 0$, the process  $\left(V(-r)+r,\, r\ge 0\right)$ is a homogeneous $(\mathscr{W}^*_r)_{r\ge 0}$-adapted Feller process starting from $z$. 
\item Suppose $z>0$ and write $\mathrm{x}:=\inf\{r\ge 0\,:\, V(-r)<z\}$.  We have $\frac{z}{\mathrm{x}+z}\sim\mathcal{B}(\frac{\delta}{2},1)$ and conditionally on  $\{{\rm{x}}= x \}$, $\frac{ V(-x)+x}{x+z}\sim \mathcal{B}(\frac{\delta+\deltahat}{2},1)$.
\end{enumerate}
\end{theorem}

\begin{proof}
    \begin{enumerate}[(i)]
\item Let $\widehat{\cS}$ denote as before the killed BESQ$(2-\deltahat\, |_0\, 2-\delta-\deltahat)$ flow, driven by $\cW$. By equations \eqref{eq: V upper bound} and \eqref{eq: V lower bound}, for any $x\in [-r,0)$,
\[
\p(V(-r)\le x)=\p(\widehat{\cS}_{x,0}\circ \cS_{-r,x}(0)>0).
\]

\noindent  The result follows from Lemma \ref{distribution:hitting0} with $\delta_1=\delta$ and $\delta_2=2-\deltahat$. 

\item  Let $y>r>0$.
 Necessarily, $V(-y)\le V(-r)$.  Let $a^r:=\cS_{-r,V(-r)}(0)$. We claim that 
\begin{equation}\label{eq:V(-y)}
    V(-y)=\inf\{x\in [-y,V(-r))\,:\, \cShat_{x,V(-r)}\circ \cS_{-y,x}(0)> a^r\}\land V(-r).
\end{equation}

\noindent Let us prove it. Suppose $x<V(-y)$. By \eqref{eq: V lower bound} with $y$ instead of $r$, $\cShat_{x,z}\circ \cS_{-y,x}(0)=0$. Either $\cShat_{x,V(-r)}\circ \cS_{-y,x}(0)=0$, or, by Proposition \ref{p:perfect} (i), $\cShat_{x,V(-r)}\circ \cS_{-y,x}(0)>0$ (in particular $V(-r)<z$) and $\cShat_{V(-r),z}\circ \cShat_{x,V(-r)}\circ \cS_{-y,x}(0)=0$. In the second case, by Proposition \ref{p:properties dual} (i) applied to $\cShat$ between levels $V(-r)$ and $z$,  $a\ge a^r=Y^*_{-V(-r)}$ and $b=0$, we necessarily have $\cShat_{x,V(-r)}\circ \cS_{-y,x}(0)< a^r$. We deduce that $V(-y)$ is smaller than the RHS of \eqref{eq:V(-y)}. We prove the reverse inequality. We can suppose that $V(-y)<V(-r)$ otherwise the inequality is clear. Let  $a^y=\cS_{-y,V(-y)}(0)=Y^*_{-V(-y)}$. By \eqref{eq: V upper bound} with $y$ instead of $r$ and $V(-y)$ instead of $x$, $\cShat_{V(-y),z}(a^y)>0$. In particular, since $\cShat$ is killed, $\cShat_{V(-y),V(-r)}(a^y)>0$. By another use of Proposition \ref{p:properties dual} (i), $\cShat_{V(-y),V(-r)}(a^y)>a^r$. We get \eqref{eq:V(-y)}.  Let $\tcS$ be the non-killed version of $\widehat{\cS}$. Observe that for any $x\in [-y,V(-r))$, $\cShat_{x,V(-r)}\circ \cS_{-y,x}(0)> a^r$ if and only if $\tcS_{x,V(-r)}\circ \cS_{-y,x}(0)> a^r$. One direction comes from $\cShat\le \tcS$. For the other direction, if $\tcS_{x,V(-r)}\circ \cS_{-y,x}(0)> a^r$, then $\tcS_{x,c}\circ \cS_{-y,x}(0)> \cS_{-r,c}(0)$ for any $c\in (x,V(-r))$. Otherwise, $\tcS_{x,c} \circ \cS_{-y,x}(0)= \cS_{-r,c}(0)$ for some $c\in (x,V(-r))$, and $\tcS_{x,V(-r)}\circ \cS_{-y,x}(0)=\tcS_{c,V(-r)}\circ \cS_{-r,c}(0)\le a_r$ since $\tcS\le \cS$ by the comparison principle of Proposition \ref{p:comparison}. Hence \eqref{eq:V(-y)} also reads
\begin{equation}\label{eq:repr V 2}
    V(-y)=\inf\{x\in [-y,V(-r))\,:\, \tcS_{x,V(-r)}\circ \cS_{-y,x}(0)> a^r\}\land V(-r).
\end{equation}

\noindent We want to apply Proposition \ref{p:difference} with $\cS_{-r,\cdot}(0)$ in place of $Y$. We have in the notation of the proposition $\cS^+_{-y,x}(0)=\cS_{-y,x}(0)-\cS_{-r,x}(0)$, where we set $\cS_{-r,x}(0)=0$ when $x< -r$.   With natural notation (using $\tcS$ instead of $\cS$ in Proposition \ref{p:difference}), we deduce that for any $x\le r'\le \widetilde{\varphi}_+\left(\cS^+_{-y,x}(0),x\right)$,
\[
\tcS^+_{x,r'}\circ \cS^+_{-y,x}(0) =\tcS_{x,r'}\circ \cS_{-y,x}(0) - \cS_{-r,r'}(0).
\]

\noindent  From our choice of parameters, $\widetilde{\varphi}_+\left(\cS^+_{-y,x}(0),x\right)=\inf\{c>x\,:\, \tcS_{x,c} \circ \cS_{-y,x}(0) \le \cS_{-r,c}(0)\}$.  We recall that by the comparison principle, $\tcS_{x,c}\circ \cS_{-y,x}(0)\le \cS_{-r,c}(0)$ for all $c\ge \widetilde{\varphi}_+\left(\cS^+_{-y,x}(0),x\right)$. Substituting $r'$ for $V(-r)$, we deduce that we can rewrite $V(-y)$ as 
\[
V(-y)=\inf\{x\in [-y,V(-r))\,:\, \tcS^+_{x,V(-r)}\circ \cS^+_{-y,x}(0)> 0\}\land V(-r).
\]

\noindent By Proposition \ref{p:V adapted}, $V(-r)$ is measurable with respect to $\mathscr{W}^*_r$. The Markov property is then a consequence of Proposition \ref{p:difference} (i) and (iii), applied to $\cS$ and $\tcS$. We can  check the time-homogeneity from the expression of $V(-y)$ in the last display. Finally, we prove the Feller property. By scaling, the distribution of $V(-r)$ under $P_*^z$ is that of $zV(-r/z)$ under $P_*^1$. 
Hence $P_*^z(V(-r)\neq z)=P_*^1(V(-r/z)\neq 1)\le P_*^1(V(-\sqrt{r})\neq 1)$ if $z\ge \sqrt{r}$ while $V(-r)\in [-r,\sqrt{r}]$ under $P_*^z$ when $z\le \sqrt{r}$. We deduce the Feller property.

\item By \eqref{eq: V second def}, $ P_*^z({\rm x} > r)=\p(\cS_{-r,z}(0) = 0)$. Then, by Lemma \ref{distribution:hitting0}, 
\begin{equation}\label{eq: distr x V}
    P_*^z({\rm x} > r) = \p(T_r<z+r)= \left( 1- \frac{r}{z+r}\right)^{\frac{\delta}{2}},
\end{equation}

\noindent where $T_r$ is the hitting time of $0$ after time $r$ by a ${\rm BESQ}(\delta \, |_r \, 0)$ process. It gives the distribution of $\rm x$. Let $r>x\ge 0$ and $y\in [-x,z)$. By equations \eqref{eq: V lower bound} and \eqref{eq: V upper bound},
\begin{align}\nonumber
P_*^z( \mathrm{x} \in (x,r],\, V(-\mathrm{x}) \le y)
&\le 
P_*^z( V(-x)=z, V(-r)\le y) \\
& \le
\p(\cS_{-x,z}(0)=0,\, \cShat_{y,z}\circ \cS_{-r,y}(0)>0).\label{eq: proof V<r}
\end{align}

\noindent  By property ({\bf P}) of Section \ref{s:def},
\begin{align*}
\p(\cS_{-x,z}(0)=0,\, \cShat_{y,z}\circ \cS_{-r,y}(0) >0)
& =\p(T_x<x+z)\p(T_{r-x,y+x}>x+z)\\
&=\p(\mathrm{x}> x)\p(T_{r-x,y+x}>x+z)
\end{align*}

\noindent where, as before, $T_x$ is the hitting time of $0$ after time $x$ by a ${\rm BESQ}_0(\delta\,|_x\, 0)$ process and  $T_{a,b}$ is the hitting time of $0$ after time $b$ by a ${\rm BESQ}_{-a}(\delta\,|_0\, 0 \,|_{b}\, 2-\delta-\deltahat)$ process. By Lemma \ref{l:conv_hitting0} (ii), 
$$
\lim_{r\downarrow x} \frac{1}{r-x}\p(T_{r-x,y+x}> x+z)= \frac{\delta}{2(y+x)} \bigg( \frac{x+y}{x+z} \bigg)^{\frac{\delta+\deltahat}{2}}.
$$

\noindent Together with \eqref{eq: distr x V}, we get 
\begin{equation}\label{eq:proof V limit}
\lim_{r\downarrow x} \frac{1}{r-x} \p(\cS_{-x,z}(0)=0,\, \cShat_{y,z} \circ \cS_{-r,y}(0) >0)=\frac{\delta}{2z} \bigg( \frac{z}{x+z} \bigg)^{\frac{\delta}{2}-1} \bigg( \frac{x+y}{x+z} \bigg)^{\frac{\delta+\deltahat}{2}-1}.
\end{equation}

\noindent Going back to \eqref{eq: proof V<r}, it yields that, 
\[
\limsup_{\varepsilon \downarrow 0}\limsup_{r\downarrow x} \frac{1}{r-x}  P_*^z( \mathrm{x} \in (x,r],\, V(-\mathrm{x}) \le -x+\varepsilon )=0.
\]

\noindent By the strong Markov property at time $\mathrm{x}$, and since $z'\to P_*^{z'}(\mathrm{x}\le r-x)$ is non-increasing in $z'$,
\begin{align*}
&P_*^z(\mathrm{x}>  x,\, \text{ jumps at least twice in } (x,r]) \\
&\le
P_*^z(\mathrm{x}\in (x,r],\, V(-\mathrm{x}) \le -x+\varepsilon)+
P_*^z(\mathrm{x}\in  (x,r])P_*^{\varepsilon}(\mathrm{x}\le  r-x).
\end{align*}

\noindent The second term is $O((r-x)^2)$ by \eqref{eq: distr x V}. Therefore, making $r\downarrow x$ then $\varepsilon\downarrow 0$, 
\[
\lim_{r\downarrow x} \frac{1}{r-x}P_*^z(\mathrm{x}>  x,\, \text{ jumps at least twice in } (x,r])=0.
\]

\noindent On the other hand,
\begin{align*}
& \p(\cS_{-x,z}(0)=0,\,  \cShat_{y,z}\circ \cS_{-r,y}(0)  >0) \\
&\le 
P_*^z( \mathrm{x} \in (x,r],\, V(-\mathrm{x}) \le y) + 
P_*^z(\mathrm{x}> x,\, \text{ jumps at least twice in } [x,r]).
\end{align*}

\noindent We conclude from \eqref{eq: proof V<r} and \eqref{eq:proof V limit} that 
$$
P(\mathrm{x} \in \d x,\, V(-x) \le y) =\frac{\delta}{2z} \bigg( \frac{z}{x+z} \bigg)^{\frac{\delta}{2}-1} \bigg( \frac{x+y}{x+z} \bigg)^{\frac{\delta+\deltahat}{2}-1} \dd x.
$$

\noindent The proof of the theorem is complete. 

\end{enumerate}
\end{proof}


\section{Application to the skew Brownian motion}
\label{s:skewBM}

In this section we  apply  the results of Section \ref{s:meeting} to the skew Brownian flow. Let $B$ be a Brownian motion, $\beta\in (-1,0)\cup (0,1)$ and $r\in \mathbb{R}$. Recall  from the introduction that $X^r$ denotes the strong solution of the SDE \eqref{eq:SDE Xtr} with $L^r$ satisfying \eqref{def:L}, and $\cL(t,x)$ is the bicontinuous version of the local time of $B$ at time $t$ and position $x$.

\subsection{Embedding in a BESQ flow}\label{s:repre_skew_BM}

This section aims at embedding the collection $(X^r)_{r\in \mathbb{R}}$ in a ${\rm BESQ}$ flow driven by the white noise $\cW$ defined in \eqref{def:W}. 
In \cite{aidekon2024infinite}, $\cW$  is shown to be related to Ray--Knight theorems for the Brownian motion.  Specifically, let $\tau_a^{B,r}:=\inf\{t\ge 0\,:\, \cL(t,r)>a \}$ be the inverse local time of $B$. Then $(\cL(\tau^{B,r}_a,x),\,x\ge r)_{(a,r)\in \r_+\times \r}$ is a ${\rm BESQ}(2\,|_0\, 0)$ flow driven by $\cW$. Its dual  $(\cL(\tau^{B,-r}_a,-x),\,x\ge r)_{(a,r)\in \r_+\times \r}$ is a ${\rm BESQ}(2\,|_0\, 0)$ flow driven by $-\cW^*$, where $\cW^*$ is the image of $\cW$ by the map $(a,x)\mapsto (a,-x)$. See \cite[Proposition 3.6]{aïdékon2023stochastic}.

\bigskip

We go back to the study of the skew Brownian flow. We first treat the case $\beta\in (0,1)$, then deduce the  case $\beta\in (-1,0)$ by symmetry. \\

Recall from \eqref{def:taurx} that for any $r\in \mathbb{R}$
\[
\tau^r_x:=\inf\{ t\ge 0 \,:\,  L_t^r > x\},\, x\ge r.
\] 

\noindent The following lemma collects some properties on the local times of $X$. Since $|X|$ is a reflecting Brownian motion, we can apply standard results of the theory of local times of the Brownian motion, see \cite[Chapter VI]{revuz2013continuous}.

\begin{lemma}\label{l:taurx}
Let $r\in \r$. Almost surely: for all $x\ge r$ and $t\ge 0$,
\begin{enumerate}[(i)]
    \item $L_t^r\le x$ for all $t\in [0, \tau_x^r)$, $L_t^r=x$ for $t=\tau_x^r$ and $L_t^r>x$ for all $t>\tau_x^r$;
    \item $X^r_{\tau^r_x}=0$ and $B_{\tau^r_x}=x$;
    \item $\cL(t,x)>\cL(\tau^r_x,x)$ for $x<L_t^r$ and $\cL(t,x)\le \cL(\tau^r_x,x)$ for $x\ge L_t^r$. In particular, $L^r_t=\sup\{x>r\,:\, \cL(t,x)>\cL(\tau^r_x,x)\}\lor r$ where $\sup\emptyset=-\infty$ by convention;
    \item $\cL(t,x)>\cL(\tau_x^r,x)$ if and only if $t>\tau_x^r$;  if $\cL(t,x)<\cL(\tau_x^r,x)$, then $L_t^r<x$.
\end{enumerate}
\end{lemma}
\begin{proof}
\begin{enumerate}[(i)]
    \item It is a consequence of the definition of $\tau_x^r$ and the continuity of $L^r$. 
    \item By \cite[Chapter VI, Proposition 1.3]{revuz2013continuous},  $ \dd_t L_t^r $ is almost surely carried by the zeros of $ X_t^r $. It implies that $\tau_x^r$ must be a zero of $X$. 
    Since $ X_t^r = L_t^r - B_t $, we conclude that $ B_{\tau_x^r} = x $. 
    \item $x<L_t^r$ implies $t>\tau_x^r$ by (i), and hence $\cL(t,x)\ge \cL(\tau_x^r,x)$. Let us show that $\cL(t,x)> \cL(\tau_x^r,x)$. For $s\ge 0$, let $A^+_s:=\int_0^s \ind_{(0,+\infty)}(X^r_u) \dd u$, $\alpha^+$ be the right-continuous inverse of $A^+$ and $B^+_u:=B_{\alpha^+_u}$. Set $s=\tau^r_x$ and notice that $s$ is a point of increase of $A^+$. By \cite[Lemma 2.3]{pitman2018squared}, $-B^+ \dequiv |B|-\mu\cL(\cdot,0)$ with $\mu=\frac{2\beta}{1+\beta}<1$, hence has no monotone points a.s.   Thus one can find  $u>A_s^+$ arbitrarily close to $A_s^+$ such that $B^+_u=B^+_{A_s^+}$, i.e. $B_{\alpha_u^+}=B_s=x$. Since $\alpha_u^+ \downarrow s$ as $u\downarrow A_s^+$, it implies that $s$ is a point of increase of $\cL(\cdot,x)$, hence $\cL(t,x)>\cL(\tau_x^r,x)$. Finally, if $x\ge L_t^r$, then $t\le \tau_x^r$ hence $\cL(t,x)\le \cL(\tau^r_x,x)$.

    \item The first statement can be derived from (i) and (iii). For the second one, $\cL(t,x)<\cL(\tau_x^r,x)$ implies $t<\tau_x^r$. By (i), we have $L_t^r\le x$. If $L_t^r=x$, then $L_u^r=x$ for any $u\in [t, \tau_x^r]$. Therefore, $X_u^r\neq 0$ for all $u\in (t,\tau_x^r)$, which means $B$ is making an excursion away from $x$. This contradicts $\cL(t,x)<\cL(\tau_x^r,x)$.
\end{enumerate}
\end{proof}

 Recall that $\mathcal W$ was defined in \eqref{def:W} and that the ${\rm BESQ}( \delta \, |_a \, \delta')$ driven by $\mathcal W$ was defined in Definition \ref{def:vary}. We prove Theorem \ref{thm:representation_by_WN} stated in the introduction. Let $\beta\in(0,1)$, $\delta:=\frac{1-\beta}{\beta}$ and $r\in \r$. We show that the process $(\cL(\tau^r_x,x),\,x\ge r)$ is the  flow line starting at the point $(0,r)$ of the ${\rm BESQ}(2+\delta \, |_0 \, \delta)$ flow driven by $\cW$.

\begin{proof} [Proof of Theorem \ref{thm:representation_by_WN}]
By definition of the ${\rm BESQ}(2+\delta \, |_0 \, \delta)$ flow, we need to show that the process $(\cL(\tau^r_x,x),\,x\ge r)$ satisfies
\begin{align*}
     \cL(\tau^r_x,x) = 2 \int_r^x \cW([0,\cL(\tau^r_y,y)], \dd y)   +\delta (x-r) + 2(r^--x^-),\, x\ge r
\end{align*}
where $z^-:=\max(-z,0)$. We write for any $z\in \r$, $z^+=\max(z,0)$. By Tanaka's formula \cite[Chapter VI, Theorem 1.2]{revuz2013continuous}, we have for $t\ge 0$ and $x\ge r$,
\begin{align*}
    (x-B_t)^+ = x^+ - \int_0^t \ind_{\{B_u< x\}} \dd B_u + \frac{1}{2}\cL(t,x)
\end{align*}
and
\begin{align*}
    (X_t^r)^+ &=  (X_0^r)^+ +\int_0^t \ind_{\{X_u^r> 0\}} \dd X_u^r + \frac{1}{2}\phi^r(t,0)\\
    &=
    r^+ -\int_0^t \ind_{\{X_u^r> 0\}} \dd B_u + \frac{1}{2}\phi^r(t,0)
\end{align*}
where $\phi^r(t,z)$ denotes the local time of $X^r$ at time $t$ and position $z$, taken continuous in $t$ and right-continuous in  $z$.  By \cite{walsh1978diffusion}, equation (8), $\phi^r(t,0)=(1+\beta)\ell_t^r$ in the notation \eqref{def:Ltilde}. By Lemma \ref{l:taurx} (ii), for all $x\ge r$,  $X_t^r=0$  and $B_t=x$ at time $t=\tau^r_x$. Recall \eqref{def:L}. Using the two equations above for such a $t$, we get
\begin{equation*}\label{eq:proofflowline1}
    \cL(\tau_x^r,x) = 2 (r^+ - x^+ ) + 2 \int_0^{\tau_x^r} ( \ind_{\{B_u<x\}} - \ind_{\{X_u^r> 0\}}) \dd B_u + \frac{1+\beta}{\beta} (x-r). 
\end{equation*}

\noindent Since $\frac{1+\beta}{\beta}-2=\delta$, it remains to show that
\begin{equation}\label{eq:proofflowline2}
\int_0^{\tau_x^r} ( \ind_{\{B_u<x\}} - \ind_{\{X_u^r> 0\}}) \dd B_u 
=\int_r^x \cW([0,\cL(\tau^r_y,y)], \dd y). 
\end{equation}

\noindent Having $X_u^r>0$ and $u\le \tau^r_x$ implies by Lemma \ref{l:taurx} (i) that $B_u=L_u^r -X_u^r \le x - X_u^r < x$. Hence, for $u\in [0,\tau^r_x]$,
$$
      \ind_{\{B_u< x\}} - \ind_{\{X_u^r> 0\}} =   \ind_{ \{B_u< x,\, X_u^r\le 0\}}. 
$$

\noindent Notice that if $X_u^r<0$, then $B_u>L_u^r$ by definition of $X^r$ hence $u<\tau^r_{B_u}$ by Lemma \ref{l:taurx} (i) and $B_u>r$ by definition of $L^r$. As long as $u$ is not the start or the end of an excursion of $B$, it also implies that $\cL(u,B_u) < \cL(\tau_{B_u}^r,B_u)$.   Conversely, if $\cL(u,B_u) < \cL(\tau_{B_u}^r,B_u)$, then $u<\tau_{B_u}^r$ hence $X_u^r=L^r_u-B_u\le 0$ by another use of Lemma \ref{l:taurx} (i). So we proved that as long as $u$ is not associated to some excursion of $B$ and $X_u^r \neq 0$, 
\begin{equation*}\label{eq:proofflowline3}
\ind_{\{B_u< x\}} - \ind_{\{X_u^r> 0\}} = \ind_{\{\cL(u,B_u) < \cL(\tau_{B_u}^r,B_u), \, r < B_u < x\}}.
\end{equation*}

\noindent This equation is therefore true almost surely for Lebesgue-a.e. $u$. Equation \eqref{eq:proofflowline2} is then a consequence of the definition of $\cW$ in equation \eqref{def:W} and Proposition 3.2 of \cite{aidekon2024infinite}. 
\end{proof}

\medskip

\begin{corollary}\label{c:embedding+}
    Let $\beta \in(0,1)$ and $\delta:=\frac{1-\beta}{\beta}$. Let $\cS$ be the ${\rm BESQ}(\delta+2 \,|_0 \, \delta)$ flow driven by $\mathcal W$. Fix $r\in \r$. Then $L_t^r=\sup\{x>r\,:\, \cL(t,x) > \cS_{r,x}(0)\}\vee r$.
\end{corollary}
\begin{proof}
It is a consequence of Theorem \ref{thm:representation_by_WN} and Lemma \ref{l:taurx} (iii). 
\end{proof}

\bigskip

We deal with the case $\beta\in (-1,0)$ by symmetry. In this case, we define for any $r\in \mathbb{R}$
\begin{equation}\label{def:tau beta negative}
    \tau^r_x:=\inf\{ t\ge 0 \,:\,  L_t^r< x\},\, x\le r.
\end{equation}

\noindent Recall that $\cW^*$ denotes the image of $\cW$ by the map $(a,x)\mapsto (a,-x)$.

\begin{proposition}\label{p:embedding-}
Let $\beta\in (-1,0)$, $\delta:=\frac{1-|\beta|}{|\beta|}$ and $r\in \mathbb{R}$. The process $(\cL(\tau^{-r}_{-x},-x),\,x\ge r)$ is the flow line starting at $(0,r)$ of the ${\rm BESQ}(\delta+2\, |_0\,  \delta)$ flow $\cS^*$ driven by $-{\mathcal W}^*$. Moreover,  for any $r\in \mathbb{R}$, a.s.,  $L_t^r=\inf\{x<r\,:\, \cL(t,x) > \cS^*_{-r,-x}(0)\} \land r$ with the convention that $\inf \emptyset=\infty$.
\end{proposition}

\begin{proof}
We notice that for any $g\in L^2(\r_+\times\r)$ with compact support,
\[
    \cW^*(g) = \int_{\r_+\times\r} g(\ell,-x) \cW(\dd \ell,\dd x)
    = \int_0^{+\infty} g(\cL(t,B_t),-B_t)\dd B_t.
\]
Therefore $-\cW^*$  admits the representation  \eqref{def:W} when replacing $B$ with $-B$. The result is then a consequence of Theorem \ref{thm:representation_by_WN} and Corollary \ref{c:embedding+} applied to $-B$.
\end{proof}

\subsection{Meeting of skew Brownian motions}

This section is devoted to the proof of Theorem \ref{thm: skewBM meeting}. Recall that $X^r$ and $\widehat X$ are the solutions of \eqref{eq:SDE Xtr}, associated to the parameter $\beta\in (-1,0)\cup (0,1)$ and the fixed parameter $\widehat{\beta} \in (0,1)$ respectively and starting from $r$ and $0$. 
We defined $\widehat L_t:=\widehat X_t+B_t$, as in \eqref{def:L}. Let $Y$ be the flow line starting at $(0,0)$ of the ${\rm BESQ}^{\deltahat}$ flow driven by $\cW$, with $\deltahat:=\frac{1-\widehat{\beta}}{\widehat{\beta}}$. Finally, for $x\ge 0$, $\widehat \tau_x:=\inf\{t\ge 0\,:\, \widehat L_t>x\}$ is the inverse local time at $0$ of $\widehat X$. Recall from the introduction that  $T(r)$ is the first meeting time of the processes $X^r$ and $\widehat X$ as defined in \eqref{def:Tbetax}. Theorem \ref{thm: skewBM meeting} (i)-(ii) and (iii) are consequences of  Propositions \ref{description:U_beta>0} and \ref{description:U_beta<0} respectively. \\

We first treat the case $\beta\in(0,1)$. Write $\cS$ for the ${\rm BESQ}(2+\delta\,|_0 \, \delta)$ flow driven by $\cW$, with $\delta:=\frac{1-\beta}{\beta}$.  We let $\widehat{L}_{T(r)}:=\infty$ if $T(r)=\infty$ and as usual $\inf \emptyset =\infty$ by convention.

\begin{proposition}\label{description:U_beta>0}
 Let $\beta\in (0,1)$.  For any $r\in \r$, a.s.,
\begin{equation}\label{eq:U}
   \widehat{L}_{T(r)}=\inf\{x\ge \max(r,0)\,: \, \cS_{r,x}(0)=Y_x \}. 
\end{equation}
In particular, Theorem \ref{thm: skewBM meeting} (i)-(ii) hold by Theorems \ref{thm:U_process} and \ref{thm:U_process-}.

\end{proposition}
\begin{proof}
Let $U$ denote the RHS of \eqref{eq:U}. Suppose that $X_t^{r}=\widehat X_t$ for some $t\ge 0$. By \eqref{eq:SDE Xtr}, we would have $L_t^{r}=\widehat L_t$. By Corollary \ref{c:embedding+} or Lemma \ref{l:taurx} (iii), we have for any $s\ge 0$,  ${\mathcal L}(s,L_s^{r})=\cS_{r,L_s^{r}}(0)$ and ${\mathcal L}(s,\widehat L_s)=Y_{\widehat L_s}$. It implies that $\cS_{r,\widehat L_t}(0)=Y_{\widehat L_t}$, therefore $\widehat L_t\ge U$.  Applying it to $t=T(r)$ in case it is finite, we get $\widehat{L}_{T(r)}\ge U$.
Suppose now that $U<\infty$. Then $\cS_{r,U}(0)=Y_U$, hence using Theorem \ref{thm:representation_by_WN}, $\cL(\tau_U^{r},U)=\cL(\widehat \tau_U, U)$. Lemma \ref{l:taurx} (iv) implies that $\tau_U^{r}=\widehat \tau_U$ and at this time $X^{r}$ and $\widehat X$ are both at position 0 by Lemma \ref{l:taurx} (ii). In particular, $T(r)\le \widehat \tau_U$ hence $\widehat{L}_{T(r)}\le U$  by Lemma \ref{l:taurx} (i).

\end{proof}

We treat the case $\beta\in(-1,0)$. From \eqref{eq:SDE Xtr}, observe that $T(r)=\inf\{ t\geq 0: L_t^{r}=\widehat{L}_t\}=\inf \{t\geq 0: \widehat{\beta}\,  \widehat{\ell}_t - \beta \ell_t^{r}= r\}$ where $\widehat{\ell}$ is the symmetrized local time at $0$ of $\widehat{X}$. It is then finite a.s. if $r\ge 0$ and infinite a.s. if $r<0$. 
Recall from \eqref{def:tau beta negative} that when $\beta<0$,
$$\tau^{r}_x:=\inf\{t\ge 0: L^{r}_t<x\},\ x\le r.$$
 Let $\cS$ be the killed $\rm BESQ(2-\delta\,|_0 \, -\delta)$ flow driven by $\cW$ and $\cS^*$ be its dual flow. Note that $\cS^*$ is a non-killed $\rm BESQ(2+\delta\,|_0 \, \delta)$ flow driven by $-\cW^*$ by Proposition \ref{p:BESQdual vary}. 

\begin{proposition}\label{description:U_beta<0}
Let $\beta\in(-1,0)$ and $\delta=\frac{1-|\beta|}{|\beta|}$. For any $r\ge 0$, a.s.
\begin{equation}\label{eq:U beta neg}
    \widehat{L}_{T(r)}=-\inf\{x\in [-r,0]\,:\, \cS^*_{-r,x}(0)=Y_{-x}\}.
\end{equation}

\noindent As a result, Theorem \ref{thm: skewBM meeting} (iii) holds (take $(\cS^*,Y,0)$ in place of $(\cS,Y^*,z)$ in Section \ref{s:meeting*}).
\end{proposition}
\begin{proof}
 Fix $r> 0$. By Proposition \ref{p:embedding-}, $(\cL(\tau^r_x,x),\,x\le r)$ is the flow line $\cS^*_{-r,-x}(0)$. 
  For $s\ge 0$, $\cL(s,L^{r}_s)=\cS^*_{-r,-L^{r}_s}(0)$ and $\cL(s,\widehat{L}_s)=Y_{\widehat{L}_s}$ by Proposition \ref{p:embedding-} and Corollary \ref{c:embedding+} respectively.
   Then \eqref{eq:SDE Xtr} implies that $\cS^*_{-r,-\widehat{L}_t}(0)=Y_{\widehat{L}_t}$ when $t=T(r)$. 
   If we show that $\cS^*_{-r,-x}(0)= Y_x$ for a unique $x\in [0,r]$, then  \eqref{eq:U beta neg} is proved. Let  ${\mathcal Y}$ be the ${\rm BESQ}^{\deltahat}$ flow driven by $\cW$ and  ${\mathcal Y^*}$ be the dual flow of ${\mathcal Y}$, i.e. a killed ${\rm BESQ}^{2-\deltahat}$ flow. Observe that $Y_x={\mathcal Y}_{0,x}(0)$ and  $2-\deltahat<2+\delta$.
   By the comparison principle in Proposition \ref{p:comparison}, $\cS^*\ge {\mathcal Y}^*$ below level $0$. 
   By Proposition \ref{p:properties dual} (i) 
   applied to ${\mathcal Y}^*$ in place of $\cS$ between level 
   $-x<0$ and $0$, ${\mathcal Y}^*_{-x,0}(Y_x)>0$.
   By Proposition \ref{p:properties dual} (ii) applied to ${\mathcal Y}$ in place of $\cS$ with $b=Y_x$ and $a'=0<a={\mathcal Y}^*_{-x,0}(Y_x)$, we have for $x'\in (0,x)$, $Y_{x'}<{\mathcal Y}^*_{-x,-x'}(Y_x)$ except if ${\mathcal Y}^*_{-x,-x'}(Y_x)=0$. 
   Since $\cS^*$ is a ${\rm BESQ}^{2+\delta}$ flow below $0$, $\cS^*_{-x,-x'}(a)>0$ for all $x' \in [0,x)$ and $a\ge 0$. We deduce from $\cS^*\ge \cY^*$ that $\cS^*_{-x,-x'}(Y_x)>Y_{x'}$ for all $x'\in(0,x)$. Hence if $\cS^*_{-r,-x}(0)=Y_x$ for some $x\in (0,r]$, then $\cS^*_{-r,-x'}(0)= \cS^*_{-x,-x'}(Y_x)>Y_{x'}$ for all $x'\in (0,x)$, where we used the perfect flow property of $\cS^*$. It finishes the proof of the claim. 
\end{proof}

\subsection{Ray-Knight theorems}

This section is devoted to the proof of Theorem \ref{thm: skewBM RK} (i) and (ii) which are consequences of Proposition \ref{p:skewBM RK i} and \ref{p:skewBM RK ii} respectively. Recall the SDE \eqref{eq:SDE Xtr} and  the notation \eqref{def:taurx}.  Fix $\beta\in(0,1)$ and let $\delta=\frac{1-\beta}{\beta}$. 
As in the last section, we make use of the embedding of the skew Brownian flow in a BESQ flow.
For $z\ge 0$, let $Y$ be the ${\rm BESQ}(\delta\,|_z\, 0)$ flow line starting from $(0,0)$ driven by the white noise $\cW$ defined in  \eqref{def:W}.  By the Ray--Knight theorem recalled at the beginning of Section \ref{s:repre_skew_BM} and  Theorem \ref{thm:representation_by_WN}, 
we have a.s.
\[
Y=(\cL(\tau^0_{\min(z,x)}, x),\, x\ge 0).
\]

\noindent In the following proposition, $\cS$ denotes the ${\rm BESQ}(2+\delta\, |_0\, \delta)$ flow driven by $\cW$. It recovers the distribution of  $(L_{\tau^0_z}^r,\, r\ge 0)$ computed in \cite[Theorem 1.2]{burdzy2001local}.

\begin{proposition}\label{p:skewBM RK i}
For any $r\ge 0$, a.s.
\begin{equation}
    L_{\tau^0_z}^r = \inf \{ x\ge \max(r,z): \cS_{r,x}(0)= Y_x\}.
\end{equation}

\noindent As a result, Theorem \ref{thm: skewBM RK} (i) holds (take $(\cS, Y, z)$ in place of $(\cS,Y^0,z)$ in Section \ref{s:meetingL}).
\end{proposition}
\begin{proof}
 Recall that the skew Brownian flow is coalescent, hence $X^r\ge X^0$ if $r\ge 0$ and by \eqref{eq:SDE Xtr} $L^r\ge L^0$. It implies that $\tau^0_x \ge \tau^r_x$, hence by Lemma \ref{l:taurx} (i), $L_{\tau^0_x}^r \ge x$ for any $x\ge 0$. By Theorem \ref{thm:representation_by_WN}, $\cS_{r,x}(0)=\cL(\tau^r_x,x)$. 
 Setting $x=L_{\tau^0_z}^r$, we have $\tau_x^r\ge \tau^0_z$. Necessarily $x\ge \max(r,z)$ and $\cS_{r,x}(0)\ge \cL(\tau^0_z,x)=Y_x$. Lemma \ref{l:taurx} (iv) says that $\cL(t,x)<\cL(\tau_x^r,x)$ implies $L_t^r<x$. Applying it to $t=\tau^0_z$, we deduce that $\cL(\tau^0_z,x)\ge \cL(\tau_x^r,x)$, i.e. $Y_x\ge \cS_{r,x}(0)$. We proved that $\cS_{r,x}(0)=Y_x$.

 Conversely, let $x\ge \max(r,z)$ such that $\cS_{r,x}(0)= Y_x$ and let us show that $x\ge L_{\tau^0_z}^r$. We have by assumption $\cL(\tau_x^r,x)=\cL(\tau^0_z,x)$. Lemma \ref{l:taurx} (iii) shows that $\cL(t,x)>\cL(\tau_x^r,x)$ if $x<L_t^r$. We apply it to $t=\tau^0_z$ to complete the proof.

\end{proof}

The next proposition gives the distribution of the process  $(L_{\tau^0_z}^{-r},\, r\ge 0)$. We keep the notation $\cS$ for the ${\rm BESQ}(2+\delta\, |_0\, \delta)$ flow driven by $\cW$. Let $\cShat$ be the ${\rm BESQ}(2\, |_0\, 0)$ flow driven by $\cW$. For any $x\ge r$ and $b\ge 0$, we let, with the convention $\cS_{0,y}(0)=\cShat_{0,y}=0$ if $y<0$,  
\begin{align*}
    \cS^+_{r,x}(b) &:=\cS_{r,x}(b+\cS_{0,r})-\cS_{0,x}(0),\\
    \cShat^+_{r,x}(b) &:= \max(\cShat_{r,x}(b+\cS_{0,r}(0))-\cS_{0,x}(0),0).
\end{align*}

\noindent It is in agreement with the notation of Proposition \ref{p:difference}. Indeed notice that  $\cShat_{r,x}(b+\cS_{0,r}(0))\le \cS_{0,x}(0)$  for $x\ge \inf\{y>r\,:\, \cShat_{r,y}(b+\cS_{0,r}(0))\le \cS_{0,y}(0)\}$ by the comparison principle in Proposition \ref{p:comparison} and the perfect flow property of $\cS$ and $\cShat$. By Proposition \ref{p:difference}, $\cS^+$ and $\cShat^+$ are respectively a ${\rm BESQ}(2+\delta\,|_0\, 0)$ flow and a ${\rm BESQ}(2\,|_0\, -\delta)$ flow driven by the same white noise.
\begin{proposition}\label{p:skewBM RK ii}
    For any $r\ge 0$, a.s.
    \begin{equation}\label{eq: RK L(-x)}
        L^{-r}_{\tau^0_z} = \inf\{x\in [-r,z)\,:\, \cShat^+_{x,z}\circ \cS_{-r,x}^+(0)>0\} \land z.
    \end{equation}
    As a result,  $(L^{-r}_{\tau^0_z},\, r\ge 0)$ is distributed as the process $(V(-r),\, r\ge 0)$ of Section \ref{s:meeting*} with $(2+\delta,0)$ in place of $(\delta,\deltahat)$ and Theorem \ref{thm: skewBM RK} (ii) holds. 
\end{proposition}
\begin{proof}
    By definition of $\cS^+$ and $\cShat^+$, we can rewrite \eqref{eq: RK L(-x)} as 
    \begin{equation}\label{eq: RK L(-x) proof}
         L^{-r}_{\tau^0_z} = \inf\{x\in [-r,z)\,:\, \cShat_{x,z}\circ \cS_{-r,x}(0)> \cS_{0,z}(0)\} \land z.
    \end{equation}

    \noindent Let us prove it. By Theorem \ref{thm:representation_by_WN}, $\cS_{-r,x}(0)=\cL(\tau^{-r}_x,x)$, and by the discussion at the beginning of Section \ref{s:repre_skew_BM}, $\cShat_{x,z}(b)=\cL(\tau^{B,x}_b,z)$. Let $x=L_{\tau_z^0}^{-r}$ and suppose that $x<z$. Let $b=\cS_{-r,x}(0)=\cL(\tau_x^{-r},x)$.  
    Hence $\tau_b^{B,x}\ge \tau_x^{-r}\ge \tau_z^0$ by Lemma \ref{l:taurx} (i). Since $B_{\tau_z^0}=z>x$ by Lemma \ref{l:taurx} (ii),  $\tau_b^{B,x} > \tau_z^0$. By the first statement of Lemma \ref{l:taurx} (iv), it implies that $\cL(\tau_b^{B,x},z)>\cL(\tau_z^0,z)$, hence $\cShat_{x,z}\circ \cS_{-r,x}(0)> \cS_{0,z}(0)$. It proves that $L^{-r}_{\tau^0_z}$ is greater than the RHS of \eqref{eq: RK L(-x) proof}. 
    Let us prove the reverse inequality. Since $X^{-r}\le X^0$, it implies that $L^{-r}\le L^0$, hence $\tau_x^0\le \tau_x^{-r}$ for any $x\ge 0$. Thus $L_{\tau^0_z}^{-r} \le z$ by Lemma \ref{l:taurx} (i). Suppose then that there exists $x\in[-r,z)$ such that $\cShat_{x,z}\circ \cS_{-r,x}(0)> \cS_{0,z}(0)$. We want to show that $x\ge L^{-r}_{\tau^0_z}$. We write again $b=\cS_{-r,x}(0)=\cL(\tau_x^{-r},x)$. Then  $\cShat_{x,z}\circ \cS_{-r,x}(0)= \cL(\tau_b^{B,x},z) > \cS_{0,z}(0)=\cL(\tau_z^0,z)$. In particular,  $\tau_b^{B,x}>\tau_z^0$, hence $\cL(\tau_z^0,x)\le b=\cL(\tau_x^{-r},x)$. We conclude with Lemma \ref{l:taurx} (iii). Use \eqref{eq: V second def} with $(\cShat^+, \cS^+,z)$ in place of $(\cShat, \cS,z)$ to recognize the process $(V(-r),\,r\ge 0)$ in \eqref{eq: RK L(-x)}. Recall indeed that we can replace in \eqref{eq: V second def} $\cShat$ by its non-killed version, see \eqref{eq:repr V 2} with $(r,0)$ in place of $(y,r)$.
    \end{proof}


\section{Bifurcation in the skew Brownian flow}\label{s:bifur_time}

Fix $\beta \in(0,1)$. In  \cite{burdzy2004lenses}, Burdzy and Kaspi generalize the skew Brownian flow to various starting times. Let $X^{s,r}$ be the solution of \eqref{eq:SDE Xtr} when replacing the Brownian motion $B$ by $(B_t-B_s,\, t\ge s)$, i.e.
\begin{align}\label{SDE_general}
    X_t^{s,r} =r - (B_t-B_s) + \beta \ell_t^{s,r},\, t\ge s
\end{align}
where  $\ell_t^{s,r}$ is the local time of $X^{s,r}$ at position $0$:
\begin{align*}
    \ell_t^{s,r}=
\begin{cases}
    \lim\limits_{\varepsilon \rightarrow 0} \frac{1}{2\varepsilon} \int_s^t \ind_{\{|X_u^{s,r}| < \varepsilon\}} \, \mathrm{d}u, &\quad t \ge s, \\
    0, & \quad 0 < t < s.
\end{cases}
\end{align*}

\noindent The processes $X^{s,r}$ can be simultaneously defined for all $s,r$  rationals. To extend the construction of the flow to all $(s,r)$, Burdzy and Kaspi define
\begin{align*}
    X_t^{s,r-}&=\sup_{u,y\in\mathbb{Q},\ u<s,\ X_s^{u,y}<r} X_t^{u,y},\\
    X_t^{s,r+}&=\inf_{u,y\in\mathbb{Q},\ u<s,\ X_s^{u,y}>r} X_t^{u,y}.
\end{align*}

\noindent We refer to \cite{burdzy2004lenses} for the properties of the flows $(X^{s,r-})_{s,r}$ and $(X^{s,r+})_{s,r}$. Almost surely, for all $s,r$ rationals,  $X^{s,r-}=X^{s,r+}=X^{s,r}$.  Nevertheless, there are exceptional times $s$, called bifurcation times, such that $X_t^{s,0-}<X_t^{s,0+}$ for all $t>s$ close enough to $s$, \cite[Theorem 1.3 (ii)]{burdzy2004lenses}.  A bifurcation time $s$ is called semi-flat if furthermore $X_t^{s,0-}<0$ for all $t>s$ close enough to $s$. By \cite[Theorem 1.4]{burdzy2004lenses}, semi-flat bifurcation times exist when $\beta\in (\frac13,1)$ and do not exist when $\beta\in (0,\frac13)$. We will show that at the critical value $\beta=\frac13$, semi-flat bifurcations time do not exist either. Burdzy and Kaspi predict that semi-flat bifurcation times are atypical, and add: ``{\it One can probably formalize the claim by computing the Hausdorff dimensions of ordinary and semi-flat bifurcation times for various values of $\beta$}", \cite{burdzy2004lenses}. This is the goal of this section. As noted by Burdzy and Kaspi, the exponent $\frac13$ already appears in \cite[Corollary 1.5]{burdzy2001local}, where it is shown that for $\beta>\frac13$, there exist times when $L_s^0=- \inf_{[0,s]} B$: ``{\it It would be interesting to find a direct link between that result and Theorem 1.4 above, for example, via a time reversal argument}", \cite{burdzy2004lenses}. We show that this intuition is correct, and give proofs of the Hausdorff dimensions via time-reversal arguments which connect \cite[Corollary 1.5]{burdzy2001local} and semi-flat bifurcation times, see Proposition \ref{p:rep_semiflattime}. To this end, we introduce for $t\ge 0$ and $x\in \r$, the backward skew Brownian motion $\check{X}^{t,x}$ which is driven by the time-reversed Brownian motion at time $t$, i.e. the solution of
\begin{align}\label{SDE_backward}
    \check{X}_s^{t,x} = x + \left(B_t-B_s\right) - \beta
    \check{\ell}_s^{t,x},\, 0\le s\le t
\end{align}
where  $\check{\ell}^{t,x}_s$ is the symmetric local time of $\check{X}^{t,x}$ at position $0$ in $[s,t]$.

\begin{proposition}\label{p:for_back_SDE}
With probability 1, for all quadruples $\left(s, r, t, x\right) \in \q^4$ with $t>s\ge 0$:

(i) if $r<\check{X}_s^{t, x}$, then $X_u^{s, r} \leq \check{X}_u^{t, x}$ for all $u \in [s,t]$, while if $r>\check{X}_s^{t, x}$, then $X_u^{s, r} \geq \check{X}_u^{t, x}$ for all $u \in [s,t]$;  

(ii) if $x<X_t^{s, r}$, then $\check{X}_u^{t, x} \leq X_u^{s, r}$ for all $u \in [s,t]$, while if $x>X_t^{s, r}$, then $\check{X}_u^{t, x} \geq X_u^{s, r}$ for all $u \in [s,t]$.
\end{proposition}
\begin{proof}
Following \cite[Theorem 1.6]{burdzy2001local}, we first introduce the approximation of the solutions to the SDE \eqref{SDE_general}. Let $f$ be a nonnegative smooth and symmetric function on $\r$, compactly supported on $[-\frac{1}{2}, \frac{1}{2}]$ with $\int_{\r} f(x)\dd x=1$. Denote $\frac{1}{2} \log ((1+\beta) /(1-\beta))$ by $\gamma$ and let $f_n(x)=n \gamma f(n x)$ for $x \in \r$ and $n \geq 1$. By \cite[Theorem 1.6]{burdzy2001local}, for any $r\in \r$ and $s\ge 0$, the solution, call it $ \Phi_{s,\cdot}^n(r)$, of 
$$
\Phi_{s,t}^n(r)= r - (B_t-B_s)+\int_s^t f_n\circ \Phi_{s,u}^n(r) \dd u,\, t\ge s
$$

\noindent converges in probability to $X^{s,r}$ in the space of continuous functions equipped with the topology of uniform convergence on compact intervals.  By \cite[Theorem 4.5.1]{kunita1990stochastic}, they define a Brownian flow, see \cite[Chapter 4]{kunita1990stochastic}. The maps $r\mapsto \Phi_{s,t}^n(r)$ are homeomorphisms of the real line. Let $\Phi_{t,s}^n:=(\Phi^{n}_{s,t})^{-1}$ denote the inverse maps. By \cite[Theorem 4.2.10]{kunita1990stochastic}, the inverse maps  also define a Brownian flow, called backward flow, solution of 
$$
\Phi_{t,s}^n(r)= r + (B_t-B_s)-\int_s^t f_n\circ \Phi_{t,u}^n(r) \dd u,\, s\in [0,t].
$$

\noindent  Since we consider homeomorphims, if $r<\Phi^n_{t,s}(x)$, resp. $r>\Phi^n_{t,s}(x)$, then $\Phi^n_{s,u}(r)<\Phi^n_{t,u}(x)$, resp. $\Phi^n_{s,t}(r)>\Phi^n_{t,u}(x)$ for $u\in [s,t]$.   By another use of \cite[Theorem 1.6]{burdzy2001local}, $(\Phi_{t,s}^n(x),\,s\in [0,t])$ converges in probability to $(\check{X}_{s}^{t,x},\,s\in [0,t])$ as $n\to\infty$. We deduce statement (i) of the proposition. The proof of (ii)  follows similar lines.
\end{proof}

\begin{proposition}\label{p:rep_ordbifurtime}
    Almost surely, the set of bifurcation times is  $\{s: \exists t\in \q_+ \text{ with } t>s \text{ and } \check{X}_s^{t,0}=0 \}$. 
\end{proposition}
\begin{proof}
    Assume $s$ is a bifurcation time. Then there exists $t>s $ rational such that $X_{t}^{s,0-}<0<X_{t}^{s,0+}$. Otherwise the signs of $X^{s,0-}$ and $X^{s,0+}$ should be the same on a neighborhood of $s$, which entails that they are simply equal to $B_s-B_t$ hence are equal. For any rationals $s',r$ such that $s'<s$ and $X_s^{s',r}<0$, we have $ X_t^{s',r} \le X_t^{s,0-}<0$.     By Proposition \ref{p:for_back_SDE} (ii), it implies that $\check{X}^{t,0}_{u}\ge X_{u}^{s',r}$ for any $u\in [s',t]$, hence for $u=s$. By definition of $X^{s,0-}_s$, we find that $\check{X}^{t,0}_{s}\ge X_s^{s,0-}=0$, the last equality by \cite[Proposition 1.1]{burdzy2004lenses}. A similar reasoning for $X^{s,0+}$ yields $\check{X}^{t,0}_s\le 0$, hence $\check{X}^{t,0}_s=0$. 
    
    Suppose now that there exists $t\in \q_+$ such that $t>s$ and $\check{X}_s^{t,0}=0$. Let rationals $s',r$ such that $s'<s$ and $X_s^{s',r}>0$. By Proposition \ref{p:for_back_SDE} (i), one must have $\check{X}^{t,0}_{s'}\le r$. The probability that $\check{X}^{t,0}_{s'}\in {\mathbb Q}$ is $0$ hence $\check{X}^{t,0}_{s'}<r$. By another use of   Proposition \ref{p:for_back_SDE} (i),  it holds that  $X_{u}^{s',r} \geq \check{X}_{u}^{t, 0}$, for all $u\in [s',t]$, hence for $u=t$. By the definition of $X^{s,0+}$,  $X_t^{s, 0+} \geq \check{X}_t^{t, 0}=0$. We cannot have $X_t^{s,0+}=0$. In fact, by \cite[Lemma 2.7]{burdzy2004lenses}, for any $s\ge 0$, if the local time of $X^{s,0+}$ at 0 is positive at some time $t'\in(s,t)$, then there exist $u,r'\in \q$ such that $X_t^{s, 0+}=X_t^{u, r'}$, and the probability that $X^{u,r'}$ for some rationals $u,r'$ hits $0$ at a rational time is zero. If the local time of $X^{s,0+}$ at 0 is still 0 at time $t$, then $X^{s,0+}_u = B_s-B_u$ for all $u\in(s,t)$ by \cite[Proposition 1.1]{burdzy2004lenses} and \eqref{SDE_general}. It implies $B$ is making an excursion in $(s,t)$, and the probability that $B$ ends an excursion at a rational time is $0$. A similar argument applies to $X^{s,0-}_t$. We deduce that $X_t^{s,0-}<0<X_t^{s,0+}$ hence $s$ is a bifurcation time.
\end{proof}

\begin{theorem}\label{orddimension_times}
The set of ordinary bifurcation times has Hausdorff dimension $\frac{1}{2}$ almost surely. 
\end{theorem}
\begin{proof}
The zero sets of the Brownian motion and of the skew Brownian motion are identical in law. We obtain the result by Proposition \ref{p:rep_ordbifurtime} and the fact that the Hausdorff dimension of the zero set of the Brownian motion is almost surely $\frac{1}{2}$.

\end{proof}

\begin{proposition}\label{p:rep_semiflattime}
The set of semi-flat bifurcation times is a.s. $\{s: \exists t\in \q_+,\, t>s \text{ such that } \check{X}_s^{t,0}=0, \, B_s = \inf_{s\le u\le t} B_u \}$.   
\end{proposition}
\begin{proof}
Suppose first that $s$ is a semi-flat bifurcation time. Since $X^{s,0-}<0$  on a neighborhood of $s$,  $X^{s,0-}$ is $B_s-B_t$ and $B$ is making an excursion above $B_s$ on some interval $[s,s']$.  Take $t\in (s,s')$ rational  such that $\check{X}^{t,0}_s=0$ by Proposition \ref{p:rep_ordbifurtime}
since $s$ is a bifurcation time (such time $t$ can be chosen arbitrarily close to $s$). We have $B_s = \inf_{s\le u\le t} B_u$ indeed. 
We now suppose that there exists $t$ rational such that $\check{X}_s^{t,0}=0$ and $B_s = \inf_{s\le u\le t} B_u$ for some $s<t$. Then $s$ is a bifurcation point by Proposition \ref{p:rep_ordbifurtime}. Moreover, for every $(u,y)\in\q_+\times\q$ with $X^{u,y}_s<0$, $X^{u,y}_v = X^{u,y}_s - (B_v-B_s)$ for every $s<v<t$ by \eqref{SDE_general}. It implies $X^{s,0-}_v=X_s^{s,0-}-( B_v-B_s)=-(B_v-B_s)<0$ for $s<v<t$ by definition of $X^{s,0-}$ and \cite[Proposition 1.1]{burdzy2004lenses}. Hence $s$ is a semi-flat bifurcation time.
\end{proof}

The phase transition at $\beta=\frac13$ in the following proposition was shown in \cite[Corollary 1.5]{burdzy2001local} and recovered via BESQ processes in \cite[Section 4]{pitman2018squared}.

\begin{proposition}\label{p:repr semiflat}
Let $X^0$ be the solution to \eqref{eq:SDE Xtr} with $r=0$. Fix $t\ge 0$ and let $I:=\{s\in (0,t): L^0_s = B_s = \sup_{0\le u\le s} B_u\}$. If $\beta\le \frac{1}{3}$, then $I=\emptyset$ a.s. If $\beta>\frac{1}{3}$, it is not empty and its Hausdorff dimension is $\frac{3\beta-1}{4\beta}$ a.s.
\end{proposition}
\begin{proof}
Recall the definition of the inverse $\tau^0$ in \eqref{def:taurx}. Note that $\tau^0$ is a $\frac{1}{2}$-stable subordinator. By Theorem \ref{thm:representation_by_WN} or \cite[Theorem 1.3]{pitman2018squared}, the process $(\cL(\tau^0_x,x))_{x\ge 0}$ is a squared Bessel process of dimension $\delta=\frac{1-\beta}{\beta}$ starting at 0, therefore can hit $0$ at $x>0$ if and only if $\delta<2$, i.e. $\beta>\frac{1}{3}$.

 For every $s\in I$, $\cL(s,B_s) = 0$ because $B$ reaches its maximum at $s$. Also $\cL(\tau^0_{B_s},B_s)=0$, otherwise it would imply $L^0_s<B_s$ by the second statement of Lemma \ref{l:taurx} (iv) applied to  $r=0$ and $x=B_s$, which would contradict $s\in I$. Therefore $\cL(\tau^0_x,x)=0$ with $x=B_s>0$.  It already implies that $I$ is empty a.s. if  $\beta\le  \frac{1}{3}$. Moreover, $s\in \{\tau^0_{x-},\tau_x^0\}$ with $x=B_s$ since $L^0_s=B_s$. Hence \[
 I\subset \{\tau^0_{x-},\tau_x^0,\,:\, 0<x\le L^0_t,\, \cL(\tau^0_x,x)=0\}.
 \]

\noindent Conversely, if $\cL(\tau^0_x,x)=0$ for some $0<x\le L^0_t$, then we have $B_{\tau^0_x} = x$ by Lemma \ref{l:taurx} (ii) and $B_s\le x$ for all $s\le \tau^0_x$. Therefore $\tau^0_x\in I$ (recall that $t$ is fixed so $\tau_x^0<t$ a.s.) hence
\[
\{\tau_x^0\,:\, 0<x\le L^0_t,\, \cL(\tau^0_x,x)=0\}\subset I.
\]

\noindent Notice that the set of $x$ such that $\tau_{x-}^0<\tau_x^0$ is countable. Hence $I$ has the same Hausdorff dimension as the set on the left-hand side. The set of zeros of the process  $(\cL(\tau^0_x,x),\,x\ge 0)$ has Hausdorff dimension $\frac{2-\delta}{2}$. The result  follows from \cite[Theorem 4.1]{hawkes1974uniform}, which gives the dimension of the range of a Borel set under a $\frac{1}{2}$-stable subordinator.

\end{proof}

Recall that $\beta\in (0,1)$.
\begin{theorem}\label{t:hausdorff quadri times}
Semi-flat bifurcation times exist if and only if $\beta >\frac{1}{3}$. In this case, the set of semi-flat bifurcation times has Hausdorff dimension $\frac{3\beta-1}{4\beta}$ almost surely.
\end{theorem}
\begin{proof}
 By Proposition \ref{p:rep_semiflattime}, it suffices to deal with the set 
\[
I:=\{s<t: \check{X}_s^{t,0}=0, \, B_s = \inf_{s\le u\le t} B_u \}
\]

 \noindent for fixed $t$ and compute its Hausdorff dimension if it is not empty. Set $\check{B}^t_s:= B_t-B_s$ for $s\in [0,t]$. The set $I$ can be rewritten as
 \[
 \{s<t\,:\, \check{X}^{t,0}_s=0,\, \check{B}_s = \sup_{s\le u\le t} \check{B}_u\}.
 \]

\noindent We apply Proposition \ref{p:repr semiflat}.
\end{proof}

\begin{remark}
    The bifurcation times are related to the bifurcation points studied in Section \ref{s:bifur}. Let $\cS$ denote the ${\rm BESQ}(2+\delta\,|_0\, \delta)$ flow driven by $\cW$ given in \eqref{def:W}. Recall from Section \ref{s:repre_skew_BM} that the skew Brownian motion $X^r$ is associated with the $\cS$-flow line starting from $(0,r)$. A similar description holds for $X^{s,r}$ using the flow line starting from the point $(\cL(s,B_s+r),B_s+r)$.  From this point of view, if $s$ is a bifurcation time, then $(\cL(s,B_s),B_s)$ is a bifurcation point. If it is a semi-flat bifurcation time, then $(\cL(s,B_s),B_s)$ is a bifurcation point as in Theorem \ref{t:bifur quadri} with $\cS^2=\cS$ and $\cS^1$ the ${\rm BESQ}(2\,|_0\, 0)$ flow there. Theorems \ref{orddimension_times} and \ref{t:hausdorff quadri times} give the Hausdorff dimension of the times when the curve $(\cL(s,B_s),B_s)$ visits these points.  Finally, the backward skew Brownian motion $\check{X}$ is obtained when replacing the flow $\cS$ with its dual.
\end{remark}

\appendix

\section{Hitting time of BESQ processes}

\begin{lemma}\label{distribution:hitting0}
Let $(S_x,x\ge 0)$ be a ${\rm BESQ}(\delta_1\, |_r\, \delta_2)$ process starting at $0$ with $\delta_1>0$, $\delta_2<2$ and $r>0$. Let $T:=\inf\{x\ge r: S_x=0\}$. Then $\frac{r}{T}$ has the beta distribution $\mathcal{B}(\frac{2-\delta_2}{2},\frac{\delta_1}{2})$. In particular, the distribution of $\frac{r}{T}$ does not depend on $r$. 
\end{lemma}
\begin{proof}
The hitting time of $0$ by a ${\rm BESQ}^{\delta}_a$ process is distributed as $\frac{a}{2X}$ where $X\sim \Gamma(\frac{2-\delta}{2})$, see \cite[Exercise 1.23, Chapter XI]{revuz2013continuous} or \cite[Equation (15)]{going2003survey}. The r.v. $S_r$ is distributed as $2r Y$ where $Y\sim \Gamma(\frac{\delta_1}{2})$, see \cite[Corollary 1.4, Chapter XI]{revuz2013continuous}. We deduce that $\frac{r}{T}$ is distributed as $\frac{X}{X+Y}$ where $X\sim \Gamma(\frac{2-\delta_2}{2}) $ and $Y\sim \Gamma(\frac{\delta_1}{2})$ are taken independent and we thus obtain the result.     
\end{proof}

\begin{lemma}\label{l:conv_hitting0}
Let $\delta_1>0$, $\delta_2<2$, and $a,b$ be positive real numbers.  Suppose either
\begin{enumerate}[(i)]
    \item $(S_x,x\ge 0)$ is a ${\rm BESQ}(\delta_1\, |_a \, 0 \, |_{b}\, \delta_2)$ process starting at $0$,
    or
    \item $(S_x,x\ge -a)$ is a ${\rm BESQ}(\delta_1\, |_0 \, 0 \, |_{b}\, \delta_2)$ process starting at $0$.
\end{enumerate}
In both cases, let $T_{a,b}:=\inf\{y\ge b: S_x=0\}$. The conditional distribution of $T_{a,b}$ given $\{S_{b}>0\}$ converges as $a\to 0$ to the distribution of $\frac{b}{A}$, where $A$ has the beta distribution $\mathcal{B}(\frac{2-\delta_2}{2},1)$.
\end{lemma}
\begin{proof}
Applying Lemma \ref{distribution:hitting0}, we compute that 
\begin{align}\label{Sb_positive}
    \p(S_{b}>0)\sim \frac{a\delta_1}{2b} \quad \text{as}\quad a\to 0.
\end{align}
We first prove that, conditional on $S_{b}>0$, $S_{b}$ converges in distribution to an exponential random variable $Y$ with rate parameter $1/b$ as $a\rightarrow 0$. This result follows from the semi-group of BESQ$^\delta$ \cite[Corollary 1.4, Chapter XI]{revuz2013continuous} and the analyticity and boundedness of the Bessel function \cite[Section 5.7]{lebedev1972special}, which allows us to apply the dominated convergence theorem. By Markov property, 
\begin{align*}
    \p(T_{a,b}> x\mid S_{b}>0)=\e[\p_{S_{b}}(T>x-b)\mid S_{b}>0],
\end{align*}
where $T$ is the hitting time of $0$ by a ${\rm BESQ}^{\delta_2}_a$ process under $\p_a$. Since $T$ is distributed as $\frac{a}{2X}$ under $\p_a$ where $X\sim \Gamma(\frac{2-\delta_2}{2})$, $T_{a,b}$ given $\{S_{b}>0\}$ converges in distribution as $a\to 0$ to $\frac{Y+2bX}{2X}$, and we obtain the result.
\end{proof}

\section{Dimension of the graph of a BESQ flow line}

If $(X_t,\, t\ge 0)$ is a Brownian motion and 
$E\subseteq \r_+$ is a Borel set, then its graph 
under $X$ defined as
$$ Gr\ X(E) = \{(t,X_t):\ t\in E\} $$
has dimension
\begin{align}\label{BM:haus_dim}
    \dim Gr\ X(E) = \min \left( 2\dim E, \dim E + \frac{1}{2} \right)\; a.s.
\end{align}

The case $E=\r_+$ is the graph of the Brownian motion, which has Hausdorff dimension $\frac32$ \cite{taylor}. The general case can be deduced from \cite[Theorem 2.1 \& Theorem 2.3]{yimin1994dimension}. Since the law of a Bessel process is locally mutually absolutely continuous with respect to the law of the Brownian motion when it is away from $0$, equation \eqref{BM:haus_dim} still holds for the ${\rm BES}^\delta$ process when $\delta>0$, and therefore also for the ${\rm BESQ}^\delta$ process since the map $(t,x)\mapsto(t,x^2)$ is a diffeomorphism from $\r_+\times(\varepsilon,+\infty)$ to $\r_+\times(\varepsilon^2,+\infty)$ for all $\varepsilon>0$. When  $\delta\le 0$, $X$ is absorbed at $0$, so we need to replace $E$ with  $E\cap [0,T]$ in both sides of \eqref{BM:haus_dim}, where $T$ is the absorption time of $X$.

\begin{corollary}\label{dimension:BESQ_difference}
Let $(a,r) \in (0,\infty)\times \r$ and  $\cS^1$, $\cS^2$ be resp. a ${\rm BESQ}^{\delta_1}$ flow and a ${\rm BESQ}^{\delta_2}$ flow driven by $\cW$. Set $d:=\delta_2-\delta_1$. We suppose that $d\in (0,2)$. Define $T:=\inf\{x\ge r\,:\, \cS^1_{r,x}(a)=0\}$ and 
$$ \B := \{(b,x):\, \cS^1_{r,x}(a) = \cS^2_{r,x}(a) = b,\, x\le T \}.$$
Then $\B$ has Hausdorff dimension $\min(2-d,\frac{3-d}{2})$.
\end{corollary}

\begin{proof}
By property ({\bf P}1) or ({\bf P}3), conditionally on $T$, $x\mapsto \cS^2_{r,x}(a)-\cS^1_{r,x}(a)$ is a squared Bessel process with dimension $d$ before time $T$, starting at position $0$ and independent of $\cS^1_{r,\cdot}(a)$. The set $\B$ can now be viewed as the graph of the zero set of $\cS^2_{r,\cdot}(a)-\cS^1_{r,\cdot}(a)$ under $\cS^1_{r,\cdot}(a)$ before time $T$. This zero set has Hausdorff dimension $\frac{2-d}{2}>0$.  The result then follows from equation \eqref{BM:haus_dim}.
\end{proof}

\bibliographystyle{abbrvurl}
\bibliography{skewBM}

\begin{thebibliography}{10}

\bibitem{aidekon2024infinite}
E.~A{\"\i}d{\'e}kon, Y.~Hu, and Z.~Shi.
\newblock An infinite-dimensional representation of the {Ray-Knight} theorems.
\newblock {\em Science China Mathematics}, 67(1):149--162, 2024.

\bibitem{aïdékon2023stochastic}
E.~A{\"\i}d{\'e}kon, Y.~Hu, and Z.~Shi.
\newblock The stochastic {Jacobi} flow.
\newblock {\em To appear in The Annals of Probability}, 2024.

\bibitem{arratia}
R.~A. Arratia.
\newblock {\em Coalescing Brownian motions on the line}.
\newblock University of Wisconsin-Madison, 1979.

\bibitem{barlow1999coalescence}
M.~Barlow, K.~Burdzy, H.~Kaspi, and A.~Mandelbaum.
\newblock Coalescence of skew brownian motions.
\newblock {\em S{\'e}minaire de Probabilit{\'e}s XXXV}, pages 202--205, 2001.

\bibitem{bertoin-legall00}
J.~Bertoin and J.-F. Le~Gall.
\newblock The {Bolthausen--Sznitman} coalescent and the genealogy of continuous-state branching processes.
\newblock {\em Probability theory and related fields}, 117:249--266, 2000.

\bibitem{borga}
J.~Borga.
\newblock The skew {Brownian} permuton: A new universality class for random constrained permutations.
\newblock {\em Proceedings of the London Mathematical Society}, 126(6):1842--1883, 2023.

\bibitem{burdzy2001local}
K.~Burdzy and Z.-Q. Chen.
\newblock Local time flow related to skew {Brownian} motion.
\newblock {\em The Annals of Probability}, 29(4):1693--1715, 2001.

\bibitem{burdzy2004lenses}
K.~Burdzy and H.~Kaspi.
\newblock Lenses in skew {Brownian} flow.
\newblock {\em The Annals of Probability}, 32(4):3085--3115, 2004.

\bibitem{carmona_besq}
P.~Carmona, F.~Petit, and M.~Yor.
\newblock Some extensions of the arc sine law as partial consequences of the scaling property of brownian motion.
\newblock {\em Probability Theory and Related Fields}, 100(1):1--29, 1994.

\bibitem{dawson-li12}
D.~A. Dawson and Z.~Li.
\newblock {Stochastic equations, flows and measure-valued processes}.
\newblock {\em The Annals of Probability}, 40(2):813 -- 857, 2012.

\bibitem{gloter2013distance}
A.~Gloter and M.~Martinez.
\newblock Distance between two skew {Brownian} motions as a {S.D.E.} with jumps and law of the hitting time.
\newblock {\em The Annals of Probability}, 41(3A):1628--1655, 2013.

\bibitem{Gloter2018BouncingSB}
A.~Gloter and M.~Martinez.
\newblock Bouncing skew {Brownian} motions.
\newblock {\em Journal of Theoretical Probability}, 31:319--363, 2018.

\bibitem{going2003survey}
A.~G{\"o}ing-Jaeschke and M.~Yor.
\newblock A survey and some generalizations of {Bessel} processes.
\newblock {\em Bernoulli}, 9(2):313--349, 2003.

\bibitem{gwynne_holden_sun}
E.~Gwynne, N.~Holden, and X.~Sun.
\newblock Joint scaling limit of a bipolar-oriented triangulation and its dual in the peanosphere sense, 2016.
\newblock \href {https://arxiv.org/abs/1603.01194} {\path{arXiv:1603.01194}}.

\bibitem{harrison1981skew}
J.~M. Harrison and L.~A. Shepp.
\newblock On skew {Brownian} motion.
\newblock {\em The Annals of probability}, pages 309--313, 1981.

\bibitem{hawkes1974uniform}
J.~Hawkes and W.~E. Pruitt.
\newblock Uniform dimension results for processes with independent increments.
\newblock {\em Zeitschrift f{\"u}r Wahrscheinlichkeitstheorie und Verwandte Gebiete}, 28(4):277--288, 1974.

\bibitem{itomckean}
K.~It\^o and H.~P. McKean.
\newblock {\em Diffusion Processes and their Sample Paths}.
\newblock Springer Berlin, Heidelberg, 1974.

\bibitem{kunita1990stochastic}
H.~Kunita.
\newblock {\em Stochastic flows and stochastic differential equations}, volume~24.
\newblock Cambridge university press, 1990.

\bibitem{lambert}
A.~Lambert.
\newblock The genealogy of continuous-state branching processes with immigration.
\newblock {\em Probability theory and related fields}, 122(1):42--70, 2002.

\bibitem{lejan_raimond}
Y.~Le~Jan and O.~Raimond.
\newblock Three examples of brownian flows on $\mathbb {R}$.
\newblock {\em Annales de l'I.H.P. Probabilit\'es et statistiques}, 50(4):1323--1346, 2014.

\bibitem{lebedev1972special}
N.~Lebedev and R.~Silverman.
\newblock {\em Special Functions and Their Applications}.
\newblock Dover Books on Mathematics. Dover Publications, 1972.

\bibitem{lejay2006constructions}
A.~Lejay.
\newblock On the constructions of the skew {Brownian} motion.
\newblock {\em Probability Surveys}, 3:413--466, 2006.

\bibitem{pitman2018squared}
J.~Pitman and M.~Winkel.
\newblock {Squared {Bessel} processes of positive and negative dimension embedded in {Brownian} local times}.
\newblock {\em Electronic Communications in Probability}, 23:1--13, 2018.

\bibitem{pitmanyor}
J.~Pitman and M.~Yor.
\newblock A decomposition of {Bessel} bridges.
\newblock {\em Zeitschrift f{\"u}r Wahrscheinlichkeitstheorie und verwandte Gebiete}, 59(4):425--457, 1982.

\bibitem{revuz2013continuous}
D.~Revuz and M.~Yor.
\newblock {\em Continuous martingales and {Brownian} motion}, volume 293.
\newblock Springer Science \& Business Media, 2013.

\bibitem{taylor}
S.~J. Taylor.
\newblock The {{\(\alpha\)}}-dimensional measure of the graph and set of zeros of a {Brownian} path.
\newblock {\em Proc. Camb. Philos. Soc.}, 51:265--274, 1955.

\bibitem{toth_werner}
B.~T{\'o}th and W.~Werner.
\newblock The true self-repelling motion.
\newblock {\em Probability Theory and Related Fields}, 111(3):375--452, 1998.

\bibitem{walsh1978diffusion}
J.~B. Walsh.
\newblock A diffusion with a discontinuous local time.
\newblock {\em Temps locaux, Astérisque}, 52-53:37--45, 1978.

\bibitem{Walsh1986AnIT}
J.~B. Walsh.
\newblock {\em An introduction to stochastic partial differential equations}.
\newblock Springer, 1986.

\bibitem{yimin1994dimension}
Y.~Xiao and H.~Lin.
\newblock Dimension properties of sample paths of self-similar processes.
\newblock {\em Acta Mathematica Sinica}, 10(3):289--300, 1994.

\end{thebibliography}

\end{document}